\documentclass[preprint,12pt]{article}

\RequirePackage{microtype}
\RequirePackage{etoolbox}
\RequirePackage{booktabs}
\RequirePackage{refcount}
\RequirePackage{totpages}
\RequirePackage{environ}
\RequirePackage{bbm}

\usepackage[dvipsnames]{xcolor}
\usepackage{fullpage,float,amsmath,amsfonts,amssymb,amsthm,thmtools,mathtools,bm}
\usepackage{mathrsfs}
\usepackage{tikz-cd}
\usepackage{caption}
\usepackage{subfiles}

\usepackage{standalone}
\standaloneconfig{mode=buildnew}

\usetikzlibrary{positioning}
\usetikzlibrary{shapes.geometric}
\usetikzlibrary{arrows.meta}

\usepackage{hyperref}
\hypersetup{linktocpage}
\hypersetup{
    colorlinks,
    citecolor=red,
    filecolor=black,
    linkcolor=blue ,
    urlcolor=black,
    breaklinks=true
}

\usepackage{algorithm}
\usepackage{chngcntr}
\counterwithin{algorithm}{section}
\usepackage[noend]{algpseudocode}
\algnewcommand\algorithmicassumptions{\sc{Assumptions:}}
\algnewcommand\Assumptions{\item[\algorithmicassumptions]}
\algrenewcomment[1]{\hfill \(\triangleright\) \emph{\small #1}} 
\algnewcommand{\algorithmicand}{\textbf{and}}
\algnewcommand{\algorithmicor}{\textbf{or}}
\algnewcommand{\FOR}{\algorithmicfor}
\algnewcommand{\OR}{\algorithmicor}
\algnewcommand{\AND}{\algorithmicand}
\algnewcommand{\IF}{\algorithmicif}
\algnewcommand{\THEN}{\algorithmicthen}
\algnewcommand{\ELSE}{\algorithmicelse}
\algnewcommand{\Fail}{\textsc{Fail}}
\algnewcommand{\Cert}{\textsc{Cert}}
\algnewcommand{\NoCert}{\textsc{NoCert}}
\algnewcommand{\Flag}{\textsc{Flag}}
\algnewcommand{\CommentLine}[1]{\(\triangleright\) \emph{\small #1}}
\algnewcommand{\InlineFor}[2]{\algorithmicfor\ #1\ \algorithmicdo\ #2} 
\algnewcommand{\InlineIf}[2]{
  \algorithmicif\ #1\ \algorithmicthen\ #2}
\algnewcommand{\InlineIfElse}[3]{
  \algorithmicif\ #1\ \algorithmicthen\ #2\ \algorithmicelse\ #3}

\makeatletter
\let\original@algocf@latexcaption\algocf@latexcaption
\long\def\algocf@latexcaption#1[#2]{%
  \@ifundefined{NR@gettitle}{%
    \def\@currentlabelname{#2}%
  }{%
    \NR@gettitle{#2}%
  }%
  \original@algocf@latexcaption{#1}[{#2}]%
}
\makeatother

\makeatletter
\newcounter{algorithmicH}
\let\oldalgorithmic\algorithmic
\renewcommand{\algorithmic}{%
  \stepcounter{algorithmicH}
  \oldalgorithmic}
\renewcommand{\theHALG@line}{ALG@line.\thealgorithmicH.\arabic{ALG@line}}
\makeatother

\algrenewcommand\Call[2]{\nameref{#1}\ifthenelse{\equal{#2}{}}{}{\ensuremath{(#2)}}}%
\newcommand{\algoName}[1]{Algorithm \nameref{#1}}
\newcommand{\algoCaptionLabel}[2]{%
     \caption[\textproc{#1}]{\textproc{#1}\ifthenelse{\equal{#2}{}}{}{$(#2)$}}%
     \label{algo:#1}%
     }%

\usepackage[capitalize]{cleveref}

\def\A{\mathbb{A}}

\def\B{\mathbb{B}}
\def\K{\mathbb{K}}
\def\k{\mathbbmss{k}}

\def\L{\mathbb{L}}
\def\Q{\mathbb{Q}}
\def\Z{\mathbb{Z}}
\def\N{\mathbb{N}}

\def\coeff{\textsc{Coefficient}}
\def\pcoeff{\textsc{PolynomialCoefficient}}
\def\bdiv{\mathrm{~div~}}
\def\brem{\mathrm{~rem~}}
\def\softO{O\tilde{~}}
\def\mm{\mathfrak{m}}

\def\iniT{\mathsf{T}}
\def\iniS{\mathsf{S}}
\def\iniU{\mathsf{U}}
\def\iniR{\mathsf{R}}

\declaretheorem[style=plain,parent=section]{definition}
\declaretheorem[sibling=definition]{theorem}
\declaretheorem[sibling=definition]{corollary}
\declaretheorem[sibling=definition]{proposition}
\declaretheorem[sibling=definition]{lemma}
\declaretheorem[sibling=definition]{remark}
\declaretheorem[sibling=definition]{example}

\usepackage{amssymb}

\date{}

\begin{document}

\title{Newton iteration for lexicographic Gr\"obner bases in two variables}
\maketitle
\begin{center}
\'Eric Schost and Catherine St-Pierre
\end{center}

{School of Computer science, University of Waterloo}, 
{Waterloo}, {Ontario}, {Canada}

\begin{abstract}
  We present an $\mm$-adic Newton iteration with quadratic convergence
  for lexicographic Gr\"obner basis of zero dimensional ideals in two
  variables. We rely on a structural result about the syzygies in such
  a basis due to Conca and Valla, that allowed them to explicitly
  describe these Gr\"obner bases by affine parameters; our Newton iteration
  works directly with these parameters.
\end{abstract}

\textbf{keyword}
\textit{Primary components, $\mm$-adic algorithm, Gr\"obner bases}

\section{Introduction}

Solving bivariate polynomial equations plays an important role in
algorithms for computational topology or computer graphics. As a
result, there exists a large body of work dedicated to this question,
using symbolic, numeric or mixed symbolic-numeric
techniques~\cite{GoKa96,EmTs05,DiEmTs09,AlMoWi10,Rouillier10,BeEmSa11,EmSa12,BoLaPoRo13,LMS13,BoLaMoPoRo14,KoSa14,MeSc16,KoSa15,BoLaMoPoRoSa16,DiDiRoRoSa18,Dahan22}.

In many instances, these algorithms find a set-theoretic
description of the solutions of a given system $f_1,\dots,f_t$ in
$\K[x,y]$ (here, $\K$ is a field). This can notably be done
through the {\em shape lemma}: in generic coordinates, the output is a
pair of polynomials $u,v$ in $\K[x]$, with $u$ squarefree, such that
$V(\langle f_1,\dots,f_t\rangle)$ is described by $u(x)=0$ and $y = v(x)/u'(x)$ (this
rational form for $y$ allows for a sharp control of the bit-size of
$v$, if $\K=\Q$). One could slightly enrich this set-theoretic
description by lifting the requirement that $u$ be squarefree, and
instead assign to a root $\xi$ of $u$, corresponding to a point
$(\xi,\nu)$, the multiplicity of $J=\langle f_1,\dots,f_t \rangle$ at
$(\xi,\nu)$ (adapting the definition of $v$ accordingly). This is
what is done in Rouillier's Rational Univariate
Parametrization~\cite{Rouillier99}, but this still only gives partial
information: for instance it is not sufficient to detect local isomorphisms.

In order to describe the solutions of $J$, but also the local
structure of $J$ at these zeros, it is natural to turn to Gr\"obner
bases. This is what we do in this paper, our focus being an $\mm$-adic
approximation procedure.

\paragraph{Our problem and our main result}
Let us assume that our base field $\K$ is the field of fractions of a
domain $\A$, and take $f_1,\dots,f_t$ in $\A[x,y]$.

Consider further the ideal $J=\langle f_1,\dots, f_t \rangle$ in
$\K[x,y]$. We are interested in finding a Gr\"obner basis of $J$
itself, or possibly of some specific primary components of it. We will
thus let $I$ be an ideal in $\K[x,y]$, which we assume to be the
intersection of some of the zero-dimensional primary components of
$J$: typical cases of interest are $I=J$, if it has dimension zero, or
$I$ being the $\langle x,y \rangle$-primary component of $J$, if the
origin is isolated in~$V(J)$.

We let $\mathcal G=(g_0,\dots,g_s)$ be the minimal, reduced Gr\"obner basis
of $I$ for the lexicographic order induced by $y \succ x$; this is
the object we are interested in. 
\begin{example}\label{ex:runningex} 
Let
$\A=\Z$, and thus $\K=\Q$, $t=2$ and input polynomials
\begin{align*}
f_1 &= -12x y^5 - 20x^2 y^4 - 14y^4 - 7x^3 y^3 - 3x^2 y^2 + 13x^3 y - 17xy + 34x^2\\[1mm]
f_2 &= -x^2 y^4 - 19x^3 y^3 + 18x y^3 + 22x^3 y^2 + 2x^2  y^2- 10x^2y.
\end{align*}
We let $I$ be the $\langle x,y\rangle$-primary component of $\langle
f_1,f_2\rangle$; its Gr\"obner basis $\mathcal G$ is
\begin{equation}\label{eq:prim}
\left |
\begin{array}{l}
  y^4 + \frac{17}{14}xy - \frac{17}{7}x^2, \\[1mm]
   x y^3 - \frac{10}{9}x^3,\\[1mm]
  x^2 y - 2x^3,\\[1mm]
  x^4.
\end{array}
\right .
\end{equation}
\end{example}

Let now $\mm$ be a maximal ideal in $\A$, with residual field
$\k=\A/\mm$. Starting from the reduction of $\mathcal G$ modulo $\mm$
(assuming it is well-defined), the goal of this paper is to show how
to recover $\mathcal G$ modulo powers of~$\mm$. The case $\A=\Z$ seen above
is the fundamental kind of example; another important situation is the
``parametric'' case, with $\A=\k[t_1,\dots,t_m]$ and $\mm$ a
maximal ideal of the form $\langle
t_1-\tau_1,\dots,t_m-\tau_m\rangle$.

Let $\A_\mm$ ($\A_\mm \subseteq \K$) be the localization of $\A$ at $\mm$.  For $K \ge 0$, there
exists a well defined reduction operator $\A_\mm \to \A/\mm^K$, which
we write $c \mapsto c \brem \mm^K$; we extend it coefficient-wise to a
reduction mapping $\A_\mm[x,y] \to \A/\mm^K[x,y]$, and further to
vectors of polynomials.

\begin{definition}\label{def:good}
  We say that $\mm$ is {\bf good} with respect to $f_1,\dots, f_t$ and $\mathcal G$ if the following holds:
  \begin{itemize}
  \item all elements in $\mathcal G$ are in $\A_\mm[x,y]$,
  \item the ideal generated by $\mathcal G \brem \mm$ in $\k[x,y]$ is the
    intersection of some of the primary components of the ideal $\langle f_1
    \brem \mm,\dots,f_t \brem \mm\rangle$.
  \end{itemize}
\end{definition}
In particular, if $\mm$ is good, we will write $\mathcal G_\mm$ for the
reduction $\mathcal G \brem \mm$. These are polynomials in $\k[x,y]$, and
they still form a minimal, reduced Gr\"obner basis for the
lexicographic order $y \succ x$. 
\begin{example}
In \autoref{ex:runningex},	
$\mm = \langle 11
\rangle$ is good with respect to $f_1,\dots,f_t$ and $\mathcal G_\mm$ is
\begin{align*}
\left |
\begin{array}{l}
  y^4 + 2 xy + 7 x^2, \\[1mm]
   x y^3 +5 x^3,\\[1mm]
  x^2 y + 9x^3,\\[1mm]
  x^4.
\end{array}
\right .
\end{align*}
\end{example}

If $\A=\Z$, there are finitely many primes $p$ for which this is not
the case. In the case $\A=\k[t_1,\dots,t_m]$, all maximal ideals of
the form $\langle t_1-\tau_1,\dots,t_m-\tau_m\rangle$ are good, except
for those $(\tau_1,\dots,\tau_m)$ lying on a certain hypersurface in $\k^m$
(a quantitative analysis of the number of bad maximal ideals will be
the subject of future work).

Our main result is an efficient lifting procedure based on Newton
iteration to compute $\mathcal G \brem \mm^K$, given $f_1,\dots,f_t$, $\mathcal G_\mm$ and $K$. Lifting methods are widely used in computer algebra,
for instance to solve linear systems or compute polynomial GCDs, and
serve two purposes. First, while the arithmetic cost of solving the
problem at hand (here, computing the Gr\"obner basis of $I$) may be
high, our result will show that lifting an approximate solution modulo
powers of $\mm$ is a relatively simple problem. Second, these
techniques are usually used in cases where elements in $\A$, and $\K$,
have a natural notion of ``size'' (such as the height when $\A=\Z$, or
degree when $\A=\k[t_1,\dots,t_m]$). Then, direct computations in $\K$
often induce a significant ``intermediate expression swell'', where
polynomials computed throughout the algorithm may have larger
coefficients than the final output; $\mm$-adic approximation schemes
avoid this issue.

Our algorithm features the quadratic convergence typical of Newton
iteration, in the sense that it computes $\mathcal G \brem \mm^2,\mathcal G
\brem \mm^4,\dots$; hence, without loss of generality, we assume that
$K=2^\kappa$ is a power of two. The cost of the algorithm is expressed
in terms of two kinds of quantities:
\begin{itemize}
\item number of operations in the rings $\A/\mm^{2^i}$ (for which we
  discuss our computational model in more detail at the end of the
  introduction)
\item the cost of reducing the coefficients of the polynomials $f_i$
  modulo $\mm^{2^i}$: we will assume that for $i \ge 0$, each such
  coefficient can be reduced modulo $\mm^{2^i}$ in time $T_{2^i}$ (for
  $\A=\Z$, this time would depend on the bit-size of these
  coefficients; over $\A=\k[t_1,\dots,t_m]$, it would depend on their
  degree, and the number $m$ of parameters).
\end{itemize}
Throughout, the $\softO$ notation indicates that we omit
polylogarithmic factors, and $\omega$ is a feasible exponent for
linear algebra. 


\begin{theorem}\label{prop:quad}
  Let $f_1,\dots,f_t$ be of degree at most $d$ in $\A[x,y]$, with $\A$
  a domain, that generate an ideal $J$ in $\K[x,y]$, with $\K$ the
  fraction field of $\A$.  Let $I$ be the intersection of some of the
  zero-dimensional primary components of $J$, with minimal, reduced
  Gr\"obner basis $\mathcal G$, for the lexicographic order induced by $y
  \succ x$.
  
  Let further $\bm E= (y^{n_0}, x^{m_1} y^{n_1}, \dots, x^{m_{s-1}}
  y^{n_{s-1}} , x^{m_s})$ be the initial terms of $\mathcal G$, and
  let $\delta = \dim_\K \K[x,y]/I$.

  Let $\mm \subseteq \A$ be a good maximal ideal for $\mathcal G$. For $K$ of
  the form $K=2^k$, given $\mathcal G \brem \mm$, one can find $\mathcal G
  \brem \mm^K$ with the following cost:
  \begin{itemize}
  \item $\softO (s^2 n_0 m_s + t\delta(d^2 + d m_s +
    s\delta + \delta^{\omega-1}))$ operations in $\A/\mm^{2^i}$, for
    $i=1,\dots,k$;
  \item $t d^2 T_{2^i}$ steps for coefficient reduction, for $i=1,\dots,k$.
  \end{itemize}
\end{theorem}

\begin{remark}
  When $I$ is the $\langle x,y\rangle$-primary component of $J$, runtimes
  can be sharpened, giving
  \begin{itemize}
  \item $\softO (s^2 n_0 m_s + t\delta^2(m_s + \delta^{\omega-2}))$
    operations in $\A/\mm^{2^i}$, for $i=1,\dots,k$;
  \item $t \delta m_s T_{2^i}$ steps for coefficient reduction, for
    $i=1,\dots,k$.
  \end{itemize}
  Since $m_s \le \delta$, these are in particular $\softO (s^2 n_0 m_s
  + t\delta^3) \subset \softO((s+t)\delta^3)$, resp.\ $t \delta^2
  T_{2^i}$. For the latter, we also have the bound $t d^2 T_{2^i}$
  stated in the theorem, but here we prefer to express the cost in
  terms of the multiplicity $\delta$ only.
\end{remark}
This paper focuses on those cases where the ideal $I$ is not radical
(that is, where some points $p\in V(I)$ are singular), with the intent
of computing the local structure at such points.
If the sole interest is to find $V(I)$, then our approach is
unnecessarily complex: the algorithms in~\cite{LMS13,MeSc16} use
Newton iteration to compute a set-theoretic description of the
solutions in an efficient manner.

\begin{example}\label{ex:generic}
  An extreme case has $t=2$ and $f_1,f_2$ ``generic'' in the sense
  that they define a radical ideal in $\K[x,y]$ with $d^2$ solutions
  in general position. In this case, if we take $I=J$, we have $s=1$,
  $m_s=\delta=d^2$ and $n_0=1$.  Then, the complexity in the first
  item of the theorem becomes $\softO (d^5)$ operations modulo each
  $\mm^{2^i}$. This is to be compared with the sub-cubic cost
  $\softO(d^{(\omega+3)/2})$ reported in~\cite{LMS13} for a similar
  task.
\end{example}

Clearly, for these generic situations, our algorithm does not compare
favourably with the state of the art.  For the situation in
\autoref{ex:generic}, some techniques from~\cite{LMS13} could be put
to use in our situation as well, but they would at best give a runtime
of $\softO(d^{2+(\omega+3)/2})$ operations in $\A/\mm^{2^i}$, still
leaving a quadratic overhead. This is due to the different ways these
papers apply Newton iteration: in our case, we linearize the problem
in dimension $d^2$ (or, in, general, $\delta$), and thus work with
matrices of such size, whereas~\cite{LMS13} work with matrices of
size $2$ (albeit with polynomial entries).


The results of \autoref{prop:quad} are of interest in the presence of
intersection with multiplicities, where approaches such
as~\cite{LMS13} do not apply. The algorithm in~\cite{MeSc16} does not
solve our problem in such cases, as it does not compute a Gr\"obner
basis of $I$, but of its radical.

Remark that to derive a complete algorithm from our result, further
ingredients are needed: quantitative bounds on the number of bad
ideals $\mm$ (if $\A=\Z$ or $\A=\k[t_1,\dots,t_m]$, for instance), a
cost analysis for computing the starting point $\mathcal G_\mm$ and bounds
on a sufficient precision $K$ that will allow us to recover $\mathcal G$
from its approximation $\mathcal G \brem \mm^K$.  In order to avoid this
paper growing to an excessive length, we will address these questions
in a separate manuscript.\\





We now review previous work on bivariate systems and Newton iteration
for Gr\"obner bases. As we will see, there is a marked difference
between Newton iteration algorithms for ``simple'' solutions (where
the Jacobian of the input equations has full rank) in generic position
and those that can handle arbitrary situations. 

\paragraph{Newton iteration for non-degenerate solutions}
Following an early discussion in~\cite{Ebert83}, $p$-adic techniques
for Gr\"obner bases were introduced by Trinks in the
1980's~\cite{Trinks84}. That article focuses on zero-dimen\-sio\-nal
{\em radical} ideals with generators in $\Z[x_1,\dots,x_n]$, in shape
lemma position, that is, with a Gr\"obner basis of the form
$x_1-G_1(x_n),\dots,x_{n-1}-G_{n-1}(x_n),G_n(x_n)$, for the
lexicographic order $x_1 \succ \cdots \succ x_n$. Under this
assumption, given a ``lucky'' prime $p$, one can apply a symbolic form
of Newton iteration to lift $(G_1,\dots,G_n) \brem p$ to
$(G_1,\dots,G_n) \brem p^K$, for an arbitrary $K \ge 0$. Similar
techniques were used in the geometric resolution
algorithm~\cite{GiHeMoPa95,GiHeMoMoPa98,GiHaHeMoMoPa97,GiLeSa01}; the
scope of this symbolic form of Newton iteration was then extended
in~\cite{Schost03} to {\em triangular sets}, which are here understood
as those particular lexicographic Gr\"obner bases $(G_1,\dots,G_n)$
with respective initial terms of the form $x_1^{e_1},\dots,x_n^{e_n}$,
for some positive integers $e_1,\dots,e_n$. In~\cite{LMS13}, these
techniques were studied in detail for the case $n=2$ that concerns us
in this paper, with a focus on the complexity of the lifting process.

Computationally, these algorithms are rather straightforward: they
mainly perform matrix multiplications in size $n$ with entries that are
polynomials with coefficients in $\Z/p^K\Z$ (or more generally
$\A/\mm^K$). These methods also share their numerical counterpart's
quadratic convergence (in one iteration, the precision doubles, from
$p^K$ to $p^{2K}$), but none of them can directly handle solutions
with multiplicities.

\paragraph{Lifting algorithms for general inputs}
\cite{Winkler88} introduced an algorithm that handles arbitrary
inputs: given a Gr\"obner basis $\mathcal G$ for $f_1,\dots,f_t$
reduced modulo a ``lucky'' prime $p$, it recovers the Gr\"obner basis
of the same system modulo $p^K$, for any $K \ge 0$. No assumption is
made on the dimension of $V(\langle f_1,\dots,f_t\rangle)$ or the rank
of the Jacobian matrix of the equations. The computations are more
complex as the ones above, as they involve lifting not only the
Gr\"obner basis $\mathcal G$ itself, but also all quotients in the
division of $f_1,\dots,f_t$, and of the $S$-polynomials of $\mathcal
G$, by $\mathcal G$.

In follow-up work, \cite{Pauer92} discussed the choice of lucky
primes; for homogeneous inputs, or graded orderings,~\cite{Arnold03}
gave an efficient criterion to stop lifting and simplified the lifting
algorithm itself, using ideas of Pauer's (the $S$-polynomials are not
needed anymore).

To our knowledge, the algorithms mentioned here only perform {\em
  linear} lifting, going from an approximation modulo $p^K$ to
precision $p^{K+1}$; whether quadratic convergence is possible is
unclear to us. No cost analysis was made.

\paragraph{Deflation}
Ojika, Watanabe and Mitsui introduced the idea of deflation in a
numerical context~\cite{OjWaMi83}, to restore Newton iteration's
quadratic convergence even for multiple roots. The core idea is to
replace the system we are given by another set of equations, having
multiplicity one at the root we are interested in, possibly
introducing new variables. There are now many references discussing
this approach, see for
instance~\cite{Yamamoto84,Lecerf02,LeVeZh06,LeVeZh08,PoSz09,DaLiZe11,MaMo11,WuZh12}.

We are in particular going to use an idea from~\cite{HaMoSz17}. In
that reference, Hauenstein, Mourrain and Szanto designed a deflation
operator for an $n$-variate system $f_1,\dots,f_t$, that converges
quadratically to an augmented root $(\xi,\nu)$, where $\nu$ is a
vector that specifies the local structure at a point $\xi \in
V(\langle f_1,\dots,f_t\rangle)$, through the coefficients of multiplication matrices
in the local algebra at $\xi$. If $\xi$ is known, this gives in
particular an operator with quadratic convergence to compute the
structure constants. 

\paragraph{Our contribution} The lifting algorithm we propose is simpler
than in~\cite{Winkler88,Arnold03} (we do not need to consider the
polynomial quotients in the division of $f_1,\dots,f_t$ by $\mathcal
G$), but so far specific to lexicographic orders in two variables.

The first step is to identify a family of free parameters that
describe Gr\"obner bases with given initial terms (these Gr\"obner
bases form a {\em Gr\"obner cell}). The coefficients that appear in
the Gr\"obner basis do not form such a family, as there are nontrivial
relations between them. However, for lexicographic orders in two
variables, Conca and Valla explicitly constructed a one-to-one
parametrization of a given Gr\"obner cell by an affine
space~\cite{CoVa08}, from a description of canonical generators of
the syzygy module. Our Newton iteration computes the parameters
corresponding to $\mathcal G_\mm$ and lifts them modulo $\mm^K$.

This is done by adapting the approach of~\cite{HaMoSz17}: the
coefficients of the normal forms of $f_1,\dots,f_t$ modulo the unknown
Gr\"obner basis $\mathcal G$ are polynomials in the parameters of the
Gr\"obner cell; we prove that they admit as a (not necessarily unique)
solution the parameters corresponding to $\mathcal G$, and that their
Jacobian matrix has full rank at this solution. We can then apply
Newton iteration to these polynomials, using $\mathcal G_\mm$ to give us
their solution modulo $\mm$ as a starting point.

Computationally, the core operation is simply reduction modulo a
lexicographic Gr\"obner basis. While we have algorithms with
quasi-linear cost for reduction modulo a single polynomial, or modulo
two polynomials with respective initial terms $y^n$ and $x^m$, we are
not aware of specific results for arbitrary lexicographic
bases. Another contribution of this paper is a reduction algorithm,
where we use techniques developed by van der Hoeven and
Larrieu~\cite{HoLa18} for certain weighted orderings, adapted
to our purposes.

\paragraph{Leitfaden}
In \cref{sec:def}, we discuss {\em initial segments} in $\N^2$;
they allow us to describe polynomials reduced modulo a Gr\"obner
basis. We give in particular an algorithm for multiplying two such
polynomials.

In \cref{sec:puncGB}, we review known results on the structure of
bivariate lexicographic Gr\"obner bases: Lazard's
theorem~\cite{Lazard85}, and Conca and Valla's description of
Gr\"obner cells; \cref{sec:reduction} then presents our algorithm
for reduction modulo a lexicographic Gr\"obner basis.

In \cref{sec:conversion} and \cref{sec:inverse conversion}, we
give algorithms to compute the Gr\"obner basis corresponding to
a set of parameters in the Gr\"obner cell, and conversely.  Finally, we
describe Newton iteration for  the Gr\"obner cell parameters 
in \cref{sec:Newton},  proving \cref{prop:quad}.

\paragraph{Computational model}
In the whole paper, the costs of algorithms are measured using 
numbers of operations in the base ring or base field.

We will first and foremost count {\em $\Z$-algebra operations}. For an
algorithm with inputs and outputs in a (unital) ring $\A$, these are
additions and multiplications involving the inputs, previously
computed quantities, and constants taken from the image of the
canonical mapping $\Z\to \A$ (e.g., integers if $\A$ has
characteristic zero); they will be simply be called ``$(+,\times)$
operations''. If an algorithm performs only this kind of operations,
its outputs are in the subring of $\A$ generated by its inputs.

Important examples are addition, multiplication and Euclidean
division (by a monic divisor) in $\A[x]$; they can all be done using a
softly linear number of $(+,\times)$ operations in $\A$, over any base
ring $\A$. For background, see Chapters~8 and~9 in~\cite{GaGe13}.

Other operations we will occasionally use are invertibility tests and
inversions (to solve linear systems). Finally, if $\mm$ is an ideal in
a ring $\A$, given $a$ in $\A/\mm$, we assume that we can find $A$ in
$\A$ with $A \brem \mm =a$ using one operation in $\A$.

\paragraph{Notation} The following notation is used throughout
the paper. In the following items, $\A$ is an arbitrary ring.
\begin{itemize}
\item For $d \ge 1$, We let $\A[x]_{< d}$ be the free $\A$-module
  of all polynomials in $\A[x]$ of degree less than $d$.
\item For $f,g$ in $\A[x]$, with $f$ monic, we define $f \brem g$ and
  $f \bdiv g$ as respectively the remainder and quotient in the
  Euclidean division of $f$ by $g$.
\item For $f$ in $\A[x,y]$, $\deg(f,x)$ and $\deg(f,y)$ respectively
  denote its partial degrees with respect to $x$ and $y$.
\item For $f$ in $\A[x,y]$ and $i \ge 0$, the {\em polynomial
  coefficient} of $y^i$ in $f$ will refer to the coefficient $f_i$ in
  the expression $f = \sum_{i=0}^d f_i y^i$, with $f_0,\dots,f_d$ in
  $\A[x]$. In the pseudo-code, we write $\pcoeff(f,y^i) \in \A[x]$ for
  this polynomial coefficient.
\item If $f \in \A[x,y]$ has degree $d$ in $y$, we say that $f$ is
  {\em monic in $y$} if the polynomial coefficient of $y^d$ is $1$
  (this definition and the previous one carry over to coefficients
  with respect to $x$ instead, but we will not need this).
\item If $\sf T$ is a subset of $\N^2$, we write
  $\A[x,y]_{\sf T}$ for the $\A$-module of polynomials {\em
  supported on $\sf T$}, that is, all polynomials of the form $\sum_{(u,v)
  \in \sf T} a_{u,v} x^u y^v$, with only finitely many non-zero
  coefficients $a_{u,v}$. 
\end{itemize}

We will not need to define Gr\"obner bases over rings. In particular,
for reduction of bivariate polynomials, we only work over fields: if
${\mathcal G}$ is a Gr\"obner basis in $\K[x,y]$, where $\K$ is a field and
$\K[x,y]$ is endowed with a monomial order, $f \brem {\mathcal G}$ denotes
the remainder of $f$ through reduction by ${\mathcal G}$.


\section{Initial segments in $\N^2$}\label{sec:def}

In this section, we first introduce terminology and basic
constructions regarding subsets of $\N^2$ called initial segments. In
the second part, we give algorithms to multiply polynomials supported
on such initial segments.


\subsection{Basic definitions} \label{ssec:init}

\paragraph{Initial segments}
We say that a set $\iniT \subset\N^2$ is an {\em initial segment} if
for all $(m,n)$ in $\iniT$, any pair $(m',n')$ with $m'\le m$ and $n'
\le n$ is also in $\iniT$.

Suppose that $\iniT$ is an initial segment in $\N^2$, let $\K$ be a
field and $x,y$ be variables over $\K$.  The elements in $\K[x,y]$
supported on $\N^2-{\sf T}$ form a monomial ideal $I \subset
\K[x,y]$. Conversely, any initial segment $\iniT$ in $\N^2$ can be
obtained in this manner from a monomial ideal $I$, as the set of
exponents of monomials not in $I$.  If $\sf T$ is finite, we write the
minimal monomial generators of $I$ as
$$\bm E = (y^{n_0}, x^{m_1} y^{n_1}, \dots, x^{m_{s-1}}
y^{n_{s-1}},x^{m_s})$$ with the $m_i$'s increasing and the $n_i$'s
decreasing, and we set $m_0=n_s=0$. We call $n_0$ the {\em height} of
$\iniT$ and $m_s$ its {\em width}. We say that $\iniT$ is {\em
  determined} by $I$, or equivalently by $\bm E$.

For $i=1,\dots,s$, we set $d_i = m_i - m_{i-1}$, so that $m_i = d_1 +
\cdots + d_i$. Then, the cardinal $\delta$ of $\iniT$ can be written
as $\sum_{i=1}^s d_i n_{i-1}$; $\delta$ is also called the {\em
  degree} of $\bm E$. Similarly, for $i=1,\dots,s$, we write $e_i =
n_{i-1}-n_{i}$. These definitions are illustrated in
Figure~\ref{fig:example}, where the monomials in $\bm E$ are the
initial terms of the Gr\"obner basis in~\cref{eq:prim}.

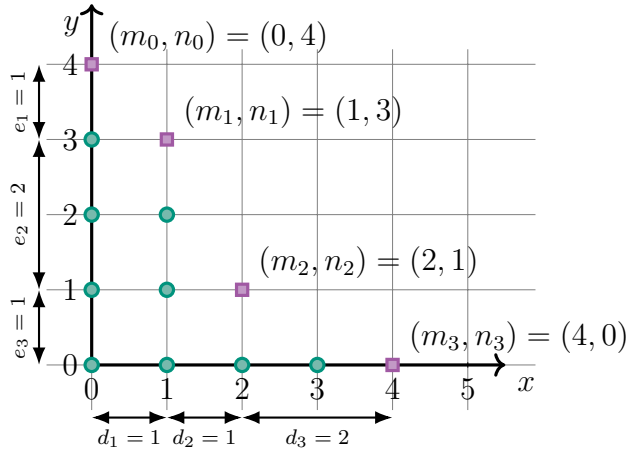
\begin{figure}[H]
\centering
  \begin{tikzpicture}
    [  
      roundnode/.style={circle, draw=PineGreen!90, fill=PineGreen!50, very thick, minimum size=5pt, inner sep=0pt, outer sep=0pt},
      square/.style={regular polygon,regular polygon sides=4, draw=Purple!90, fill=Purple!50, very thick, minimum size=6pt, inner sep=0pt, outer sep=0pt},
    ]
    \draw[step=1cm,gray,very thin] (-0.6,-0.6) grid (5.9,4.2);
    \draw[very thick,->] (-0.1,0) -- (5.5,0) node[anchor=north west] {$x$};
    \draw[very thick,->] (0,-0.1) -- (0,4.8) node[anchor=north east] {$y$};
    \foreach \x in {0,1,2,3,4,5} \draw (\x cm,1pt) -- (\x cm,-1pt) node[anchor=north] {$\x$};
    \foreach \y in {0,1,2,3,4} \draw (1pt,\y cm) -- (-1pt,\y cm) node[anchor=east] {$\y$};
    \node[roundnode] at (0,0) {};
    \node[roundnode] at (0,1) {};
    \node[roundnode] at (0,2) {};
    \node[roundnode] at (0,3) {};
    \node[square,label=10:{$(m_0,n_0)=(0,4)$},] at (0,4) {};
    \node[roundnode] at (1,0) {};
    \node[roundnode] at (1,1) {};
    \node[roundnode] at (1,2) {};
    \node[square,label=10:{$(m_1,n_1)=(1,3)$},] at (1,3) {};
    \node[roundnode] at (2,0) {};
    \node[square,label=10:{$(m_2,n_2)=(2,1)$},] at (2,1) {};
    \node[roundnode] at (3,0) {};
    \node[square,label=10:{$(m_3,n_3)=(4,0)$},] at (4,0) {};
    \draw[thick,{Latex}-{Latex},font=\scriptsize] (0,-0.7) -- node [below,midway] {$d_1=1$} (1,-0.7);
    \draw[thick,{Latex}-{Latex},font=\scriptsize] (1,-0.7) -- node [below,midway] {$d_2=1$} (2,-0.7);
    \draw[thick,{Latex}-{Latex},font=\scriptsize] (2,-0.7) -- node [below,midway] {$d_3=2$} (4,-0.7);
    \draw[thick,{Latex}-{Latex},font=\scriptsize] (-0.7,4) -- node [above,midway,rotate=90] {$e_1=1$} (-0.7,3);
    \draw[thick,{Latex}-{Latex},font=\scriptsize] (-0.7,3) -- node [above,midway,rotate=90] {$e_2=2$} (-0.7,1);
    \draw[thick,{Latex}-{Latex},font=\scriptsize] (-0.7,1) -- node [above,midway,rotate=90] {$e_3=1$} (-0.7,0);
  \end{tikzpicture}
  \caption{An initial segment $\iniT$ (green) and the monomials  $\bm E = (y^4,
xy^3, x^2y, x^4)$ (purple), with $s=3$ and $\delta=9$.}
  \label{fig:example}
\end{figure}

The cost analyses in this paper will be done using in particular the
parameters $s$ and $\delta$. If desired, one can simplify such
expressions using the following explicit upper bound for $s$.
\begin{lemma}\label{bound_s}
  The integer $s$ is in $O(\sqrt{\delta})$, and this bound is sharp in
  some instances.
\end{lemma}
\begin{proof}
  Start from the equality $\delta=\sum_{i=1}^s d_i n_{i-1}$, which
  implies $\delta \ge \sum_{i=1}^s n_{i-1}$. Since $n_s = 0$ and
  $n_{i-1} > n_i$, we get by induction $n_i \ge s-i$ for all $i$.
  This implies $\delta \ge s(s-1)/2$, so that $s$ is in
  $O(\sqrt{\delta})$. For the lower bound, for any integer $d$ we can
  take $\bm E = (x^i y^{d-i},\ i=0,\dots,d)$, for which $s=d$ and
  $\delta = d(d+1)/2$.
\end{proof}

\paragraph{Translates of an initial segment}
We will occasionally make use of the following construction. Let
$\iniT$ be a finite initial segment in $\N^2$, and suppose that
$\iniT$ is determined by a monomial ideal $I$, with minimal monomial
generators $\bm E$ as above. For $i=0,\dots,s$ we let
$\iniT_{\leftarrow i}$ be the initial segment determined by the colon
ideal $I : x^{m_i}$, with minimal monomial generators
$$ \bm E_{\leftarrow i}= (y^{n_i}, x^{m_{i+1}-m_i} y^{n_{i+1}}, \dots,
x^{m_{s-1}-m_i} y^{n_{s-1}},x^{m_s-m_i}).$$ The set $\iniT_{\leftarrow i}$
has height $n_i$ and width $m_s-m_i$; its cardinal will be written
$\delta_i$, and is equal to $\sum_{j=i+1}^{s} d_{j}
n_{j-1}$. We call $\iniT_{\leftarrow i}$ the $i$th {\em translate} of $\iniT$.

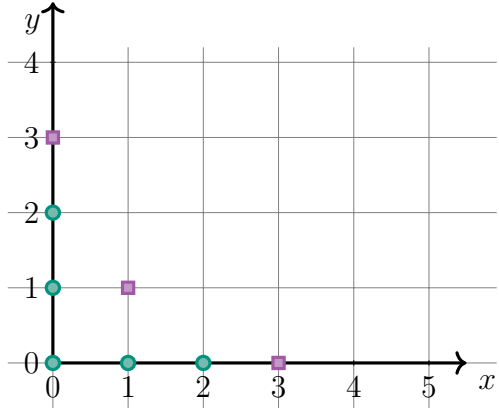
\begin{figure}[H]
\centering
  \begin{tikzpicture}
    [  
      roundnode/.style={circle, draw=PineGreen!90, fill=PineGreen!50, very thick, minimum size=5pt, inner sep=0pt, outer sep=0pt},
      square/.style={regular polygon,regular polygon sides=4, draw=Purple!90, fill=Purple!50, very thick, minimum size=6pt, inner sep=0pt, outer sep=0pt},
    ]
    \draw[step=1cm,gray,very thin] (-0.6,-0.6) grid (5.9,4.2);
    \draw[very thick,->] (-0.1,0) -- (5.5,0) node[anchor=north west] {$x$};
    \draw[very thick,->] (0,-0.1) -- (0,4.8) node[anchor=north east] {$y$};
    \foreach \x in {0,1,2,3,4,5} \draw (\x cm,1pt) -- (\x cm,-1pt) node[anchor=north] {$\x$};
    \foreach \y in {0,1,2,3,4} \draw (1pt,\y cm) -- (-1pt,\y cm) node[anchor=east] {$\y$};
    \node[roundnode] at (0,0) {};
    \node[roundnode] at (0,1) {};
    \node[roundnode] at (0,2) {};
    \node[square] at (0,3) {};
    \node[roundnode] at (1,0) {};
    \node[square] at (1,1) {};
    \node[roundnode] at (2,0) {};
    \node[square] at (3,0) {};
  \end{tikzpicture}
  \caption{The first translate $\iniT_{\leftarrow 1}$ of $\iniT$ from Figure~\ref{fig:example}.}
  \label{fig:translate}
\end{figure}

\paragraph{The shell of an initial segment} 
Let $\iniT$ be finite initial segment in $\N^2$. In this paragraph, we
define its {\em shell} $\iniT'$, which is another initial segment that
forms an outer approximation of $\iniT$ with few generators. The
definition and the lemma below are from~\cite[{\bf A.2}]{HyMeScSt19}.

As we did before, we let $$\bm E = (y^{n_0}, x^{m_1} y^{n_1}, \dots,
x^{m_{s-1}} y^{n_{s-1}},x^{m_s})$$ be the minimal monomial generating
set associated to $\iniT$. We define $\iniT'$ by introducing indices $i_{\sigma} <
i_{{\sigma}-1} < \cdots < i_0$, defined as follows. Set $i_0=s$. We let
$i_1 \ge 0$ be the largest index less than $i_0$ and such that
$m_{i_1} < m_{i_0}/2$, and iterate the process to define a sequence
$i_{\sigma} = 0 < i_{{\sigma}-1} < \cdots < i_0=s$.  We can then consider the
monomials
$$\bm E' =
(y^{n_{i_{\sigma}}},x^{m_{i_{{\sigma}-1}}}y^{n_{i_{{\sigma}-1}}},\dots,x^{m_{i_0}})
= (y^{n_0},x^{m_{i_{{\sigma}-1}}}y^{n_{i_{{\sigma}-1}}},\dots,x^{m_s}),$$ and
let $\iniT'$ be the initial segment determined by $\bm E'$.

\begin{lemma}\label{lemma:shell}
  The initial segment $\iniT'$ contains $\iniT$, its cardinal is
  at most $2 \delta$ and $\sigma$ is in $O(\log(\delta))$.
\end{lemma}


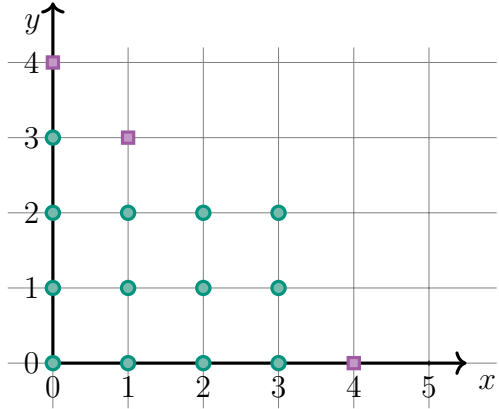
\begin{figure}[H]
\centering
  \begin{tikzpicture}
    [  
      roundnode/.style={circle, draw=PineGreen!90, fill=PineGreen!50, very thick, minimum size=5pt, inner sep=0pt, outer sep=0pt},
      square/.style={regular polygon,regular polygon sides=4, draw=Purple!90, fill=Purple!50, very thick, minimum size=6pt, inner sep=0pt, outer sep=0pt},
    ]
    \draw[step=1cm,gray,very thin] (-0.6,-0.6) grid (5.9,4.2);
    \draw[very thick,->] (-0.1,0) -- (5.5,0) node[anchor=north west] {$x$};
    \draw[very thick,->] (0,-0.1) -- (0,4.8) node[anchor=north east] {$y$};
    \foreach \x in {0,1,2,3,4,5} \draw (\x cm,1pt) -- (\x cm,-1pt) node[anchor=north] {$\x$};
    \foreach \y in {0,1,2,3,4} \draw (1pt,\y cm) -- (-1pt,\y cm) node[anchor=east] {$\y$};
    \node[roundnode] at (0,0) {};
    \node[roundnode] at (0,1) {};
    \node[roundnode] at (0,2) {};
    \node[roundnode] at (0,3) {};
    \node[square] at (0,4) {};
    \node[roundnode] at (1,0) {};
    \node[roundnode] at (1,1) {};
    \node[roundnode] at (1,2) {};
    \node[square] at (1,3) {};
    \node[roundnode] at (2,0) {};
    \node[roundnode] at (2,1) {};
    \node[roundnode] at (2,2) {};
    \node[roundnode] at (3,0) {};
    \node[roundnode] at (3,1) {};
    \node[roundnode] at (3,2) {};
    \node[square] at (4,0) {};
  \end{tikzpicture}
  \caption{The shell of $\iniT$ from Figure~\ref{fig:example}.}
  \label{fig:example_shell}
\end{figure}

In our pseudo-code, we will write $\iniT' \gets \textsc{Shell}(\iniT)$
to indicate that $\iniT'$ is the shell of $\iniT$. The algorithm
$\textsc{Shell}$ does not use any base field or base ring operation,
only index manipulations (in particular, it does not show up
in our cost analyses).



\subsection{Structured polynomial multiplication}


We now prove two propositions regarding polynomial multiplication in
$\A[x,y]$, for an arbitrary ring $\A$, which will be the basis of the
runtime analysis of several algorithms.  We mention in all
propositions below that the algorithms in this section only use
additions and multiplications in $\A$, as we will need this property
in the sequel. In what follows, given two sets $\iniS,\iniT$ in
$\N^2$, $\iniS+\iniT$ denotes their Minkowski sum.

The main prerequisite is the following fact: if $\iniS \subset \N^2$
is a rectangle, given $A$ and $B$ in $\A[x,y]_\iniS$, we can compute
$AB \in \A[x,y]_{\iniS+\iniS}$ using $\softO(|\iniS|)$ operations
$(+,\times)$ in $\A$: if $\iniS$ contains the origin, this is done
using Kronecker substitution to reduce to multiplication in~$\A[x]$,
see~\cite[Corollary~8.28]{GaGe13}; in the general case, we reduce to
the situation where $\iniS$ contains the origin by factoring out $x^u
y^v$ from $A$ and $B$, with $(u,v)$ being the unique minimal element
of~$\iniS$.

This being said, the first result we highlight here gives the cost of
computing the product $AB$, for $A$ and $B$ supported on the same
initial segment $\iniT$. Note that $AB$ is supported on $\iniT+\iniT$, and that if
$\iniT$ has height $n$ and width $m$, $\iniT+\iniT$ has cardinal
$\Theta(nm)$. Indeed, this set contains the rectangle $\{0,\dots,m-1\}
\times \{0,\dots,n-1\}$ of cardinal $nm$, and is contained in the
rectangle $\{0,\dots,2m-2\} \times \{0,\dots,2n-2\}$ of cardinal
less than $4n m$, so that $|\iniT+\iniT| \in \Theta(n m)$. This is to be
contrasted with the cardinal of $\iniT$ itself, which can range
anywhere between $n+m$ and $nm$.

\begin{proposition}\label{prop:mulTT}
  Consider a finite initial segment $\iniT \subset \N^2$, of height $n$ and
  width $m$. Given $A$ and $B$ in $\A[x,y]_\iniT$, one can compute $AB$
  using $\softO(|\iniT+\iniT|)=\softO(n m)$ operations  $(+,\times)$ in $\A$.
\end{proposition}
\begin{proof}
  Let $\iniS$ be the rectangle $\{0,\dots,m-1\} \times \{0,\dots,n-1\}$,
  so that $\iniS$ contains $\iniT$. Then, $A$ and $B$ are in $\A[x,y]_\iniS$, so
  we can multiply them using $\softO(|\iniS+\iniS|)=\softO(n m)$ operations $(+,\times)$ in
  $\A$ with Kronecker substitution, as pointed out above, and this
  runtime is also $\softO(|\iniT+\iniT|)$.
\end{proof}

Our second proposition gives an algorithm to compute $AB \in \A[x,y]$,
where $A$ is supported on a rectangle containing the origin and $B$ on
an initial segment.

\begin{proposition}\label{prop:mulST}
  Consider a rectangle $\iniS \subset \N^2$ and a finite initial segment
  $\iniT\subset \N^2$. Given $A$ in $\A[x,y]_\iniS$ and $B$ in $\A[x,y]_\iniT$,
  one can compute $AB$ using $\softO(|\iniS+\iniT|)$ operations  $(+,\times)$ in $\A$.
\end{proposition}

Without loss of generality, we assume that $\iniS$ contains the origin
$(0,0)$; if not, as above, factor out the monomial $x^uy^v$ from $A$,
with $(u,v)$ the minimal element in $\iniS$. We can thus suppose that
$\iniS$ is the rectangle $\{0,\dots,\ell-1\}\times \{0,\dots,h-1\}$,
for some integers $\ell,h \ge 1$, so in particular $|\iniS|=\ell h$,
and that $\iniT$ is an initial segment of cardinal $|\iniT|=\delta$,
with height $n$ and width $m$.

If $\A$ is a field of characteristic zero, this result follows
directly from the sparse evaluation and interpolation algorithms
of~\cite{CaKaYa89}. More generally, if $\A$ is a field of cardinal at
least $\max(\ell+m,h+n)-1$, this is also the case, using the algorithm
in~\cite{HoSc13}. The algorithm below achieves the same asymptotic
runtime, without assumption on $\A$. The proof is slightly more
involved than that of the previous proposition, and occupies the rest
of this section.

\paragraph{An algorithm when $\iniT$ is a rectangle} 
Suppose first that $\iniT=\{0,\dots,m-1\}\times \{0,\dots,n-1\}$, so that
$\delta = n m$; then the cardinal of $\iniS+\iniT$ is
$(\ell+m-1)(h+n-1)$. 

Take $A$ in $\A[x,y]_\iniS$ and $B$ in $\A[x,y]_\iniT$. Then, both $A$
and $B$ are in $\A[x,y]_{\iniS+\iniT}$. Since $\iniS+\iniT$ is a
rectangle, we saw in the preamble of this section that using
Kronecker's substitution, we can compute their product using
$\softO(|\iniS+\iniT|)= \softO((\ell+m-1)(h+n-1))$ operations
$(+,\times)$ in $\A$. In the main algorithm below, this is written
$\textproc{KroneckerMultiply}(A,B)$.

\paragraph{A first general algorithm} We now suppose that $\iniT$ is an arbitrary
initial segment, and that it is determined by the monomials
$$\bm E = (y^{n_0}, x^{m_1} y^{n_1}, \dots, x^{m_{s-1}}
y^{n_{s-1}},x^{m_s}),$$ with the $m_i$'s increasing, the $n_i$'s
decreasing, and $m_0=n_s=0$; note that we also have $n_0=n$ and
$m_s=m$. As before, for $i=1,\dots,s$, we set $d_i = m_i - m_{i-1}$,
so that $m_i = d_1 + \cdots + d_i$.

The input $B \in \A[x,y]_\iniT$ can then be written as $B=\sum_{0 \le i <
  s} B_i x^{m_i}$, with $B_i$ supported on $\iniT_i =
\{0,\dots,d_{i+1}-1\} \times \{0,\dots,n_i-1\}$. To compute $AB$, with
$A$ in $\A[x,y]_\iniS$, we thus compute all $A B_i$ and add up the
results.

\begin{algorithm}
 \algoCaptionLabel{MultiplyNaive}{A,\iniS,B,\iniT}
  \begin{algorithmic}[1]
    \Require $A$ in $\A[x,y]_\iniS$, $B$ in $\A[x,y]_\iniT$
    \Ensure $AB$ in $\A[x,y]_{\iniS+\iniT}$
    \State write $B=B_0 + B_1 x^{m_1} + \cdots + B_{s-1} x^{m_{s-1}}$ with $B_i \in \A[x,y]_{\{0,\dots,d_{i+1}-1\}
\times \{0,\dots,n_i-1\}}$ for all~$i$
    \State \InlineFor{$i=0,\dots,s-1$}{$C_i \gets \textproc{KroneckerMultiply}(A,B_i)$}
    \State \Return $C_0 + C_1 x^{m_1} + \cdots + C_{s-1} x^{m_{s-1}}$
  \end{algorithmic}
\end{algorithm}
By the result in the previous paragraph, each product $A B_i$ can be computed
in $$ \softO((\ell+d_{i+1}-1)(h+n_i-1))=\softO((\ell-1)(h-1)
+(\ell-1)n_i + d_{i+1}(h-1) + d_{i+1} n_i)$$ operations in $\A$, and
the cost of adding this product to the final result fits into the same
bound.  Using the inequality $n_i \le n_0=n$ for all $i$, as well as
$d_1+\cdots+d_s = m_s = m$ and $d_1 n_0 + \cdots + d_s n_{s-1}=\delta$
(the cardinal of $\iniT$), we see that the total cost is $$\softO(
s(\ell-1)(h-1)+ s (\ell-1) n + m (h-1) + \delta).$$

On the other hand, we can determine the cardinal of the sum $\iniU=\iniS+\iniT$ as
follows. The set $\iniU$ is the disjoint union of the following sets:
\begin{itemize}
\item $\iniU_1 = \{0,\dots,\ell-2\}\times \{0,\dots,h-2\}$, 
\item $\iniU_2 = (0,h-1) + \{0,\dots,\ell-2\}\times \{0,\dots,n-1\}$
\item $\iniU_3 = (\ell-1,0) + \{0,\dots,m-1\}\times \{0,\dots,h-2\}$
\item $\iniU_4 = (\ell-1,h-1) + \iniT$.
\end{itemize}
This is established by taking $(i,j)$ in $\iniS$, $(v,w)$ in $\iniT$, and
discussing according to the signs of $v-(\ell-1-i)$ and $w-(h-1-j)$.
As a result, we obtain
$$|\iniS+\iniT|=(\ell-1) (h-1)+ (\ell-1)
n + m (h-1) + \delta.$$
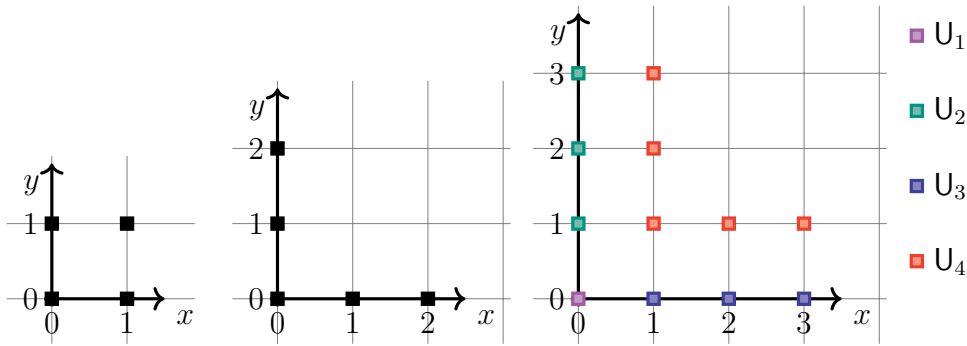
\begin{figure}[H]
\centering
\begin{tikzpicture}
  [  
    blacknode/.style={regular polygon,regular polygon sides=4, draw=black, fill=black, very thick, minimum size=6pt, inner sep=0pt, outer sep=0pt},
    square/.style={regular polygon,regular polygon sides=4, draw=Purple!90, fill=Purple!50, very thick, minimum size=6pt, inner sep=0pt, outer sep=0pt},
    lightsquare/.style={regular polygon,regular polygon sides=4, draw=Blue!90, fill=Blue!50, very thick, minimum size=6pt, inner sep=0pt, outer sep=0pt},
    roundnode/.style={regular polygon,regular polygon sides=4, draw=PineGreen!90, fill=PineGreen!50, very thick, minimum size=6pt, inner sep=0pt, outer sep=0pt},
    rednode/.style={regular polygon,regular polygon sides=4, draw=Red!90, fill=Red!50, very thick, minimum size=6pt, inner sep=0pt, outer sep=0pt},
  ]
  \draw[step=1cm,gray,very thin] (-0.6,-0.6) grid (1.9,1.9);
  \draw[very thick,->] (-0.1,0) -- (1.5,0) node[anchor=north west] {$x$};
  \draw[very thick,->] (0,-0.1) -- (0,1.8) node[anchor=north east] {$y$};
  \foreach \x in {0,1} \draw (\x cm,1pt) -- (\x cm,-1pt) node[anchor=north] {$\x$};
  \foreach \y in {0,1} \draw (1pt,\y cm) -- (-1pt,\y cm) node[anchor=east] {$\y$};
  \node[blacknode] at (0,0) {};
  \node[blacknode] at (0,1) {};
  \node[blacknode] at (1,0) {};
  \node[blacknode] at (1,1) {};

  \draw[step=1cm,gray,very thin] (2.4,-0.6) grid (6.1,2.9);
  \draw[very thick,->] (2.9,0) -- (5.5,0) node[anchor=north west] {$x$};
  \draw[very thick,->] (3,-0.1) -- (3,2.8) node[anchor=north east] {$y$};
  \foreach \x in {0,1,2} \draw ({{(3+\x)}},1pt) -- ({{(3+\x)}},-1pt) node[anchor=north] {$\x$};
  \foreach \y in {0,1,2} \draw (3cm + 1pt,\y cm) -- (3cm - 1pt,\y cm) node[anchor=east] {$\y$};
  \node[blacknode] at (3,0) {};
  \node[blacknode] at (3,1) {};
  \node[blacknode] at (3,2) {};
  \node[blacknode] at (4,0) {};
  \node[blacknode] at (5,0) {};

  \draw[step=1cm,gray,very thin] (6.4,-0.6) grid (11.1,3.9);
  \draw[very thick,->] (6.9,0) -- (10.5,0) node[anchor=north west] {$x$};
  \draw[very thick,->] (7,-0.1) -- (7,3.8) node[anchor=north east] {$y$};
  \foreach \x in {0,1,2,3} \draw ({{(7+\x)}},1pt) -- ({{(7+\x)}},-1pt) node[anchor=north] {$\x$};
  \foreach \y in {0,1,2,3} \draw (7cm + 1pt,\y cm) -- (7cm - 1pt,\y cm) node[anchor=east] {$\y$};
  \node[square] at (7,0) {};

  \node[roundnode] at (7,1) {};
  \node[roundnode] at (7,2) {};
  \node[roundnode] at (7,3) {};

  \node[lightsquare] at (8,0) {};
  \node[lightsquare] at (9,0) {};
  \node[lightsquare] at (10,0) {};

  \node[rednode] at (8,1) {};
  \node[rednode] at (8,2) {};
  \node[rednode] at (8,3) {};
  \node[rednode] at (9,1) {};
  \node[rednode] at (10,1) {};

  \node[rednode, label=east:$\iniU_4$] at (11.5,0.5) {};
  \node[lightsquare, label=east:$\iniU_3$] at (11.5,1.5) {};
  \node[roundnode, label=east:$\iniU_2$] at (11.5,2.5) {};
  \node[square, label=east:$\iniU_1$] at (11.5,3.5) {};
  
\end{tikzpicture}
\label{fig:Minkowski}
\caption{the sets $\iniS$, $\iniT$ and $\iniU=\iniS+\iniT$, with $\ell=h=2$ and $n=m=3$.
}
\end{figure}

\paragraph{The main algorithm.}
The runtime reported above does not fit in the target cost
$\softO(|\iniS+\iniT|)$, as $s$ could be large. To circumvent this
issue, we apply the algorithm of the previous paragraph, but we
replace $\iniT$ by its shell $\iniT'$. We know (\cref{lemma:shell})
that the cardinal of $\iniT'$ is at most $2\delta$, that its width and
height are the same as those of $\iniT$, and that it is generated by
$\sigma \in O(\log(s)) \subset O(\log(\delta))$ terms.

\begin{algorithm}[H]
  \algoCaptionLabel{Multiply}{A,\iniS,B,\iniT}
  \begin{algorithmic}[1]
    \Require $A$ in $\A[x,y]_\iniS$, $B$ in $\A[x,y]_\iniT$
    \Ensure $AB$ in $\A[x,y]_{\iniS+\iniT}$
    \State $\iniT' \gets \textproc{Shell}(\iniT)$
    \State \Return $\Call{algo:MultiplyNaive}{A,\iniS,B,\iniT'}$
  \end{algorithmic}
\end{algorithm}

The algorithm of the previous paragraph still applies (since $\iniT$
is contained in $\iniT'$), and its runtime is then $\softO(
(\ell-1)(h-1)\log(\delta)+ (\ell-1) n\log(\delta) + m (h-1) + \delta)$
operations $(+,\times)$ in $\A$. Since we saw that
$|\iniS+\iniT|=(\ell-1) (h-1)+ (\ell-1) n + m (h-1) + \delta$, the
above expression is indeed in $\softO(|\iniS+\iniT|)$. This finishes
the proof of \cref{prop:mulST}.


\section{Lexicographic Gr\"obner bases}\label{sec:puncGB}

In this section, we first review Lazard's structure
theorem~\cite{Lazard85} for lexicographic Gr\"obner bases in
$\K[x,y]$, for a field $\K$, then a parametrization of such bases due
to \cite{CoVa08}. While the core of the discussion makes no assumption
on the ideals we consider, we also highlight the case of ideals that
are primary at the origin, that is, $\langle x,y\rangle$-primary.

In all that follows, we use the lexicographic monomial order $\succ$
on $\K[x,y]$ induced by $y\succ x$.


\subsection{The structure theorem}\label{ssec:structTh}

Consider a zero dimensional ideal $I \subseteq \K[x,y]$, and let ${
  \mathcal G}=(g_0, \dots, g_s)$ be its reduced minimal Gr\"obner basis, listed
in decreasing order. Let further
$$\bm E = (y^{n_0}, x^{m_1} y^{n_1}, \dots, x^{m_{s-1}} y^{n_{s-1}} ,
x^{m_s})$$ be the minimal reduced basis of the initial ideal $in(I)$ of $I$,
listed in decreasing order, so the $n_i$'s are decreasing and the
$m_i$'s are increasing; as before, we set $m_0=n_s=0$.

It follows that $g_i$ has initial term $x^{m_{i}} y^{n_{i}}$ for all
$i$; in particular $g_0$ is monic in $y$ with initial term $y^{n_0}$.

As in Section~\ref{ssec:init}, for $i=1,\dots,s$, we set $d_i = m_i -
m_{i-1}$, with thus $m_i = d_1 + \cdots + d_i$, and $e_i =
n_{i-1}-n_i$.

Lazard proved in~\cite[Theorem~1]{Lazard85} the existence of
polynomials $D_1,\dots,D_s$ in $\K[x]$, all monic in $x$ and of
respective degrees $d_1,\dots,d_s$, such that for $i=0,\dots,s$, $g_i$
can be written as $M_i G_i$, with $M_i = D_{1}\cdots D_i \in \K[x]$
and $G_i\in \K[x,y]$ monic of degree $n_i$ in $y$ (for $i=0$, we set $D_0=1$). In
particular, for $i=s$, this gives $g_s = M_s = D_1 \cdots D_s$ and
$G_s=1$ . In addition, for $i=0,\dots,s-1$, we have the membership
relation
\begin{align}\label{eq:membership}
  G_i \in 
\left \langle G_{i+1},\  D_{i+2} G_{i+2},\ \dots,\ D_{i+2} \cdots  D_s
\right \rangle =\left
\langle \frac{g_{i+1}}{{M_{i+1}}},\frac{g_{i+2}}{{M_{i+1}}},\dots,
\frac{g_s}{{M_{i+1}}} \right\rangle,
\end{align}
where the polynomials $ G_{i+1}, D_{i+2} G_{i+2},\dots,D_{i+2} \cdots  D_s$
also form a zero-dimensional Gr\"obner basis. 

If ${\mathcal G}$ generates an $\langle x,y\rangle$-primary ideal, we
have $D_i=x^{m_i}$ for all $i$, with thus $g_s = x^{m_s}$. Besides,
for all $i$, $G_i(0,y)$ vanishes only at $y=0$, {\it i.e.}
$G_i(0,y)=y^{n_i}$, see~\cite[Theorem~2]{Lazard85}.
\medskip

In terms of data structures, representing ${\mathcal G} =(g_0,\dots,g_s)$
involves $O(s \delta)$ field elements, with $\delta$ the degree of
$I$. As a remark, we note that it would be sufficient to store the
polynomials $\bm D=(D_1,\dots,D_s)$ and $\bm G = (G_0,\dots,G_s)$
instead. If $\iniT \subset \N^2$ is the initial segment determined by
$\bm E$, the structure theorem implies that for $i=0,\dots,s$, $G_i -
y^{n_i}$ is supported on the $i$th translate $\iniT_{\leftarrow i}$ of
$\iniT$. In particular, $\delta_i$ field elements are needed to store
it, with $\delta_i = |\iniT_{\leftarrow i}|$, hence a slightly
improved total of $O(\sum_{i=0}^{s} \delta_i)$ field elements for $\bm
D$ and $\bm G$.


\subsection{Conca and Valla's parametrization}\label{ssec:CV}

In this subsection, we suppose that the tuple $\bm E= (y^{n_0},
x^{m_1} y^{n_1}, \dots, x^{m_{s-1}} y^{n_{s-1}} , x^{m_s})$ is
fixed. Following~\cite{CoVa08}, we are interested in
describing the set of ideals $I$ in $\K[x,y]$ that have initial ideal
generated by $\bm E$. We call this set the {\em Gr\"obner cell} of
$\bm E$, and we write it $\mathcal{C}(\bm E) := \{I\mid in(I) =
\langle\bm E\rangle\}$.  We will also mention a subset of it, the set
of ideals $I$ in $\K[x,y]$ with initial ideal generated by $\bm E$ and
that are $\langle x,y\rangle$-primary; this is called the {\em
  punctual} Gr\"obner cell of $\bm E$, and is written
$\mathcal{C}_0(\bm E)$.

The idea of describing ideals with a prescribed initial ideal goes
back to~\cite{BrGa73,Briancon77,Iarrobino77} for ideals in
$\K[\![x,y]\!]$ and~\cite{CarraFerro88} for $\K[x_1,\dots,x_n]$; it
was then developed in~\cite{NoSp00,Robbiano09,LeRo11} and several
further references. It is known that these Gr\"obner cells, also
called strata, have corresponding moduli spaces that are affine
spaces, but to our knowledge, no general explicit description has yet
been given. In our case however, Conca and Valla obtained
in~\cite{CoVa08} a complete description of Gr\"obner cells and
punctual Gr\"obner cells for bivariate ideals under the lexicographic
order (following previous work of~\cite{ElSt87}, where the dimensions
of these cells were already made explicit).

\begin{example}
For an example of a punctual Gr\"obner cell, taking $\bm E = (y^4,
xy^3, x^2y, x^4)$ as in Figure~\ref{fig:example}, using the facts that
$g_i = x^{m_i} G_i$ and that $G_i(0,y)=y^{n_i}$, we deduce that the
lexicographic Gr\"obner basis of an ideal in $\mathcal{C}_0(\bm E)$
necessarily has the following shape, for some coefficients $c_1,\dots,
c_8$ in $\K$:
\begin{align*}
g_1&= y^4 + c_1xy^2 + c_2xy + c_3 x^3 + c_4 x^2 + c_5 x\\
g_2&= xy^3 + c_6 x^3 + c_7 x^2 \\
g_3&= x^2y + c_8 x^3\\
g_4&= x^4
\end{align*}
So far, though, we have not taken into account the membership equality
in~\eqref{eq:membership}, which imposes relations on the
coefficients~$c_i$. The parametrizations of $\mathcal{C}(\bm
E)$ and $\mathcal{C}_0(\bm E)$ given below resolve this issue.
\end{example}

Recall that we write $d_i = m_i - m_{i-1}$ and $e_i = n_{i-1}-n_{i}$,
for $i=1,\dots,s$.  Given $I$ in $\mathcal{C}(\bm E)$, Conca and Valla
prove the existence and uniqueness of polynomials $(\sigma_{j,i})_{0\le i
  \le s-1, i \le j \le s}$ in $\K[x,y]$ with the following degree 
constraints:
\begin{itemize}\label{deg-ct}
 \item for all $i=0,\dots,s-1$ and $j=i,\dots,s$, $\deg(\sigma_{j,i}, x) < d_{i+1}$
\item for all $i=0,\dots,s-1$, $\sigma_{i,i}$ is in $\K[x]$ and 
  $\deg( \sigma_{j,i}, y) < e_{j}$ holds for $j=i+1,\dots,s$,
\end{itemize}
and such that the following properties hold. Define polynomials $\mathcal
H=(h_0,\dots,h_s)$ in $\K[x,y]$ by
\begin{itemize}
\item $h_s = (x^{d_1}-\sigma_{0,0})\cdots (x^{d_s}-\sigma_{s-1,s-1})$
\item for $i=0,\dots,s-1$,
  \begin{equation}\label{eq:syzygy}
  x^{d_{i+1}} h_i - y^{e_{i+1}} h_{i+1} = \sigma_{i,i} h_i + \sigma_{i+1,i}h_{i+1} + \cdots + \sigma_{s,i} h_s;
\end{equation}
\end{itemize}
then, all polynomials $h_i$'s are in $I$. Since the relations above
imply that for $i=0,\dots,s$, $h_i$ has initial term $x^{m_{i}}
y^{n_{i}} $, $\mathcal{H} = (h_0,\dots,h_s)$ is a minimal Gr\"obner basis of
$I$. (Note that \cref{eq:syzygy} then gives the normal form of
the syzygy between $h_i$ and $h_{i+1}$.)

Conversely, for any choice of the polynomials $\sigma_{j,i}$ satisfying the
degree constraints above, the resulting polynomials $\mathcal{H}$ form a
minimal Gr\"obner basis of an ideal $I$ in $\mathcal{C}(\bm E)$.

\medskip

Let us briefly mention some properties of the polynomials
$h_0,\dots,h_s$. First, we claim that they have $x$-degree either
exactly $m_s$ (for $h_s$), or less than $m_s$, for
$h_0,\dots,h_{s-1}$. This is true for $h_s$ by construction. For the
other indices, this follows from a decreasing induction, by rewriting
\eqref{eq:syzygy} as
\begin{equation}\label{eq:syzygy2}
(x^{d_{i+1}}-\sigma_{i,i}) h_i = y^{e_{i+1}} h_{i+1} + \sigma_{i+1,i}h_{i+1} +
\cdots + \sigma_{s,i} h_s,  
\end{equation}
where all terms $\sigma_{j,i}h_j$ on the right have
$x$-degree less than $d_{i+1} + m_s$.

Next, note that for $i=0,\dots,s$, $(x^{d_1}-\sigma_{0,0})\cdots
(x^{d_i}-\sigma_{i-1,i-1})$ divides $h_i$, and thus all polynomials
$h_i,\dots,h_s$; this follows from~\eqref{eq:syzygy2} by a decreasing
induction (for $i=0$, the empty product is set to $1$). Since $h_i$
has initial term $x^{m_i} y^{n_i}=x^{d_1+\cdots+d_i}y^{n_i}$, we
deduce that $(x^{d_1}-\sigma_{0,0})\cdots (x^{d_i}-\sigma_{i-1,i-1})$ is
precisely the polynomial coefficient of $y^{n_i}$ in $h_i$.

Let then
${\mathcal G} = (g_0,\dots,g_s)$ be the reduced Gr\"obner basis obtained by
inter-reducing $\mathcal{H}$. Since none of the terms in
$(x^{d_1}-\sigma_{0,0})\cdots (x^{d_i}-\sigma_{i-1,i-1})y^{n_i}$ can be reduced
by $h_0,\dots,h_{i-1}$ or $h_{i+1},\dots,h_s$, we see that
$(x^{d_1}-\sigma_{0,0})\cdots (x^{d_i}-\sigma_{i-1,i-1})$ is also the polynomial
coefficient of $y^{n_i}$ in $g_i$. Hence, the polynomials $D_i$ and
$M_i$ that appear in Lazard's structure theorem are respectively given
by $D_i = x^{d_i} - \sigma_{i-1,i-1}$ and $M_i = (x^{d_1}-\sigma_{0,0})\cdots
(x^{d_i}-\sigma_{i-1,i-1})$.

Altogether, the total number $N$ of coefficients that appear in
the polynomials
$(\sigma_{j,i})_{0\le i \le s-1, i \le j \le s}$, for the Gr\"obner cell
$\mathcal{C}(\bm E)$, is given by
\begin{align*}
N&=\sum_{i=0}^{s-1} \left (\sum_{j=i+1}^s d_{i+1} e_j + d_{i+1} \right ) \\
 &=\sum_{i=0}^{s-1} d_{i+1} n_i + \sum_{i=0}^{s-1} d_{i+1} \\
 &=\delta + m_s,
\end{align*}
with $\delta$ the degree of $\bm E$. These coefficients will be written
$\lambda_1,\dots,\lambda_N$ and called {\em Gr\"obner parameters}; this
gives us a bijection $\Phi_{\bm E}$ between $\K^N$ and $\mathcal{C}(\bm E)$.

The elements in the {\em punctual} Gr\"obner cell $\mathcal{C}_0(\bm
E)$ are obtained by setting some of the Gr\"obner parameters to zero,
corresponding to the following extra conditions:
\begin{itemize}
\item the polynomials $\sigma_{0,0},\dots,\sigma_{s-1,s-1}$ vanish (recall that
  for the punctual Gr\"obner cell, we have $D_i=x^{d_i}$ and
  $M_i=x^{m_i}$ for all $i$)
\item $\sigma_{i+1,i}$ is divisible by $x$, for $i=0,\dots,s-1$.
\end{itemize}
The number of remaining coefficients in $\sigma_{1,0},\dots,\sigma_{s,s-1}$
is 
\begin{align*}
N_0&=\sum_{i=0}^{s-1} \left (\sum_{j=i+1}^s d_{i+1} e_j - e_{i+1} \right ) \\
 &=\sum_{i=0}^{s-1} d_{i+1} n_i - \sum_{i=0}^{s-1} e_{i+1} \\
 &=\delta - n_0,
\end{align*}
establishing a bijection between $\K^{N_0}$ and $\mathcal{C}_0(\bm
E)$. Recall that in the primary case, the degree $\delta$ of $\bm E$
is the multiplicity of the ideals in $\mathcal{C}_0(\bm E)$ at the
origin.

\begin{example}
Let us describe the punctual Gr\"obner cell of $\bm E$ in our
running example (\autoref{ex:runningex}). It has dimension $N_0=9-4=5$, so that we can use parameters
$\lambda_1,\dots,\lambda_5$, with polynomials $(\sigma_{i,j})$ of the
form
$$\sigma_{0,0}=\sigma_{1,0}=0,\quad \sigma_{2,0}=\lambda_1 y +\lambda_2,\quad \sigma_{3,0}=\lambda_3,\quad \sigma_{1,1}=n_{2,1} = 0,\quad \sigma_{3,1} = \lambda_4,\quad 
\sigma_{2,2}=0, \quad \sigma_{3,2} = \lambda_5 x.$$
Then, the ideals in
$\mathcal{C}_0(\bm E)$ are exactly those ideals with Gr\"obner bases as
follows:
\begin{align*}
h_0&=    y^4 + \lambda_5xy^3 + \lambda_1xy^2 + (\lambda_1\lambda_5 + \lambda_4)x^2y + \lambda_2xy + \lambda_3x^3 +  \lambda_2\lambda_5x^2\\
h_1&=    xy^3 + \lambda_5x^2y^2 + \lambda_4x^3\\
h_2&=    x^2y + \lambda_5x^3\\
h_3&=   x^4.
\end{align*}
As expected, these are not reduced Gr\"obner bases. After reduction,
we obtain the following polynomials ${\mathcal G}$:
\begin{equation}
\begin{aligned}\label{eq:GL}
g_0&= y^4 + \lambda_1xy^2 + \lambda_2xy + (-\lambda_1\lambda_5^2 + \lambda_3 - 2\lambda_4\lambda_5)x^3 + \lambda_2\lambda_5x^2\\
g_1&= xy^3 + \lambda_4x^3\\
g_2&= x^2y + \lambda_5x^3\\
g_3&= x^4.
\end{aligned}
\end{equation}
\end{example}


\section{Reduction modulo a lexicographic Gr\"obner basis}\label{sec:reduction}

As before, suppose that ${\mathcal G}=(g_0,\dots,g_s)$ is a lexicographic
Gr\"obner basis in $\K[x,y]$, with initial segment $\iniT \subset
\N^2$.  Given $f$ in $\K[x,y]$, we are interested in computing the
remainder $r = f \brem {\mathcal G} \in \K[x,y]_\iniT$; this will be used on
multiple occasions in this paper, and is also an interesting question
in itself. Remarkably, we are not aware of previous work on the
complexity of this particular question.

We start by developing the necessary background as a problem in plane
geometry. This is inspired by work of~\cite{HoLa18}, which was
specific to certain weighted orderings (we discuss this further
below). We continue with algorithms to convert polynomials into a
so-called {\em mixed-radix} representation, and back; the reduction
algorithm itself is then given in the last subsection.


\subsection{A paving problem} \label{sec:paving}

For ${\mathcal G}$ as above and $f$ in $\K[x,y]$, the remainder $r=f \brem
{\mathcal G}$ is uniquely defined, but the quotients $Q_i$ in the relation $f = Q_0
g_0 + \cdots + Q_s g_s + r$ are not. The reduction algorithm will
obtain $r$ by computing the $Q_i$'s one after the other. Hence, to
completely specify the algorithm, we need to make these quotients
unambiguous: whenever a monomial $x^uy^v$ can be reduced by more than
one of the Gr\"obner basis elements, we must prescribe which of the
$g_i$'s is used. The cost of the resulting algorithm will depend in an
essential manner on these decisions.

\cite{HoLa18} introduced a dichotomic scheme, in the context of
reduction modulo certain ``nice'' Gr\"obner bases, for weighted degree
orderings. In this subsection, we adapt their construction to our
situation.

As before, suppose that the initial terms of ${\mathcal G}$ are the monomials
$$\bm E = (y^{n_0}, x^{m_1} y^{n_1}, \dots, x^{m_{s-1}}
y^{n_{s-1}},x^{m_s});$$ we still write $d_i = m_i - m_{i-1}$ and $e_i
= n_{i-1}-n_{i}$, for $i=1,\dots,s$. The set of monomials to which we
will apply the main reduction algorithm is $\{x^u y^v,\ 0\le u <
m_s, \ 0 \le v < n_0\}$, so it has cardinal $n_0 m_s$ (the general
case will be reduced to this situation). In particular, neither $g_0$
nor $g_s$ can reduce any of these monomials.

We can then translate our question into a paving problem in the plane.
We want  to cover $\iniS=\{0,\dots,m_s-1\}\times \{0,\dots,n_0-1\}
- \iniT$ by rectangles, under the following constraints:
\begin{itemize}
\item we use $s-1$ pairwise disjoint rectangles,
  $\iniR_1,\dots,\iniR_{s-1}$, so that $\iniR_i$ will index the set of
  monomials that are reduced using $g_i$
\item for all $i$, $\iniR_i$ has the form $\{m_i,\dots,m_{i +
  \ell_i}-1\} \times \{n_i,\dots,n_{i - h_i}-1\}$, for some positive
  integers $\ell_i,h_i$ such that $i + \ell_i \le s$ and $i -h_i \ge
  0$
\item the union of all $\iniR_i$'s covers $\iniS$.
\end{itemize}
The sequence
$((\ell_1,h_1),\dots,(\ell_{s-1},h_{s-1}))$ is sufficient to specify
such a paving.  Our goal is then to minimize the quantity
$$c :=n_0\sum_{i=1}^{s-1} (m_{i+\ell_i}-m_i) + m_s \sum_{i=1}^{s-1}
(n_{i-h_i}-n_i),$$ where $(m_{i+\ell_i}-m_i)$ and $(n_{i-h_i}-n_i)$
are respectively the width and height of $\iniR_i$. This quantity will
turn out to determine the cost of the reduction algorithm; the target
is to keep $c$ in $\softO(n_0 m_s)$, since we mentioned that $n_0 m_s$
in an upper bound on the number of monomials in the polynomials we
want to reduce.

The following figure shows two possible pavings, for the case $d=4$ of
the family already seen in the proof of Lemma~\ref{bound_s}, with $\bm
E = (y^d,xy^{d-1},\dots,x^d)$. For this family, $n_0=m_s=d$ and
$n_0m_s = d^2$; the strategies showed on the example below have either
$\sum_{i=1}^{s-1} (m_{i+\ell_i}-m_i)$ or $\sum_{i=1}^{s-1}
(n_{i-h_i}-n_i)$ in $\Theta(d^2)$, so $c$ is in
$\Theta(d^3)=\Theta((n_0m_s)^{1.5})$ in either case.

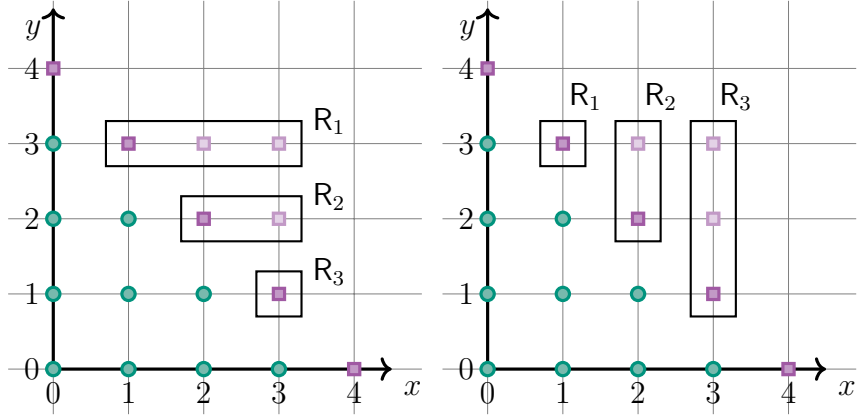
\begin{figure}[H]
\centering
  \begin{tikzpicture}
    [  
      roundnode/.style={circle, draw=PineGreen!90, fill=PineGreen!50, very thick, minimum size=5pt, inner sep=0pt, outer sep=0pt},
      square/.style={regular polygon,regular polygon sides=4, draw=Purple!90, fill=Purple!50, very thick, minimum size=6pt, inner sep=0pt, outer sep=0pt},
      lightsquare/.style={regular polygon,regular polygon sides=4, draw=Purple!50, fill=Purple!20, very thick, minimum size=6pt, inner sep=0pt, outer sep=0pt},
    ]
    \draw[step=1cm,gray,very thin] (-0.6,-0.6) grid (4.9,4.9);
    \draw[very thick,->] (-0.1,0) -- (4.5,0) node[anchor=north west] {$x$};
    \draw[very thick,->] (0,-0.1) -- (0,4.8) node[anchor=north east] {$y$};
    \foreach \x in {0,1,2,3,4} \draw (\x cm,1pt) -- (\x cm,-1pt) node[anchor=north] {$\x$};
    \foreach \y in {0,1,2,3,4} \draw (1pt,\y cm) -- (-1pt,\y cm) node[anchor=east] {$\y$};
    \node[roundnode] at (0,0) {};
    \node[roundnode] at (0,1) {};
    \node[roundnode] at (0,2) {};
    \node[roundnode] at (0,3) {};
    \node[square] at (0,4) {};
    \node[roundnode] at (1,0) {};
    \node[roundnode] at (1,1) {};
    \node[roundnode] at (1,2) {};
    \node[square] at (1,3) {};
    \node[roundnode] at (2,0) {};
    \node[roundnode] at (2,1) {};
    \node[square] at (2,2) {};
    \node[lightsquare] at (2,3) {};
    \node[roundnode] at (3,0) {};
    \node[square] at (3,1) {};
    \node[lightsquare] at (3,2) {};
    \node[lightsquare] at (3,3) {};
    \node[square] at (4,0) {};

    \draw[thick] (0.7,2.7) rectangle (3.3,3.3) node[right] {$\iniR_1$};
    \draw[thick] (1.7,1.7) rectangle (3.3,2.3) node[right] {$\iniR_2$};
    \draw[thick] (2.7,0.7) rectangle (3.3,1.3) node[right] {$\iniR_3$};
  \end{tikzpicture}
  \begin{tikzpicture}
    [  
      roundnode/.style={circle, draw=PineGreen!90, fill=PineGreen!50, very thick, minimum size=5pt, inner sep=0pt, outer sep=0pt},
      square/.style={regular polygon,regular polygon sides=4, draw=Purple!90, fill=Purple!50, very thick, minimum size=6pt, inner sep=0pt, outer sep=0pt},
      lightsquare/.style={regular polygon,regular polygon sides=4, draw=Purple!50, fill=Purple!20, very thick, minimum size=6pt, inner sep=0pt, outer sep=0pt},
    ]
    \draw[step=1cm,gray,very thin] (-0.6,-0.6) grid (4.9,4.9);
    \draw[very thick,->] (-0.1,0) -- (4.5,0) node[anchor=north west] {$x$};
    \draw[very thick,->] (0,-0.1) -- (0,4.8) node[anchor=north east] {$y$};
    \foreach \x in {0,1,2,3,4} \draw (\x cm,1pt) -- (\x cm,-1pt) node[anchor=north] {$\x$};
    \foreach \y in {0,1,2,3,4} \draw (1pt,\y cm) -- (-1pt,\y cm) node[anchor=east] {$\y$};
    \node[roundnode] at (0,0) {};
    \node[roundnode] at (0,1) {};
    \node[roundnode] at (0,2) {};
    \node[roundnode] at (0,3) {};
    \node[square] at (0,4) {};
    \node[roundnode] at (1,0) {};
    \node[roundnode] at (1,1) {};
    \node[roundnode] at (1,2) {};
    \node[square] at (1,3) {};
    \node[roundnode] at (2,0) {};
    \node[roundnode] at (2,1) {};
    \node[square] at (2,2) {};
    \node[lightsquare] at (2,3) {};
    \node[roundnode] at (3,0) {};
    \node[square] at (3,1) {};
    \node[lightsquare] at (3,2) {};
    \node[lightsquare] at (3,3) {};
    \node[square] at (4,0) {};
    \draw[thick] (0.7,2.7) rectangle (1.3,3.3) node[above] {$\iniR_1$};
    \draw[thick] (1.7,1.7) rectangle (2.3,3.3) node[above] {$\iniR_2$};
    \draw[thick] (2.7,0.7) rectangle (3.3,3.3) node[above] {$\iniR_3$};
  \end{tikzpicture}
 \caption{two possible pavings with $d=4$.}
  \label{fig:paving}
\end{figure}
For this family, a better solution is given below.
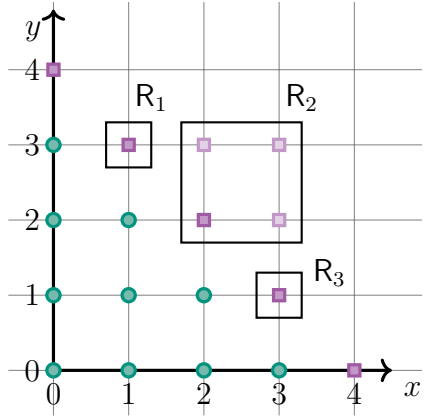
\begin{figure}[H]
\centering
  \begin{tikzpicture}
    [  
      roundnode/.style={circle, draw=PineGreen!90, fill=PineGreen!50, very thick, minimum size=5pt, inner sep=0pt, outer sep=0pt},
      square/.style={regular polygon,regular polygon sides=4, draw=Purple!90, fill=Purple!50, very thick, minimum size=6pt, inner sep=0pt, outer sep=0pt},
      lightsquare/.style={regular polygon,regular polygon sides=4, draw=Purple!50, fill=Purple!20, very thick, minimum size=6pt, inner sep=0pt, outer sep=0pt},
    ]
    \draw[step=1cm,gray,very thin] (-0.6,-0.6) grid (4.9,4.9);
    \draw[very thick,->] (-0.1,0) -- (4.5,0) node[anchor=north west] {$x$};
    \draw[very thick,->] (0,-0.1) -- (0,4.8) node[anchor=north east] {$y$};
    \foreach \x in {0,1,2,3,4} \draw (\x cm,1pt) -- (\x cm,-1pt) node[anchor=north] {$\x$};
    \foreach \y in {0,1,2,3,4} \draw (1pt,\y cm) -- (-1pt,\y cm) node[anchor=east] {$\y$};
    \node[roundnode] at (0,0) {};
    \node[roundnode] at (0,1) {};
    \node[roundnode] at (0,2) {};
    \node[roundnode] at (0,3) {};
    \node[square] at (0,4) {};
    \node[roundnode] at (1,0) {};
    \node[roundnode] at (1,1) {};
    \node[roundnode] at (1,2) {};
    \node[square] at (1,3) {};
    \node[roundnode] at (2,0) {};
    \node[roundnode] at (2,1) {};
    \node[square] at (2,2) {};
    \node[lightsquare] at (2,3) {};
    \node[roundnode] at (3,0) {};
    \node[square] at (3,1) {};
    \node[lightsquare] at (3,2) {};
    \node[lightsquare] at (3,3) {};
    \node[square] at (4,0) {};
    \draw[thick] (0.7,2.7) rectangle (1.3,3.3) node[above] {$\iniR_1$};
    \draw[thick] (1.7,1.7) rectangle (3.3,3.3) node[above] {$\iniR_2$};
    \draw[thick] (2.7,0.7) rectangle (3.3,1.3) node[right] {$\iniR_3$};
  \end{tikzpicture}
 \caption{an improved paving.}
  \label{fig:sol_paving}
\end{figure}
This design was introduced 
in~\cite{HoLa18}, for families $\bm E$ similar to the one in the
example, where the step widths $d_i$ are (almost) constant, and all
step heights $e_i$ are equal to $1$. The construction we give below
for arbitrary inputs is derived from it in a direct manner. In what
follows, ${\rm val}_2(i)$ denotes the $2$-adic valuation of a positive
integer $i$.

\begin{definition}\label{def:elli}
For $i=1,\dots,s-1$, define:
\begin{itemize}
\item $h_i = 2^{{\rm val}_2(i)}$
\item $\ell_i = \min(h_i , s-i)$
\end{itemize}
\end{definition}
As a result, the rectangle $\iniR_i$ is $\{m_i,\dots,m_{\min(i+h_i ,s)}-1\} 
\times \{n_i,\dots,n_{i-h_i }-1\}$.

\begin{proposition}
  For any $s$ and any choices of $m_1,\dots,m_s$ and
  $n_0,\dots,n_{s-1}$, the rectangles $\iniR_1,\dots,\iniR_{s-1}$ are
  pairwise disjoint, cover $\iniS=\{0,\dots,m_s-1\}\times
  \{0,\dots,n_0-1\} - \iniT$, and satisfy $i + \ell_i \le s$ and
  $i -h_i\ge 0$ for all $i$.
\end{proposition}
\begin{proof}
  The last claim is a direct consequence of the definitions. We prove
  the rest of the proposition by reduction to the case where all $d_i$'s and
  $e_i$'s are equal to one. The proof is technical but raises no
  special difficulty.

  For any positive integer $s$, we define the monomials
  ${\mathscr{E}}_s=(x^i y^{s-i},\ 0 \le i \le s)$, 
  the initial segment $\mathscr{T}_s$ determined by $\mathscr{E}_s$ and
  $\mathscr{S}_s=\{0,\dots,s-1\}\times \{0,\dots,s-1\} -
  \mathscr{T}_s$; note that $\mathscr{T}_s$ is the set of all pairs of
  non-negative integers $(a,b)$ with $b < s-a$. Finally, for
  $i=1,\dots,s-1$ we define the rectangle
  $\mathscr{R}_{i,s}=\{i,\dots,\min(i+h_i ,s)-1\} \times
  \{s-i,\dots,s-i+h_i -1\} \subset \mathscr{S}_s$.
  
  We start from $m_1,\dots,m_s$ and $n_0,\dots,n_{s-1}$ as in the
  proposition's statement, with corresponding sets $\iniT$ and $\iniS$ in
  $\N^2$. Take a point $(u,v)$ in $\iniS$. Because $u < m_s$, there
  exists a unique pair $(\alpha,u')$ such that $u=m_\alpha + u'$, with
  $0 \le \alpha \le s-1$ and $0\le u' < d_{\alpha+1}$. Similarly,
  because $v < n_0$, there exists a unique pair $(\beta,v')$ such that
  $v=n_\beta + v'$, with $1 \le \beta \le s$ and $0 \le v' <
  e_\beta$. We claim that $(\alpha,s-\beta)$ is in the set
  $\mathscr{S}_s$ defined in the previous paragraph, and that for
  $i=1,\dots,s-1$, $(u,v)$ is in the rectangle $\iniR_i$ if and only if
  $(\alpha,s-\beta)$ is in the rectangle $\mathscr{R}_{i,s}$.
  \begin{itemize}
  \item For the first claim, we already pointed out the inequalities
    $0\le \alpha \le s-1$ and $1 \le \beta\le s$, which gives $0 \le
    s-\beta \le s-1$, so that $(\alpha,s-\beta)$ is in the square
    $\{0,\dots,s-1\}\times \{0,\dots,s-1\}$. On the other hand, we
    have $v \ge n_\alpha$ (otherwise $(\alpha,\beta)$ would be in
    $\iniT$), and so $\beta \le \alpha$ and $s-\beta \ge
    s-\alpha$. This proves that the point $(\alpha,s-\beta)$ is not in
    $\mathscr{T}_s$, so altogether, it lies in $\mathscr{S}_s$.
  \item For the second claim, note that since $u=m_\alpha+u'$, with
    $0\le u' < d_{\alpha+1}$, $m_i \le u <
    m_{\min(i+h_i ,s)}$ is equivalent to $i \le \alpha <
    \min(i+h_i ,s)$. Similarly, the inequalities $n_i \le v <
    n_{i-h_i }$ are equivalent to $s-i \le s-\beta <
    s-i+h_i $. This proves the claim.

  \end{itemize}
  To conclude, it is now sufficient to prove that for all $s$, the
  following property, written $P(s)$, holds: the rectangles
  $\mathscr{R}_{1,s},\dots,\mathscr{R}_{s-1,s}$ are pairwise disjoint
  and cover $\mathscr{S}_s$. First, we prove it for $s$ a power of
  two, of the form $s=2^k$, by induction on $k \ge 1$. For $k=1$ (so
  $s=2$), there is nothing to prove, as $\mathscr{S}_2 = \{1\} \times
  \{1 \}=\mathscr{R}_{1,2}$.

  Supposing that $P(s)$ is true for $s=2^k$, we now prove it for
  $s'=2s$. For $\mathscr{S}$ a subset of $\N^2$, we write $\mathscr{S}
  \cap \{x \le t\}$ for the set of all $(x,y)$ in $\mathscr{S}$ with
  $x \le t$. The sets $\mathscr{S} \cap \{x \ge t\}$, $\mathscr{S}
  \cap \{x \le t, y \le t'\}$, etc, are defined similarly.

  First, we note that for any power of two $\sigma=2^t$ and
  $i=1,\dots,\sigma-1$, we have $i+h_i \le \sigma$, so the
  rectangle $\mathscr{R}_{i,\sigma}$ is simply
  $\mathscr{R}_{i,\sigma}=\{i,\dots,i+h_i -1\} \times
  \{\sigma-i,\dots,\sigma-i+h_i -1\}$. As a result, the
  rectangles $\mathscr{R}_{1,s'},\dots,\mathscr{R}_{s-1,s'}$ are
  translates of $\mathscr{R}_{1,s},\dots,\mathscr{R}_{s-1,s}$ by
  $(0,s)$, so by the induction assumption, they are pairwise disjoint,
  cover $\mathscr{S}_{s'} \cap \{ x \le s-1 \}$, and do not meet
  $\mathscr{S}_{s'} \cap \{ x \ge s \}$ (on
  Figure~\ref{fig:sol_paving}, we have $s=2$, $s'=4$, and there is only
  one such rectangle, written $R_1$).  Since
  $h_i =\varepsilon_{i+s}$ for $i=1,\dots,s-1$, we also
  deduce that the rectangles
  $\mathscr{R}_{s+1,s'},\dots,\mathscr{R}_{2s-1,s'}$ are translates of
  $\mathscr{R}_{1,s},\dots,\mathscr{R}_{s-1,s}$ by $(s,0)$. Thus, they
  are pairwise disjoint, cover $\mathscr{S}_{s'} \cap \{ x \ge s, y
  \le s-1 \}$, and do not meet $\mathscr{S}_{s'} \cap \{ x \ge s, y
  \ge s\}$ (on Figure~\ref{fig:sol_paving}, this is $R_3$).  Finally,
  $\mathscr{R}_{s,s'}$ is the rectangle $\{s,2s-1\} \times \{s,2s-1\}$
  (on Figure~\ref{fig:sol_paving}, this is $R_2$). Altogether, $P(s')$
  holds and the induction is complete.

  The last step is to prove that $P(s)$ holds for all $s$, knowing
  that it holds for all powers of two. Let $s$ be arbitrary and let
  $s'$ be the first power of two greater than or equal to $s$, so that
  we know that $P(s')$ holds. Let $s''=s'/2$. Since $s' < 2s$, $s''
  \le s$. For $i < s''$, $\mathscr{R}_{i,s}=\mathscr{R}_{i,s'} -
  (s'-s,0)$, whereas for $s'' \le i \le s-1$,
  $\mathscr{R}_{i,s}=\mathscr{R}_{i,s'} \cap \{x \le s-1\}-(s'-s,0)$.
  Knowing $P(s')$, this implies that all these sets are pairwise
  disjoint. In addition, they cover $\mathscr{S}_{s'}\cap \{x\le
  s-1\}-(s'-s,0)$, which is none other that $\mathscr{S}_s$. Thus,
  $P(s)$ is proved.
\end{proof}

The key property of this construction is that the corresponding value
of $c =n_0\sum_{i=1}^{s-1} (m_{i+\ell_i}-m_i) + m_s \sum_{i=1}^{s-1}
(n_{i-h_i}-n_i)$ is softly linear in $n_0 m_s$. This is close to
optimal, since the inequalities $\sum_{i=1}^{s-1}(m_{i+\ell_i}-m_i)
\ge m_s-1$ and $\sum_{i=1}^{s-1} (n_{i-h_i}-n_i) \ge n_0 -1$ imply
that $c$ is in $\Omega(n_0m_s)$.

\begin{proposition}\label{lemma:costpaving}
  For $\iniR_1,\dots,\iniR_{s-1}$ as above, $c =n_0\sum_{i=1}^{s-1}
  (m_{i+\ell_i}-m_i) + m_s \sum_{i=1}^{s-1} (n_{i-h_i}-n_i)$ is in
  $\softO(n_0 m_s)$.
\end{proposition}
\begin{proof}
  We prove that with the choices in \cref{def:elli}, $\sum_{i=1}^{s-1}
  (m_{i+\ell_i}-m_i)$ is in $\softO(m_s)$; we omit the remaining part
  of the argument that proves that $\sum_{i=1}^{s-1} (n_{i-h_i}-n_i)$
  is in $\softO(n_0)$ in a similar manner.

  First, we reduce to the case where $s$ is a power of $2$. For $i \ge
  s$, set $\ell_i=0$ and $m_i=m_s$; the sum $\sum_{i=1}^{s-1}
  (m_{i+\ell_i}-m_i)$ is then equal to $\sum_{i=1}^{s'-1}
  (m_{i+\ell_i}-m_i)$, where $s'=2^k$ is the first power of two
  greater than or equal to $s$. Besides, this convention implies
  $m_{i+\ell_i}=m_{i+h_i }$ for all $i$.

  For a given $\kappa$ in $\{0,\dots,k-1\}$, the indices $i \in
  \{1,\dots,s'-1\}$ of 2-adic valuation $\kappa$ are the integers
  $2^\kappa (1+2j)$, for $j=1,\dots,2^{k-\kappa-1}-1$, so we can
  rewrite the sum $\sum_{i=1}^{s'-1} (m_{i+\ell_i}-m_i)$ as
  $$\sum_{\kappa = 0}^{k-1} \sum_{j = 0}^{2^{k-\kappa-1}-1}
  (m_{2^\kappa(1+ 2j)+2^\kappa} - m_{2^\kappa(1+ 2j)}) = \sum_{\kappa
    = 0}^{k-1} \sum_{j = 0}^{2^{k-\kappa-1}-1} (d_{2^\kappa(1+ 2j)+1}
  + \cdots + d_{2^\kappa(1+ 2j)+2^\kappa}),$$ where we set $d_i=0$ for
  $i > s$. In particular, for a fixed $\kappa$, the last index
  occurring at summation step $j$ is less than the first index
  occurring at $j+1$, so the inner sum is bounded above by
  $\sum_{i=1}^{s'} d_i = m_s$. It follows that $\sum_{i=1}^{s-1}
  (m_{i+\ell_i}-m_i)\le \sum_{\kappa = 0}^{k-1} m_s \in
  O(m_s\log(s))$. Since $s \le m_s$, our claim is proved.
\end{proof}


\subsection{Mixed radix representation}\label{ssec:mr}

In this subsection, we discuss an alternative basis for our
polynomials.
Our motivation is the following: if $\mathcal G = (g_0,\dots,g_s)$ is the
minimal, reduced lexicographic Gr\"obner basis that we want to use in our reduction
algorithm, we saw that for $i=0,\dots,s$, $g_i$ can be written as $M_i
G_i$, with $M_i$ of degree $m_i$ in $\K[x]$ and $G_i \in \K[x,y]$
monic in $y$, of degree $n_i$ in $y$. Recall also that for
$i=1,\dots,s$ we write $D_{i}=M_{i}/M_{i-1}$, which is a polynomial of
degree $d_i=m_i-m_{i-1}$ in $\K[x]$.

The main reduction algorithm will perform many univariate reductions
modulo the polynomials $M_1,\dots,M_s$. When working with $\langle
x,y\rangle$-primary ideals, all $M_i$'s are powers of $x$, so these
operations are free of arithmetic cost. In general, though, this is
not the case anymore, if the inputs are represented on the monomial
basis.  In this paragraph, we introduce a mixed radix representation
where reductions by the $M_i$'s are free, and we discuss conversion
algorithms.

Given polynomial $\bm K = (K_1,\dots,K_t)$ in $\K[x]$, with
respective degrees $k_1,\dots,k_t$, and writing $h=k_1 + \cdots + k_t$,
we consider the $\K$-linear mapping
\begin{align*}
\Phi_{\bm K}:\K[x]_{< k_1} \times \cdots \times \K[x]_{<k_t} & \to \K[x]_{ < h}\\
(F_1,\dots,F_t) & \mapsto F_1 +  K_1 F_2 + K_1 K_2 F_3 + \cdots + K_1 \cdots K_{t-1} F_t.
\end{align*}
The domain and codomain both have dimension $h$; from this, we easily
deduce that $\Phi_{\bm K}$ is a $\K$-vector space isomorphism. For
$F$ in $\K[x]_{< h}$, we call $(F_1,\dots,F_t)=\Phi_{\bm K}^{-1}(F)$
its {\em mixed radix} representation with respect to the basis $\bm K$.

We will rely on the following fact: given $(F_1,\dots,F_t)=\Phi_{\bm
  K}^{-1}(F)$, for $i$ in $\{1,\dots,t\}$, the mixed radix
representation of $F \bdiv K_1 \cdots K_i$, with respect to the basis
$(K_{i+1},\dots,K_t)$, is $(F_{i+1},\dots,F_t)$, so we have access to it
free of cost. Similarly, the mixed radix representation of $F \brem
K_1\cdots K_i$, with respect to the basis $(K_1,\dots,K_i)$, is
$(F_1,\dots,F_i)$. In particular, if $F$ is given in its mixed radix
representation, quotient and remainder by the product $K_1\cdots K_i$
are free; we still denote these operations by div and rem.

Conversely, for $F$ of degree less than $k_{i+1} + \cdots + k_t$,
given on the mixed radix basis associated to $(K_{i+1},\dots,K_t)$ as a
vector $(F_{i+1},\dots,F_t)$, the mixed radix representation of $K_1
\cdots K_i F$, for the basis $(K_{1},\dots,K_t)$, is
$(0,\dots,0,F_{i+1},\dots,F_t)$, so it can be computed for free.

For completeness, we give algorithms with softly linear runtime to
apply $\Phi_{\bm K}$ and its inverse. These are elementary variants of
the algorithms for Chinese remaindering
in~\cite[Chapter~10.3]{GaGe13}, or generalized Taylor
expansion~\cite[Chapter~9.2]{GaGe13}. We start with the conversion
from the mixed radix to monomial representation.

\begin{algorithm}[H]
  \algoCaptionLabel{FromMixedRadix}{(F_1,\dots,F_t), (K_1,\dots,K_t)}
  \begin{algorithmic}[1]
    \Require $(F_1,\dots,F_t)$ in $\K[x]_{< k_1} \times \cdots \times \K[x]_{<k_t}$, $\bm K = (K_1,\dots,K_t)$ of
    respective degrees $k_1,\dots,k_t$
    \Ensure $\Phi_{\bm K}(F_1,\dots,F_t) \in \K[x]_{< h}$, with $h=k_1 + \cdots + k_t$
    \State \InlineIf{$t=1$} {\Return $F_1$}
    \State $t' \gets \lceil t/2\rceil$
    \State $L \gets  \Call{algo:FromMixedRadix}{(F_1,\dots,F_{t'}), (K_1,\dots,K_{t'})}$
    \State $R \gets  \Call{algo:FromMixedRadix}{(F_{t'+1},\dots,F_t), (K_{t' +1},\dots,K_t)}$
    \If{$R=0$}\label{step:check0}
    \State \Return $L$
    \Else
    \State \Return $L + K_1 \cdots K_{t'} R$ \label{step:mulP}
    \EndIf
  \end{algorithmic}
\end{algorithm}
Correctness is clear: if we write $F=\Phi_{\bm K}(F_1,\dots,F_t)$,
then the previous discussion shows that $L = F \brem K_1 \cdots
K_{t'}$ and $R = F \bdiv K_1 \cdots K_{t'}$, so that the
output is indeed $F$. If we enter \cref{step:mulP}, computing $P$
takes $\softO(k_1 + \cdots + k_{t'})$ operations
$(+,\times)$ in $\K$~\cite[Lemma~10.4]{GaGe13}; however, in this case
$R$ is nonzero, so $F$ has degree at least $k_1 + \cdots + k_{t'}$,
and $\softO(k_1 + \cdots + k_{t'})$ is $\softO(\deg(F))$.  It follows
that, excluding the recursive calls, the cost of a single call to
\algoName{algo:FromMixedRadix} is $\softO(\deg(F))$ if $\deg(F) \ge k_1
+ \cdots + k_{t'}$, and zero otherwise.

There are $O(\log(\deg(F)))$ levels of the recursion tree that will
incur a nonzero cost, and the degrees of the polynomials computed at
any of these levels add up to at most $\deg(F)$. Hence, the overall
cost is $\softO(\deg(F))$ operations $(+,\times)$ in $\K$.

For the inverse operation, the algorithm is recursive as well.  Using
the test at \cref{step:degF}, we avoid doing any computation if $F$
has degree less than $k_1 + \cdots + k_{t'}$. The discussion is as
above, yielding a runtime of $\softO(\deg(F))$ operations $(+,\times)$
in $\K$.

\begin{algorithm}[H]
  \algoCaptionLabel{ToMixedRadix}{F, (K_1,\dots,K_t)}
  \begin{algorithmic}[1]
    \Require $F$ in $\K[x]_{< h}$, $\bm K=(K_1,\dots,K_t)$ of
    respective degrees $k_1,\dots,k_t$, with $h=k_1 + \cdots + k_t$
    \Ensure $(F_1,\dots,F_t)=\Phi_{\bm K}^{-1}(F)$
    \State \InlineIf{$t=1$} {\Return $(F)$}
    \State $t' \gets \lceil t/2\rceil$
    \If{$\deg(F) < k_1 + \cdots + k_{t'}$}\label{step:degF}
    \State \Return $\Call{algo:ToMixedRadix}{F, (K_1,\dots,K_{t'})} \text{~cat~} (0,\dots,0)$ \Comment{$t-t'$ zeros}
    \Else
    \State $P \gets K_1 \cdots K_{t'}$
    \State $Q,R \gets F \bdiv P, F \brem P$
    \State \Return $\Call{algo:ToMixedRadix}{R, (K_1,\dots,K_{t'})} \text{~cat~}\Call{algo:ToMixedRadix}{Q, (K_{t' +1},\dots,K_t)}$
    \EndIf
  \end{algorithmic}
\end{algorithm}

In the next paragraphs, we apply these algorithms to polynomials in
$\K[x,y]$ (we use the same names for the algorithms). In this case, we
simply proceed coefficient-wise with respect to $y$, the mixed-radix
representation of $F \in \K[x,y]$ being now a two-dimensional array.
If the sum of the degrees of $K_1,\dots,K_t$ is $h$, and for $F$ in
$\K[x,y]$ supported on an initial segment $\iniU$, with also $\deg(F,x) <
h$, the runtime of both algorithms is $\softO(|\iniU|)$.


\subsection{The main algorithm}\label{ssec:reduction}

We can now use the results from the previous subsections in order to
give an algorithm for the reduction of a polynomial $f\in\K[x,y]$
modulo a minimal reduced lexicographic Gr\"obner basis
$\mathcal{G}=(g_0,\dots,g_s)$. For the time being, we only consider
the ``balanced'' case, where $f$ is already reduced modulo $g_0$ and
$g_s$. Let us write, as usual, the initial terms of ${\mathcal G}$ as
$$\bm E = (y^{n_0}, x^{m_1} y^{n_1}, \dots, x^{m_{s-1}}
y^{n_{s-1}},x^{m_s})$$ with the $m_i$'s increasing and the $n_i$'s
decreasing, and let $\iniS$ be the rectangle $\{0,\dots,m_s-1\}\times
\{0,\dots,n_0-1\}$. Then, our assumption is that $f$ is in
$\K[x,y]_\iniS$. More general inputs can be handled by performing a
reduction by $(g_0,g_s)$ first; this is discussed in the last
paragraph of this section.

In what follows, we let $\iniT$ be the initial segment determined by
${\mathcal G}$, and $\delta=\dim_\K(\K[x,y]/\mathcal{G})$ be the
degree of~${\mathcal G}$.

\paragraph{Overview of the algorithm} 
Given $f$ in $\K[x,y]$ with $\deg(f,x) < m_s$ and $\deg(f,y) < n_0$,
our main algorithm \nameref{algo:Reduction} computes $r = f \brem \mathcal G$ by calling $s-1$ times a procedure called
\nameref{algo:PartialReduction}, which is described further. The main
algorithm returns the remainder $r$, together with quotients
$Q_1,\dots,Q_{s-1}$, such that $f=Q_1 g_1 + \cdots + Q_s g_s + r$.  While we
do not need the quotients in this paper, we return them as a byproduct
that could possibly be of use in other contexts (the algorithm does
not compute the last quotient $Q_s$, but it would be straightforward
to deduce it from the output, if needed). Since we assume $\deg(f,y) <
n_0$, $g_0$ does not appear in the reduction equality.

The mixed radix basis is used throughout the algorithm to handle
intermediate data; input and output are on the usual monomial
basis.

\begin{algorithm}[H]
  \algoCaptionLabel{Reduction}{f,{\mathcal G}}
  \begin{algorithmic}[1]
    \Require $f$ in $\K[x,y]$, ${\mathcal G}=(g_0,\dots,g_s)$ as above
    \Assumptions  $\deg(f,x) < m_s$, $\deg(f,y) < n_0$
    \Ensure $f \brem {\mathcal G}$ and quotients $Q_1,\dots,Q_{s-1}$
    \State $M_0 \gets 1$, $G_0 \gets g_0$
    \For{$i=1,\dots,s$} \label{step:mainFOR}
    \State $M_i \gets \textsc{PolynomialCoefficient}(g_i,y^{n_i}) \in \K[x]$\label{step:RMi}
    \State $G_i \gets g_i \bdiv M_i$
    \State $D_i \gets M_{i} \bdiv M_{i-1}$
    \EndFor
    \State $f^{(0)} \gets \Call{algo:ToMixedRadix}{f, (D_1,\dots,D_s)}$
    \Comment{$f^{(0)}$ is on the mixed radix basis}
    \For{$i=1,\dots,s-1$}
    \State {$f^{(i)},Q_i \gets \Call{algo:PartialReduction}{f^{(i-1)}, i}$}
    \Comment{all $f^{(i)}$ are on the mixed radix basis}
    \EndFor
    \State \Return $\Call{algo:FromMixedRadix}{f^{(s-1)},D_1,\dots,D_s},Q_1,\dots,Q_{s-1}$

  \end{algorithmic}
\end{algorithm}
To simplify notation, the polynomials $g_0,\dots,g_s$,
$G_0,\dots,G_s$, $M_0,\dots,M_s$ and $D_1,\dots,D_s$, the latter of
which are computed at the beginning of the main algorithm, are assumed
to be known in our calls to \algoName{algo:PartialReduction}, rather
than passed as arguments.

The main result in this section is the following proposition. The
runtime given here is softly linear in $n_0m_s$ and $s\delta$: the
former represents the size of the input polynomial $f$, and the latter
is the upper bound on the number of coefficients needed to represent
${\mathcal G}$ discussed in \cref{ssec:structTh}. Whether a better
algorithm is possible (which would not need all coefficients of ${\mathcal G}$, but only, for instance, its Gr\"obner parameters) is not clear
to us.

\begin{proposition}\label{prop:red}
  Given $f$ and ${\mathcal G}$, with $\deg(f,x) < m_s$ and $\deg(f,y) <
  n_0$, \algoName{algo:Reduction} returns $f \brem {\mathcal G}$ using
  $\softO(n_0 m_s + s \delta)$ operations $(+,\times)$ in $\K$.
\end{proposition}

Before proving the proposition, we mention an important particular
case, where a simplified runtime is available.  Suppose that $e_i=1$
for all $i$, that is, that all steps in the staircase have height $1$.
In this case, $n_0=s$, and since we have $m_s \le \delta$, we obtain
$n_0 m_s \le s \delta$. In other words, the runtime of the algorithm
is simply $\softO(s \delta)$.

\paragraph{A single reduction step} 
We start with a description of the key subroutine,
\algoName{algo:PartialReduction}.

In \cref{sec:paving}, we described a way to cover
$\iniS=\{0,\dots,m_s-1\}\times \{0,\dots,n_0-1\} - \iniT$ by
rectangles $\iniR_i = \{m_i,\dots,m_{i +\ell_i}-1\} \times
\{n_i,\dots,n_{i- h_i}-1\}$, with
 $h_i =2^{{\rm val}_2(i)}$ 
and
$\ell_i=\min(h_i , s-i)$ 
for
$i=1,\dots,s-1$. 

In \algoName{algo:PartialReduction}, we are given $f
\in\K[x,y]_{\iniS}$, and an index $i$ in $\{1,\dots,s-1\}$. The
essential operation is a Euclidean division with respect to the
variable $y$, with coefficients suitably reduced with respect to $x$;
the mixed radix basis is used to control the cost of this reduction in
$x$. We prove below that the output $r$ has the same remainder as $f$
modulo ${\mathcal G}$; we also return the partial quotient $Q$, which
is supported on a translate of $\iniR_i$.

\begin{algorithm}
  \algoCaptionLabel{PartialReduction}{f,i}
  \begin{algorithmic}[1]
    \Require $f$ in $\K[x,y]$, $i$ in $\{1,\dots,s-1\}$
    \Assumptions  $\deg(f,x) < m_s$, $\deg(f,y) < n_0$, $i$ in $\{1,\dots,s-1\}$. $f$ is given on the
    mixed radix basis associated to $D_1,\dots,D_s$
    \Ensure $r$ and $Q$ in $\K[x,y]$. $r$ is given on the mixed radix basis associated to $D_1,\dots,D_s$
    \State $h_i \gets 2^{{\rm val}_2(i)}$, $\ell_i \gets  \min(h_i,s-i)$
    \State $F_1 \gets f \bdiv M_i$  \Comment{division in the mixed radix basis} \label{getF1}
    \Statex\Comment{$F_1$ is given on the mixed radix basis associated to $D_{i+1},\dots,D_s$}
    \State $F_2 \gets F_1 \brem D_{i+1}\cdots D_{i+\ell_i} $\label{getF2}
    \Comment{division in the mixed radix basis}
    \Statex\Comment{$F_2$ is given on the mixed radix basis associated to $D_{i+1},\dots,D_{i+\ell_i}$}
    \State $ F_3 \gets F_2 \brem y^{n_{i-h_i}}$\label{getF3}
    \State $F_4 \gets \Call{algo:FromMixedRadix}{F_3,(D_{i+1},\dots,D_{i+\ell_i}) }$ \Comment{$F_4$ is on the monomial basis}
    \label{getF4}
    \State $q \gets F_4 \bdiv G_i$ in $\A[y]$ \Comment{$G_i$ such that $g_i = M_iG_i$, $\A =\K[x]/\langle D_{i+1}\cdots D_{i+\ell_i} \rangle$}\label{getq}
    \State let $Q$ be the canonical lift of $q$ to $\K[x,y]$ \Comment{$\deg(Q,x) < m_{i+\ell_i}-m_i$}
    \State $V \gets \Call{algo:Multiply}{Q, \{0,\dots,m_{i+\ell_i}-m_i-1\}\times \{0,\dots,n_{i-h_i}-n_i-1\},G_i,\iniT}$
    \label{getV}
    \Statex \Comment{$V = Q G_i$ on the monomial basis}
    \State $V_1 \gets V\brem (D_{i+1}\cdots D_s)$
    \Comment{$V_1 = Q G_i \brem (D_{i+1}\cdots D_s)$ on the monomial basis}\label{getV1}
    \State $V_2 \gets \Call{algo:ToMixedRadix}{V_1, (D_{i+1},\dots,D_s)}$\label{getV2}
    \Statex \Comment{$V_2=Q G_i \brem (D_{i+1}\cdots D_s)$, given on the mixed radix basis associated to $D_{i+1},\dots,D_s$}
    \State $V_3 \gets M_i V_2$
    \Comment{multiplication in the mixed radix basis}
    \Statex\Comment{$V_3=Q g_i \brem M_s$, given on the mixed radix basis associated to $D_{1},\dots,D_s$}
    \State $r \gets f - V_3$  \label{step:red_G}
    \Comment{subtraction  in the mixed radix basis}
    \Statex\Comment{$r=(f-Q g_i) \brem M_s$, given on the mixed radix basis associated to $D_{1},\dots,D_s$}
    \State \Return $r,Q$
  \end{algorithmic}
\end{algorithm}

\begin{lemma}\label{lemma:reduce}
  Calling  $\Call{algo:PartialReduction}{f, i}$ takes
  $$\softO(n_0(m_{i+\ell_i}-m_i) + m_s(n_{i-h_i}-n_i) + \delta)$$
  operations $(+,\times)$ in $\K$, with $h_i= 2^{{\rm val}_2(i)}$ and
  $\ell_i = \min(h_i,s-i)$. The output $r,Q$ satisfies the following
  properties:
  \begin{enumerate}
  \item $\deg(r,x) < m_s$ and $\deg(r,y) < n_0$
  \item $r \brem {{\mathcal G}} = f \brem {\mathcal G}$
  \item $r \brem {M_i} = f \brem {M_i}$ \label{FH3}
  \item $r \bdiv y^{n_{i-h_i}} = f \bdiv y^{n_{i-h_i}}$ \label{FH4}
  \item $((r \bdiv M_i) \bdiv y^{n_i}) \brem (D_{i+1}\cdots D_{i+\ell_i}, y^{n_{i-h_i}-n_i}) = 0$ \label{FH5}
  \item $\deg(Q,x) < m_{i+\ell_i}-m_i$ and $\deg(Q,y) < n_{i-h_i}-n_i$
  \end{enumerate}
\end{lemma}
\begin{proof}
  We first verify that all steps are well-defined, and discuss degree
  properties of the polynomials in the algorithm.
  
  As per our discussion in the preamble, the division and
  remainder at \cref{getF1,getF2} output a bivariate polynomial $F_2$
  on the mixed radix basis associated to $D_{i+1},\dots,D_{i+\ell_i}$.
  The polynomial $F_3$ is written on the same basis; $F_4$ represents
  the same polynomial, this time on the monomial basis. 
  
  That polynomial has $y$-degree less than $n_{i-h_i}$; since $G_i$
  has $y$-degree $n_i$, $q$, and thus $Q$, have $y$-degree less than
  $n_{i-h_i}-n_i$. Since $Q$ also has $x$-degree less than
  $m_{i+\ell_i}-m_i$, it is supported on the rectangle
  $\{0,\dots,m_{i+\ell_i}-m_i-1\}\times \{0,\dots,n_{i-h_i}-n_i-1\}$
  (which is the translate of $\iniR_i$ to the origin). This proves the
  last claim in the lemma.

  On the other hand, $G_i$ is supported on $\iniT$ (this is true
  because $i \ge 1$; for $i=0$, the initial term of $G_0$, which is
  $y^{n_0}$, is not in $\iniT$), so altogether, the call to
  \nameref{algo:Multiply} at \cref{getV} is justified. The variables
  $V_1$ and $V_2$ then represent the same polynomial, namely $Q G_i
  \brem (D_{i+1} \cdots D_s)$, on two different bases (resp. monomial
  and mixed radix).  It follows that $V_3$ represents the polynomial
  \begin{align*}
    M_i (Q G_i  \brem (D_{i+1} \cdots D_s)) &= M_i Q G_i \brem (M_iD_{i+1} \cdots D_s)\\
    &= Q g_i \brem M_s.
  \end{align*}
  As we noted in the previous subsection, since $V_2$ is written on
  the mixed basis associated to $(D_{i+1},\dots,D_s)$, $V_3$ is
  written on the mixed basis associated to $(D_{1},\dots,D_s)$. Since
  this is also the case for $f$, the subtraction at \cref{step:red_G}
  is done coefficient-wise, and results in the polynomial $(f-Q g_i)
  \brem {M_s}$, written on the same mixed basis.

  This being said, we establish properties 1-5. First item: We have
  $\deg(f,y) < n_0$.  On the other hand, the degree bound on $Q$
  implies that $Q g_i$ has $y$-degree less than $n_{i-h_i}$. Since
  $n_{i-h_i} \le n_0$, the product $Q g_i$ has $y$-degree less than
  $n_0$ as well, and it is then also the case for $r$. The bound
  $\deg(r,x) < m_s$ holds by construction.

  Second item: we can write $r=f-Qg_i + h M_s=f-Q g_i + h g_s$, for
  some $h$ in $\K[x,y]$, so that $r-f$ is in the ideal $\langle {\mathcal G}
  \rangle$.

  Third item: consider again the expression $r=f-Qg_i + h {M_s}$,
  and notice that ${M_i}$ divides both $g_i$ and ${M_s}$.

  Fourth item: because $\deg(f,x) < m_s$, the quotient $h$ in the
  relation $r=f-Qg_i + h {M_s}$ is $- Q g_i \bdiv {M_s}$.  Since $Q
  g_i$ has $y$-degree less than $n_{i - h_i}$, it is thus also the
  case for $h$. This shows that $r \bdiv y^{n_{i-h_i}} = f \bdiv
  y^{n_{i-h_i}}$, as claimed.

  
  Fifth item: since $r = f - Q g_i + h {M_s} = f -Q {M_i} G_i + h
  {M_s}$, we have $r \bdiv {M_i} = F_1 - Q G_i + h
  D_{i+1} \cdots D_s$.  By definition, we have $F_1 =
  F_2 + L D_{i+1}\cdots D_{i+\ell_i}$ and $F_2 = F_3 + K
  y^{n_{i-h_i}}$ for some $K,L$ in $\K[x,y]$. $F_4$ is the same
  polynomial as $F_3$, written on a different basis, and satisfies
  $F_4 = Q G_i + P + L' D_{i+1}\cdots D_{i+\ell_i}$, for some $P$ and
  $L'$ in $\K[x,y]$, with $P$ of $y$-degree less than $n_i$.
  Altogether, we obtain
  $r \bdiv {M_i} = P + (L+L') D_{i+1}\cdots D_{i+\ell_i} + h  D_{i+1} \cdots D_s + K y^{n_{i-h_i}}$. 
  As a result,
  \[(r \bdiv {M_i}) \bdiv y^{n_i} = ((L+L') \bdiv y^{n_i}) D_{i+1}\cdots D_{i+\ell_i} +
  (h \bdiv y^{n_i}) D_{i+1} \cdots D_s + K y^{n_{i-h_i}-n_i}.\]
  Because $i + \ell_i \le s$, this expression taken modulo
  $(D_{i+1}\cdots D_{i+\ell_i}, y^{n_{i-h_i}-n_i})$ vanishes, as
  claimed.

  It remains to estimate the cost of the algorithm. The divisions with
  remainders at \cref{getF1,getF2} are free of cost (because we work
  in the suitable mixed radix bases); the same holds for 
  \cref{getF3}, since it only involves a power of $y$.

  Since $D_{i+1}\cdots D_{i+\ell_i}$ has degree $m_{i+\ell_i}-m_i$,
  the conversion at \cref{getF4} uses
  $\softO(n_{i-h_i}(m_{i+\ell_i}-m_i))$ operations $(+,\times)$ in
  $\K$, which is $\softO(n_0 (m_{i+\ell_i}-m_i))$.

  Prior to the division at \cref{getq}, $G_i$ has to be reduced modulo
  $D_{i+1} \cdots D_{i+\ell_i}$; proceeding coefficient-wise in $y$,
  this takes $\softO(|\iniT|)=\softO(\delta)$ operations $(+,\times)$
  in $\K$. Then, the division in $\A[y]$ takes $\softO(n_{i-h_i})$
  operations $(+,\times)$ in $\A$, which is
  $\softO(n_{i-h_i}(m_{i+\ell_i}-m_i))$ operations $(+,\times)$ in
  $\K$. For this expression, it will be enough to use the same upper
  bound $\softO(n_0 (m_{i+\ell_i}-m_i))$ as above.

  Next, we consider the cost of computing the product $V$ in
  $\K[x,y]$.  The input $Q$ has $x$-degree less than
  $m_{i+\ell_i}-m_i$ and $y$-degree less than $n_{i-h_i}-n_i$, whereas
  $G_i$ is supported on the initial segment $\iniT$ of height $n_0$,
  width $m_s$, and cardinal $\delta$. Hence, using \cref{prop:mulST}
  (and the remarks that follow the proposition on the size of the
  support of $Q G_i$), we see that $Q G_i$ can be computed in $\softO(
  (m_{i+\ell_i}-m_i)(n_{i-h_i}-n_i) + n_0(m_{i+\ell_i}-m_i) + m_s
  (n_{i-h_i}-n_i) + \delta)$ operations $(+,\times)$ in $\K$. This is also
  $\softO(n_0(m_{i+\ell_i}-m_i) + m_s (n_{i-h_i}-n_i) + \delta)$.

  The Euclidean division at \cref{getV1} is done in the monomial
  basis, proceeding coefficient-wise in $y$. Computing $D_{i+1} \cdots
  D_s$ takes $\softO(m_s)$ operations $(+,\times)$ in $\K$. Then, the
  reduction is done in quasi-linear time in the size of the support of
  $V$, that is, $\softO(n_0(m_{i+\ell_i}-m_i) + m_s (n_{i-h_i}-n_i) +
  \delta)$ again. Recall that for polynomials supported on an initial
  segment $\iniU$, the conversion to the mixed radix basis takes
  quasi-linear time in the size of $\iniU$. Here, the support $\iniU$
  is contained in the support of $V=Q G_i$, so the conversion at
  \cref{getV2} takes time $\softO(n_0(m_{i+\ell_i}-m_i) + m_s
  (n_{i-h_i}-n_i) + \delta)$ once more.

  The multiplication by $M_i$ in the mixed radix basis is free, as we
  simply prepend a vector of zeros to each entry of $V_2$ to obtain
  $V_3$. Finally, the polynomial subtraction at the last step involves
  one subtraction in $\K$ for each nonzero coefficient of $V_3$, so
  $\softO(n_0(m_{i+\ell_i}-m_i) + m_s (n_{i-h_i}-n_i) + \delta)$
  altogether.  
\end{proof}

\paragraph{Correctness of the main algorithm} The properties stated 
above allow us to prove that \algoName{algo:Reduction} correctly
computes the remainder of $f$ by ${\mathcal G}$.

We define indices $(b_{i,j})_{0 \le i < s, 0 \le j < n_0}$ in
$\{1,\dots,s\}$ as follows.  For $i=0,\dots,s-1$, let $\iniT_i \subset
\N^2$ be the union of the initial segment $\iniT$ and the rectangles
$\iniR_1,\dots,\iniR_i$; in particular, $\iniT_0 = \iniT$ and
$\iniT_{s-1}$ is the rectangle $\{0,\dots,m_s-1\} \times
\{0,\dots,n_0-1\}$. Then, for $i=0,\dots,s-1$ and $j=0,\dots,n_0-1$,
we let $b_{i,j} \in \{1,\dots,s\}$ be the smallest index $k$ such that
$(m_k,j)$ is not in $\iniT_i$. In particular, $b_{s-1,j}=s$ for all $j
< n_0$. On the other hand, for $i=0$, we see that any pair $(u,j)$
with $u < m_{b_{0,j}}$ is in $\iniT$, so $x^u y^j$ is reduced modulo
${\mathcal G}$.

Let $f^{(0)},\dots,f^{(s-1)}$ be the polynomials computed throughout
the algorithm (the first item of \cref{lemma:reduce} proves that these
polynomials are well-defined, and all supported on the rectangle
$\{0,\dots,m_s-1\} \times \{0,\dots,n_0-1\}$). We prove the following
claim, written $A(i)$ in the sequel, by induction on $i=0,\dots,s-1$:
for $n_i \le j < n_0$, the polynomial
$\pcoeff(f^{(i)},y^j)\brem M_{b_{i,j}} \in \K[x]$ has degree
less than $m_{b_{0,j}}$. For $i=0$, there is nothing to prove (since
no index $j$ needs to be considered). Suppose that $A(i-1)$ holds, for
some $i$ in $\{1,\dots,s-1\}$; we prove $A(i)$.

For $j \ge n_{i-h_i}$, \cref{FH4} of \cref{lemma:reduce} shows that
$\pcoeff(f^{(i)},y^j)=\pcoeff(f^{(i-1)},y^j)$. Since in that case we
also have $b_{i,j}=b_{i-1,j}$, our claim holds. Now, suppose that $j$
is in $\{n_i,\dots,n_{i-h_i}-1\}$; in this case, \cref{FH3,FH5} of the
same lemma imply that $\pcoeff(f^{(i)},y^j)\brem M_{i+\ell_i}$
is equal to $\pcoeff(f^{(i-1)},y^j)\brem M_i.$ On the other hand, we
also have $b_{i-1,j}=i$ and $b_{i,j}=i+\ell_i$, so the left-hand side
is the term $\pcoeff(f^{(i)},y^j)\brem M_{b_{i,j}}$ that appears in
our claim.  Thus, to conclude the induction proof, it is enough to
show that $\pcoeff(f^{(i-1)},y^j)\brem M_i$ has degree less than
$m_{b_{0,j}}$.  We do this using a further case discussion:
\begin{itemize}
\item if $j \ge n_{i-1}$, we can use the induction assumption. It
  implies that the remainder $\pcoeff(f^{(i-1)},y^j)\brem
  M_{b_{i-1,j}}$ has degree less than $m_{b_{0,j}}$. Since we saw that
  have $b_{i-1,j} = i$, we are done.
\item if $j < n_{i-1}$, we have $b_{0,j} = i$, so that 
  $m_{b_{0,j}}=m_i=\deg(M_i)$, and our claim holds as well.
\end{itemize}
Having established our induction claim, we can take $i=s-1$. Then,
$A(s-1)$ shows that for $j$ in $n_{s-1},\dots,n_0-1$,
$\pcoeff(f^{(s-1)},y^j)\brem M_{s}$ has degree less than
$m_{b_{0,j}}$.  By construction, $f^{(s-1)}$ is reduced modulo $M_s$,
so that $\pcoeff(f^{(s-1)},y^j)$ itself has degree less than
$m_{b_{0,j}}$.  Now, for $j$ in $0,\dots,n_{s-1}-1$, we have
$b_{0,j}=s$, so $\pcoeff(f^{(s-1)},y^j)$ has degree less than
$m_{b_{0,j}}$ as well in this case. Altogether, as we pointed out when
we introduced $m_{b_{0,j}}$, this proves that $f^{(s-1)}$ is reduced
modulo~${\mathcal G}$.

The second item of \cref{lemma:reduce} finally shows that $f \brem {\mathcal G} =
f^{(s-1)} \brem {\mathcal G}$, so $f^{(s-1)}$ is indeed the normal form of
$f$ modulo ${\mathcal G}$. This finishes the correctness proof.

\paragraph{Cost analysis}
For the cost analysis, we start with the computation of polynomials
$M_i$, $G_i$ and $D_i$, at the beginning of the main algorithm. Since
division by a monic univariate polynomial take softly linear time,
each pass in the loop at \cref{step:mainFOR} of \nameref{algo:Reduction}
takes $\softO(\delta)$ operations, for a total of $\softO(s \delta)$.

The conversions to and from the mixed radix basis take quasi-linear
time in the size of the support of $f$, that is, $\softO(n_0 m_s)$
operations. Then, it suffices to add the costs of the calls to
\nameref{algo:PartialReduction}. By \cref{lemma:reduce}, deducing
$f^{(i)}$ from $f^{(i-1)}$ takes $\softO(n_0(m_{i+\ell_i}-m_i) +
m_s(n_{i-h_i}-n_i) + \delta)$ operations in $\K$, with $\delta =
|\iniT|$, so it suffices to sum this quantity for $i=1$ to $s-1$. The
first two terms add up to a total of $\softO(
n_0\sum_{i=1}^{s-1}(m_{i+\ell_i}-m_i) + m_s
\sum_{i=1}^{s-1}(n_{i-h_i}-n_i)).$ \cref{lemma:costpaving} shows that
this sum is in $\softO(n_0 m_s)$, so taking into account
the term $\softO(\delta)$ in each summand, the total is 
$\softO(n_0m_s + s\delta)$, as claimed.

\paragraph{Generalization to arbitrary inputs and discussion}
If the input $f$ does not satisfy the conditions $\deg(f,x) < m_s$ and
$\deg(f,y) < n_0$, we fall back to this case by reduction modulo the
pair of polynomials $(g_0,g_s)$, which have respective initial terms
$y^{n_0}$ and $x^{m_s}$. The following straightforward algorithm
achieves this; we discuss possible improvements below.

\begin{algorithm}[H]
  \algoCaptionLabel{ReductionGeneralInput}{f,{\mathcal G}}
  \begin{algorithmic}[1]
    \Require $f$ in $\K[x,y]$, ${\mathcal G}=(g_0,\dots,g_s)$ 
    \Ensure $f \brem {\mathcal G}$ 
    \State $f_1 \gets f \brem g_s$
    \State $f_2 \gets f \brem g_0$ in $\A[y]$
    \Comment{$\A=\K[x]/\langle g_s \rangle$}
    \State let $f_3$ be the canonical lift of $f_2$ to $\K[x,y]$ \Comment{$\deg(f_3,x) < m_s$}
    \State \Return $\Call{algo:Reduction}{f_3,{\mathcal G}}$
  \end{algorithmic}
\end{algorithm}

\begin{proposition}\label{prop:genred}
  Given $f$ and ${\mathcal G}$, with $\deg(f,x) < d$ and $\deg(f,y) < e$,
  \algoName{algo:ReductionGeneralInput} returns $f \brem {\mathcal G}$
  using $\softO(e d + e m_s + n_0 m_s + s \delta)$ operations
  $(+,\times)$ in~$\K$. If $\mathcal G$ generates an $\langle
  x,y\rangle$-primary ideal, the runtime becomes $\softO(\delta m_s)$
  operations $(+,\times)$ in~$\K$.
\end{proposition}
\begin{proof}
  Reducing $f$ modulo $g_s$ takes $\softO(e d)$ operations (and is
  actually free if $d <m_s$). Then, Euclidean division by $g_0$ in
  $\A[y]$ uses $\softO(e)$ steps in $\A$, which is $\softO(e m_s)$
  steps in $\K$. Finally, \cref{prop:red} gives a cost of $\softO(n_0
  m_s + s \delta)$ for the last step.
  
  If $\mathcal G$ generates an $\langle x,y\rangle$-primary ideal, all
  terms of $y$-degree at least $\delta$ vanish through the reduction
  (so we can replace $e$ by $\delta$), as do all terms of $x$-degree
  at least $m_s$ (so we can replace $d$ by $m_s$).
\end{proof}
In the runtime for the general case, $ed$ is the size of the support
of input $f$, and $s\delta$ our bound on the size of $\mathcal{G}$, so
they are essentially unavoidable (unless of course one could avoid
using ${\mathcal G}$ itself but only its Gr\"obner parameters). The
runtime also features the extra terms $e m_s$ and $n_0 m_s$, but
getting rid of them and improving the runtime to $\softO(e d + s
\delta)$ unconditionally seems to be very challenging.

Indeed, consider the {\em modular composition} problem: given $F,G,H$
in $\K[x]$, with $F$ monic of degree $n$ and $G,H$ of degrees less
than $n$, this amounts to computing $G(H) \brem F$. A direct approach
takes quadratic time, and Brent-Kung's baby-steps / giant-steps
algorithm uses $O(n^{1.69})$ operations (and relies on fast matrix
arithmetic). Bringing this down to a quasi-linear runtime has been an
open question since 1978: it is so far known to be feasible only over
finite $\K$~\cite{KeUm11}, with the best algorithm for an arbitrary
$\K$ to date featuring a Las Vegas cost of
$O(n^{1.43})$~\cite{NeSaScVi21}.

It turns out that modular composition is a particular case of the
reduction problem we are considering here. With $F,G,H$ as above, if
we consider ${\mathcal G}=(y-H(x), F(x))$ and the polynomial $f=G(y)$, then
the remainder $f \brem {\mathcal G}$ is precisely $G(H) \brem F$.  Here, we
have $n_0=1$, $s=1$, $m_s=n$, $\delta=n$, $d=1$ and $e=\deg(G,y)+1$,
so that in general $e=n$; on such input, the runtime of our algorithm
is $\softO(n^2)$. Improving our result to $\softO(ed +s\delta)$ would
give a softly linear modular composition algorithm, thus solving a
long-standing open question.

On the other hand, the case where $f$ has large degree in both $x$
and $y$, {\it i.e.} when $m_s \le d$ and $n_0 \le e$, is particularly
favourable, since then the runtime does become $\softO(ed + s \delta)$.
Another favourable situation is when all $e_i$'s are equal to $1$,
since we said before that we have $n_0 m_s \le s\delta$ in this case,
with thus a runtime of $\softO(e d + e m_s + s \delta)$.

Finally, we point out an application of \cref{prop:genred} to modular
multiplication: given $A,B$ in $\K[x,y]_{\iniT}$, where $\iniT$ is the
initial segment determined by ${\mathcal G}$, compute $f=AB \brem {\mathcal G} \in
\K[x,y]_\iniT$.  In this case, we have $d < 2 m_s$ and $e < 2 n_0$, so
the runtime is $\softO(n_0 m_s + s \delta)$; when all $e_i$'s are
equal to $1$, this becomes $\softO(s \delta)$. We are not aware of
previous results for this question.


\section{From Gr\"obner parameters to Gr\"obner basis} \label{sec:conversion}

In this section, we fix a given Gröbner cell (or equivalently, the
monomials $\bm E$). We show how make explicit the mapping $\Phi_{\bm
  E}:\K^N \to \mathcal{C}(\bm E)$, which takes as input Gr\"obner
parameters and outputs the corresponding reduced Gr\"obner basis (see
\cref{ssec:CV}).

First, we fix a way to index the $N$ coefficients of
the polynomials $(\sigma_{i,j})_{0\le i \le s-1, i \le j \le s}$ that
appear in the syzygy~\eqref{eq:syzygy}; this will be done in the
mutually inverse routines given below. Here, for simplicity, we assume
that given the monomials $\bm E$, we can directly access the integers
$s$, $(d_i)_{1\le i \le s}$ and $(e_i)_{1 \le i \le s}$.

\begin{algorithm}[H]
 \algoCaptionLabel{SigmaFromParameters}{{\bm E}, (\lambda_1,\dots,\lambda_N)}
  \begin{algorithmic}[1]
    \Require monomials $\bm E$, $(\lambda_1,\dots,\lambda_N)$ in $\K^N$
    \Ensure polynomials $(\sigma_{i,j})_{0\le i \le s-1, i \le j s}$ in $\K[x,y]$
    \State $k \gets 1$
    \For{$i=0,\dots,s-1$}
    \State $\sigma_{i,i} \gets \sum_{0 \le \ell < d_{i+1}} \lambda_{k+\ell} x^\ell$
    \State $k \gets k + d_{i+1}$
    \For{$j=i+1,\dots,s$}
    \State $\sigma_{i,j} \gets 0$
    \For{$m=0,\dots,e_{j-1}$}
    \State $\sigma_{i,j} \gets \sigma_{i,j} + \sum_{0 \le \ell < d_{i+1}} \lambda_{k+\ell} x^\ell y^m$
    \State $k \gets k + d_{i+1}$
    \EndFor
    \EndFor
    \EndFor
    \State \Return $(\sigma_{i,j})_{0\le i \le s-1, i \le j \le s}$
  \end{algorithmic}
\end{algorithm}

\begin{algorithm}[H]
 \algoCaptionLabel{ParametersFromSigma}{{\bm E}, (\sigma_{i,j})_{i,j}}
  \begin{algorithmic}[1]
    \Require monomials $\bm E$, polynomials $(\sigma_{i,j})_{i,j}$ in $\K[x,y]$
    \Ensure $(\lambda_1,\dots,\lambda_N)$ in $\K^N$
    \State $k \gets 1$
    \For{$i=0,\dots,s-1$}
    \State \InlineFor{$\ell=0,\dots,d_{i+1}-1$}{$\lambda_{k+\ell} \gets \coeff(\sigma_{i,i},x^\ell)$}
    \State $k \gets k + d_{i+1}$
    \For{$j=i+1,\dots,s$}
    \State $\sigma_{i,j} \gets 0$
    \For{$m=0,\dots,e_{j-1}$}
    \State \InlineFor{$\ell=0,\dots,d_{i+1}-1$}{$\lambda_{k+\ell} \gets \coeff(\sigma_{i,j},x^\ell y^m)$}
    \State $k \gets k + d_{i+1}$
    \EndFor
    \EndFor
    \EndFor
    \State \Return $(\lambda_1,\dots,\lambda_N)$
  \end{algorithmic}
\end{algorithm}
To deal with the particular case of punctual Gr\"obner parameters, a
few obvious modifications are needed, such as setting
$\sigma_{0,0},\dots,\sigma_{s-1,s-1}$ to zero and ensuring that $x$ divides
$\sigma_{1,0},\dots,\sigma_{s,s-1}$ in \nameref{algo:SigmaFromParameters}. We call
\textsc{SigmaFromPunctualParameters} and \textsc{PunctualParametersFromSigma}
the resulting procedures.

\medskip

We can now give an algorithm called
\nameref{algo:ReducedBasisFromParameters}, which describes the mapping
$\Phi_{\bm E}:\K^N \to \mathcal{C}(\bm E)$. This procedure is rather
straightforward; the algorithm for the inverse operation, called
\nameref{algo:ParametersFromReducedBasis}, is slightly more involved,
and is described in the next section. We still use the notation of
\cref{ssec:CV}, writing in particular $M_i \in \K[x]$ for the
polynomial coefficient of $y^{n_i}$ in both $g_i$ and $h_i$, for all
$i$, and $m_i$ for its degree.

We compute the $h_i$'s, and then the $g_i$'s, in descending order.  To
obtain the former, we simply use \cref{eq:syzygy}. For any
$i=s-1,\dots,0$, assuming we know $h_i$ and $g_{i+1},\dots,g_s$, let
us show how to obtain $g_i$ by reducing $h_i$ (for $i=s$, we have
$g_s=h_s$), using procedure \nameref{algo:Reduction} from the previous
section.

Using Euclidean division with respect to $x$, the polynomial $h_i$ can
be written as $h_i = A_i M_{i+1} + B_i$, with $A_i$ and $B_i$ in
$\K[x,y]$ and $\deg(B_i,x)  < m_{i+1}$.

Recall now that all polynomials $g_{i+1},\dots,g_s$ are multiples of
$M_{i+1}$, and that the family
$\mathcal{G}_i=(g_{i+1}/M_{i+1},\dots,g_s/M_{i+1})$ is a
zero-dimensional Gr\"obner basis (as pointed out after
\cref{eq:membership}).  Set $\bar h_i = (A_i \brem
\mathcal{G}_i)M_{i+1} + B_i$; we claim that $\bar h_i = g_i$. First,
we determine its initial term: all monomials in $A_i \brem {\mathcal
  G}_i$, and thus in $(A_i \brem \mathcal{G}_i)M_{i+1}$, have $y$-degree
less than $n_{i+1}$, whereas $B_i$ contains the initial term $x^{m_i}
y^{n_i}$ of $h_i$. Thus the initial term of $\bar h_i$ is still
$x^{m_i} y^{n_i}$. Next, we verify that $\bar h_i$ is reduced modulo
$g_1,\dots,g_{i-1},g_{i+1},\dots,g_s$.
\begin{itemize}
\item None of $g_1,\dots,g_{i-1}$ can reduce any term in $\bar h_i$,
  since this polynomial has $y$-degree $n_i$.
\item Since $A_i \brem {\mathcal G}_i$ is reduced modulo ${\mathcal
  G}_i$, $(A_i \brem {\mathcal G}_i)M_{i+1}$ is reduced modulo
  $g_{i+1},\dots,g_s$.
\item Since $B_i$ has $x$-degree less than $m_{i+1}$, it is also
  reduced modulo $g_{i+1},\dots,g_s$.
\end{itemize}
The last observation is that the difference $\bar h_i-h_i$ is in the
ideal $\langle g_{i+1},\dots,g_s \rangle$. Altogether, this
establishes $\bar h_i=g_i$.

With this, we can give our algorithm to compute $g_0,\dots,g_s$.  For
the reduction of the bivariate polynomial $A_i$ modulo ${\mathcal
  G}_i$, we use our procedure \nameref{algo:Reduction}. Note that the
degree assumptions for that procedure are satisfied: the polynomial
$A_i$ has $x$-degree less than $m_s-m_{i+1}$ and $y$-degree less than
$n_{i+1}$, which are precisely the maximal $x$-degrees and $y$-degrees
of the elements in ${\mathcal G}_i$.

As before, we assume that given $\bm E$, we can directly access the
integers $s$, $(d_i)_{1\le i \le s}$ and $(e_i)_{1 \le i \le s}$
and use them freely in the pseudo-code.

\begin{algorithm}[H]
 \algoCaptionLabel{ReducedBasisFromParameters}{{\bm E}, (\lambda_1,\dots,\lambda_N)}
 \begin{algorithmic}[1]
   \Require monomials $\bm E$, $(\lambda_1,\dots,\lambda_N)$ in $\K^N$
   \Ensure the reduced Gr\"obner basis of $\Phi_{\bm E}(\lambda_1,\dots,\lambda_N)$
   \State $(\sigma_{i,j})_{i,j} \gets \Call{algo:SigmaFromParameters}{\bm E, (\lambda_1,\dots,\lambda_N)}$ \label{RBFP1}
   \State $M_0 \gets 1$
   \State \InlineFor {$i=1,\dots,s$} $M_i \gets (x^{d_i}-\sigma_{i-1,i-1}) M_{i-1}$
   \State $h_s \gets M_s$; $g_s \gets M_s$
   \For {$i=0,\dots,s-1$}
   \State $T_i \gets  \textsc{KroneckerMultiply}(y^{e_{i+1}}, h_{i+1})+ \cdots +\textsc{KroneckerMultiply}(\sigma_{s,i}, h_s)$
   \State $h_i \gets T_i \bdiv    ( x^{d_{i+1}}  - \sigma_{i,i})$
   \State ${\mathcal G}_i \gets (g_{i+1} \bdiv M_{i+1},\dots,g_s \bdiv M_{i+1})$
   \State $A_i,B_i \gets h_i \bdiv M_{i+1}, h_i \brem M_{i+1}$
   \State $\bar A_i \gets \Call{algo:Reduction}{A_i, {\mathcal G}_i}$
   \State $g_i \gets \bar A_i M_{i+1} + B_i$
   \EndFor
   \State \Return $(g_0,\dots,g_s)$
  \end{algorithmic}
\end{algorithm}

\begin{proposition}\label{prop:redfromP}
  Given monomials $\bm E$ and $(\lambda_1,\dots,\lambda_N)$ in $\K$,
  \Call{algo:ReducedBasisFromParameters}{{\bm E},
    (\lambda_1,\dots,\lambda_N)} returns the reduced Gr\"obner basis
  of $\Phi_{\bm E}(\lambda_1,\dots,\lambda_N)$ using $\softO(s^2n_0
  m_s)$ operations $(+,\times)$ in $\K$.
\end{proposition} 
\begin{proof}
  Correctness follows from the previous discussion. Regarding the
  runtime, the first step does no arithmetic operation, and computing
  each polynomial $M_i$ takes $\softO(\delta)$ operations, for a total
  of $\softO(s\delta)$.

  For a given index $i$, computing $T_i$ involves at most $s$
  polynomial multiplications, each of which uses $\softO(n_0 m_s)$
  operations $(+,\times)$ in $\K$; we can deduce $h_i$ in the same
  asymptotic time. The Euclidean divisions needed to compute
  ${{\mathcal G}}_i$ cost $\softO(s\delta)$ operations (since all
  polynomials in ${\mathcal G}$ are supported on an initial segment of
  size $\delta$), and the one for $A_i$ and $B_i$ costs
  $\softO(n_0m_s)$, for the same reason. \cref{prop:red} shows that we
  compute $\bar A_i$ in $\softO(n_0 m_s + s \delta)$ operations
  $(+,\times)$. Finally, the product and sum giving $g_i$ take
  $\softO(n_0 m_s)$ operations $(+,\times)$ as well.

  Altogether, the cost at step $i$ is $\softO(sn_0 m_s + s\delta)$,
  which is $\softO(s n_0 m_s)$, and the overall runtime estimate
  follows.
\end{proof}
It will be useful to note that the algorithm does not perform
divisions, so if the input parameters lie in a ring $\A \subset \K$,
the output polynomials ${\mathcal G}$ all have coefficients in $\A$.

The whole procedure can be adapted to deal with punctual Gr\"obner
cells in a straightforward manner, by using
\textsc{SigmaFromPunctualParameters} at \cref{RBFP1}. The resulting
function is called \textsc{ReducedBasisFromPunctualParameters}, and
features a similar runtime.


\section{Computing the Gr\"obner parameters}\label{sec:inverse conversion}

We can now give our algorithms to compute the Gr\"obner parameters of
a zero-dimensional ideal $I$.

We do this in two different contexts. The first situation is the
recovery of these parameters starting from the reduced Gr\"obner basis
of $I$ ({\it i.e.}, computing the map $\Phi_{\bm E}^{-1}$ defined in
the previous sections). This is relatively straightforward, using a
sequence of Euclidean divisions.

The second variant we present is the core ingredient of our main
algorithm. Here, we consider an ideal $J$ given by generators
$f_1,\dots,f_t$, a zero-dimensional ideal $I$ containing $J$, and we
describe a system of polynomials which admits the Gr\"obner parameters
of $I$ as a solution with multiplicity one. In that, we follow
previous work of Hauenstein, Mourrain, Szanto~\cite{HaMoSz17} that was
in the context of border bases representations.

These latter equations are in general too complex to be dealt with
directly. In the next section, we will use them to describe our main
algorithm, a version of Newton iteration to compute the Gr\"obner
parameters of $I$ as above.


\subsection{Starting from a reduced basis}\label{ssec:fromGB}

In this subsection, we assume that we are given the reduced Gr\"obner
basis ${\mathcal G}=(g_0,\dots,g_s)$ of a zero-dimensional ideal $I$, and we
show how to compute its Gr\"obner parameters. We also indicate how the
procedure simplifies slightly when $I$ is $\langle
x,y\rangle$-primary.

Our notation is as before: the initial terms of the polynomials
$(g_0,\dots,g_s)$ are written $\bm E= (y^{n_0}, x^{m_1} y^{n_1},
\dots, x^{m_{s-1}} y^{n_{s-1}} , x^{m_s})$, the degree of ${\mathcal G}$ is
$\delta$ and $N=\delta+m_s$ is the number of Gr\"obner parameters. In
what follows, we compute the polynomials $(\sigma_{i,j})_{i,j}$ appearing
in the syzygies~\eqref{eq:syzygy}, whose coefficients are the
Gr\"obner parameters of $I$. Recall that we write $D_i = x^{d_i} -
\sigma_{i-1,i-1}$ for $i=1,\dots,s$, and $M_i=(x^{d_1}-\sigma_{0,0})\cdots
(x^{d_i}-\sigma_{i-1,i-1})$ for $i=0,\dots,s$, with the empty product being
equal to $1$.

\paragraph{Deriving the algorithm}
Knowing the reduced Gr\"obner basis ${\mathcal G}=(g_0,\dots,g_s)$, some of
the polynomials $(\sigma_{i,j})$ are easy to compute: for $i=1,\dots,s$, we
saw in the previous section that the polynomial coefficient of
$y^{n_i}$ in $g_i$ is none other than $M_i$. Knowing $M_1,\dots,M_s$
gives us $D_1,\dots,D_s$, and thus $\sigma_{0,0},\dots,\sigma_{s-1,s-1}$, by
successive divisions.

Let now $h_0,\dots,h_s$ be the non-reduced Gr\"obner basis already
used previously, that satisfies \cref{eq:syzygy}, and recall that for
$i=0,\dots,s$, $M_i$ divides $h_i$. We define $H_{i} = h_i/M_i$,
and consider again~\cref{eq:syzygy2}, which is a rewriting of~\eqref{eq:syzygy}:
$$D_{i+1} h_i - y^{e_{i+1}} h_{i+1} = \sum_{j=i+1}^s \sigma_{j,i}h_{j},$$
in which both left- and right-hand sides can be divided by $M_{i+1}$.
Carrying out the division, we obtain
\begin{equation}\label{eq:hij}
H_{i} - y^{e_{i+1}} H_{i+1} = \sum_{j=i+1}^s \sigma_{j,i} D_{i+2}\cdots D_j H_{j}.   
\end{equation}
Fix $i$ in $\{0,\dots,s-1\}$, and assume that we have computed
$H_{i+1},\dots,H_s$; we show how to compute
$\sigma_{i+1,i},\dots,\sigma_{s,i}$, and then $H_i$.

By construction, the polynomials $(g_0,\dots,g_i,h_{i+1},\dots,h_s)$
also form a minimal Gr\"obner basis of $I$. The polynomial $h_i-g_i$
is in $I$, so it reduces to zero through division by these
polynomials. Since $g_i$ and $h_i$ both have $M_i$ as polynomial coefficient of
$y^{n_i}$, $h_i-g_i$ has degree less than $n_i$ in $y$.  This implies
that the only polynomials in the list that can reduce it are
$h_{i+1},\dots,h_s$. We reduce $h_i-g_i$ by $h_{i+1}$, then $h_{i+2}$,
etc, in this order; for $j=i,\dots,s$, write $R_{i,j}$ for the remainder
obtained after reduction by $h_{i+1},\dots,h_j$, so that
$R_{i,i}=h_i-g_i$.

\begin{lemma}
  For $j=i,\dots,s$, $R_{i,j}$ has $y$-degree less than $n_j$.  
\end{lemma}
\begin{proof}
  We pointed out that this is true for $j=i$, so we suppose that the
  claim holds for some index $j < s$ and prove it for index $j+1$. To
  obtain $R_{i,j+1}$, we reduce $R_{i,j}$ by $h_{j+1}$, which has initial
  term $x^{m_{j+1}} y^{n_{j+1}}$, so that we can write $R_{i,j+1} =
  A_{j+1} + B_{j+1}$, with $\deg(B_{j+1},y) < n_{j+1}$,
  $\deg(A_{j+1},x) < m_{j+1}$ and all terms in $A_{j+1}$ having
  $y$-degree at least $n_{j+1}$. To conclude, we prove that
  $A_{j+1}=0$.

  Since we use the lexicographic order $y\succ x$, reduction of a term by
  $h_{j+1}$ does not increase its $y$-degree; since $R_{i,j}$ had
  $y$-degree less than $n_j$ by assumption, it is also the case for
  $A_{j+1}$. In particular, $A_{j+1}$ is reduced modulo $\mathcal{H}$.
  Since $R_{i,j}$ reduces to zero modulo $\mathcal{H}$, it follows that $A_{j+1}
  + (B_{j+1} \brem \mathcal{H}) = 0$. Now, for the same reason as above,
  $(B_{j+1} \brem \mathcal{H})$ has $y$-degree less than $n_{j+1}$, so that
  the supports of $A_{j+1}$ and $(B_{j+1} \brem \mathcal{H})$ do not
  overlap. This implies that $A_{j+1} = (B_{j+1} \brem \mathcal{H}) = 0$, as
  claimed.
\end{proof}

This lemma shows that the reduction of $h_i-g_i$ induces an equality
of the form
$$h_i - g_i = \sum_{j=i+1}^s q_{j,i} h_j,$$ for some polynomials
$q_{j,i}$ in $\K[x,y]$ satisfying $\deg( q_{j,i}, y) < n_{j-1}-n_{j} = e_{j}$ for
all $j$. Equivalently, we may rewrite this as
$$h_i = g_i + \sum_{j=i+1}^s q_{j,i}  {M_j} H_{j},$$ whence,
after dividing by ${M_i}$,
\begin{equation}\label{eq:hgD}
H_{i} = G_i + \sum_{j=i+1}^s q_{j,i} D_{i+1}\cdots D_j  H_{j}.  
\end{equation}
Combining~\eqref{eq:hij} and~\eqref{eq:hgD}, we get
\begin{equation}\label{eq:red}
G_i-y^{e_{i+1}} H_{i+1} = \sum_{j=i+1}^s Q_{j,i} H_{j},
\quad\text{with}\quad Q_{j,i}= (\sigma_{j,i} -q_{j,i}
{D_{i+1}})  D_{i+2}\cdots D_j.
\end{equation}
Notice in particular that for all $j$,
we have $\deg(Q_{j,i},y) < e_j$ and thus $\deg(Q_{j,i} H_j, y) <
n_{j-1}$.  

In this paragraph, for $F$ in $\K[x,y]$, we write $\bar F$ for its
residue class in $\B[y]$, with $\B=\K[x]/\langle D_{i+1} \cdots D_s\rangle$.
Take $j$ in $i+1,\dots,s-1$ and suppose that we know $\bar
Q_{i+1,i},\dots,\bar Q_{j-1,i}$. Split the sum in~\eqref{eq:red} as
$A= Q_{j,i}H_j+R$ with
$$A= G_i-y^{e_{i+1}} H_{i+1} - \sum_{k=i+1}^{j-1} Q_{k,i}H_k
\quad\text{and}\quad R= \sum_{k=j+1}^s Q_{k,i} H_{k}.$$ Over $\B[y]$,
$\bar R$ has degree (in $y$) less than $n_j$; since $\bar H_j$ is
monic of degree $n_j$, the relation $\bar A= \bar Q_{j,i}\bar H_j+\bar
R$ describes the Euclidean division of $\bar A$, which is known, by
$\bar H_j$, which is known as well. If we let $Q_{i,j}^*$ be the
canonical lift of $\bar Q_{i,j}$ to $\K[x,y]$, we obtain
\begin{align*}
Q_{j,i}^*&=Q_{j,i} \brem D_{i+1} \cdots D_s  \\
&=(\sigma_{j,i} -q_{j,i} {D_{i+1}}) D_{i+2}\cdots D_j  \brem D_{i+1} \cdots D_s.
\end{align*}
It follows that $Q_{i,j}^*$ is divisible by $D_{i+2}\cdots D_j$, and
that $$Q_{i,j}^*\bdiv (D_{i+2} \cdots D_j) = (\sigma_{j,i} -q_{j,i} {D_{i+1}})
\brem D_{i+1} D_{j+1} \cdots D_s .$$  Since $\deg(\sigma_{j,i}, x)<d_{i+1}$,
reducing this modulo $D_{i+1}$ finally gives us $\sigma_{j,i}$. Noticing
also that the remainder $\bar R$ gives us the next value of $\bar A$,
we obtain \algoName{algo:ParametersFromReducedBasis}.

In the following proposition, in preparation for the discussion in the
next subsection, we point out in particular that the algorithm does
not perform any division.



\begin{proposition}\label{prop:overA}
  Given a minimal reduced Gr\"obner basis ${\mathcal G} =
  (g_{0},\dots,g_{s})$ in $\K[x,y]$,
  $\Call{algo:ParametersFromReducedBasis}{\mathcal G}$ returns the
  Gr\"obner coefficients of $\mathcal G$ using $\softO(s^2 n_0 m_s)$
  operations $(+,\times)$ in $\K$.
\end{proposition}

\begin{algorithm}[H]
 \algoCaptionLabel{ParametersFromReducedBasis}{{\mathcal G}}
  \begin{algorithmic}[1]
    \Require ${\mathcal G} = (g_0,\dots,g_s)$ in $\K[x,y]^s$
    \Assumptions ${\mathcal G}$ is a minimal reduced Gr\"obner basis, with initial terms $(y^{n_0},\dots,x^{m_s})$ listed in decreasing order
    \Ensure $(\lambda_1,\dots,\lambda_N)$ in $\K^N$
    \State \InlineFor{$i=0,\dots,s$} $x^{m_i}y^{n_i} \gets \textsc{InitialTerm}(g_i)$
    \State $M_0 \gets 1$, $G_0 \gets g_0$
    \For{$i=1,\dots,s$} \label{step:getMDg}
    \State $M_i \gets \textsc{PolynomialCoefficient}(g_i,y^{n_i})$\label{step:Mi} \Comment{$M_i$ monic in $\K[x]$}
    \State $G_i \gets g_i \bdiv M_i$
    \State $D_i \gets M_{i} \bdiv M_{i-1}$ \Comment{$D_i$ monic in $\K[x]$}
    \State $n_{i-1,i-1} \gets x^{d_i}-D_i$ \Comment{$d_i=m_i-m_{i-1}$}
    \EndFor
    \State $H_s \gets 1$
    \For{$i=s-1,\dots,0$}\label{forI}
    \State $H_i \gets y^{e_{i+1}}H_{i+1}$\label{initHi} \Comment{$e_{i+1} = n_{i}-n_{i+1}$}
    \State $\bar A \gets \bar G_{i}-y^{e_{i+1}}\bar H_{i+1}$\label{step:redhi}
    \Comment {computation done in $\B[y]$, with $\B=\K[x]/\langle D_{i+1} \cdots D_s\rangle$}
    \For{$j=i+1,\dots,s$}\label{forJ}
    \State $\bar Q_{j,i} \gets \bar A \bdiv \bar H_j$, $\bar A \gets \bar A \brem \bar H_j$\label{updatebarA}
    \Comment {Euclidean division done in $\B[y]$}\label{step:eucdivA}
    \State $Q^*_{j,i} \gets $ canonical lift of $\bar Q_{j,i}$ to $\K[x,y]$
    \State $\sigma_{j,i} \gets (Q^*_{j,i} \bdiv D_{i+2} \cdots D_j) \brem D_{i+1}$ \label{step:nji}
    \State $H_i \gets H_i + \textsc{KroneckerMultiply}(\sigma_{j,i}, D_{i+2} \cdots D_j  H_j)$\label{step:updatehi}
    \EndFor
    \EndFor
    \State \Return $\Call{algo:ParametersFromSigma}{(y^{n_0},\dots,x^{m_s}), (\sigma_{j,i})_{0\le i \le s-1, i \le j \le s}}$
  \end{algorithmic}
\end{algorithm}

As before, the modifications needed to deal with the punctual
Gr\"obner cell are elementary; it suffices to invoke
\textsc{PunctualParametersFromSigma} at the last step.  The resulting
procedure will be written \textsc{PunctualParametersFromReducedBasis}.
Before proving the proposition, we give an example of computation
of punctual Gr\"obner coefficients.

\begin{example}
Given ${\mathcal G}$ as in the introduction from \autoref{ex:runningex},
$$\begin{array}{l}
  y^4 + \frac{17}{14}xy - \frac{17}{7}x^2, \\[1mm]
   x y^3 - \frac{10}{9}x^3,\\[1mm]
  x^2 y - 2x^3,\\[1mm]
  x^4,
\end{array}$$
Algorithm \textsc{PunctualParametersFromReducedBasis} computes
\begin{align*}
\sigma_{0,0}=\sigma_{1,0} =0, \quad \sigma_{2,0} = \frac{17}{14},\quad \sigma_{3,0}=\frac{40}9,\\[1mm]
\quad \sigma_{1,1}=\sigma_{2,1} =0, \quad \sigma_{3,1}=-\frac{10}9, \\[1mm]
\sigma_{2,2}=0, \quad \sigma_{3,2}=-2x
\end{align*}
and thus
\begin{align}\label{eq:valueslambda}
\lambda_1 = 0, \quad \lambda_2 =\frac{17}{14}, \quad \lambda_3 =\frac{40}9,\quad \lambda_4 =-\frac{10}9, \quad \lambda_5 = -2. 
\end{align}
\end{example}


\begin{proof}
  We already established correctness of the algorithm. By inspection,
  we see that all steps involve only additions and multiplications in
  $\K$, using only integer constants, since all that is done are
  multiplications or Euclidean divisions by monic polynomials, either
  in $\K[x]$ or in $\B[y]$, with $\B$ of the form $\K[x]/\langle
  D_{i+1} \cdots D_s\rangle$ (this in turn reduces to additions and
  multiplications in $\K$).

  It remains to establish the runtime of the algorithm. Each pass in
  the loop at \cref{step:getMDg} uses $\softO(\delta)$ operations
  $(+,\times)$, for a total of $\softO(s\delta)$. To continue the
  analysis, we first note that for all $i$, the polynomial $H_i$
  computed by the algorithm has $x$-degree less than $ 
  d_{i+1}+\cdots+d_s$, which is less than $m_s$, and $y$-degree
  $n_i$. The same bounds holds for $\deg(Q^*_{j,i},x)$ (by
  construction); the $y$-degree of this polynomial is less than $e_j$,
  as mentioned during the derivation of the algorithm.

  Since $G_i$ satisfies the same degree bound $\deg(G_i,x) <
  d_{i+1}+\cdots+d_s$ as $H_{i}$, the reduction of
  $G_i-y^{e_{i+1}}H_{i+1}$ modulo $D_{i+1}\cdots D_s$ at is free.  At
  each pass through \cref{step:eucdivA}, the Euclidean division takes
  $\softO(n_{j-1}) \subset \softO(n_0)$ operations $(+,\times)$ in
  $\B$, which is $\softO(n_0 m_s)$ operations $(+,\times)$ in
  $\K$. The degree bounds given above show that the cost of computing
  $\sigma_{j,i}$ and updating $H_i$ admits the same upper bound
  $\softO(n_0 m_s)$. Since we enter the inner \textsc{For} loop at
  \cref{forJ} $O(s^2)$ times, this gives a total cost $\softO(s^2 n_0
  m_s)$.
\end{proof}

Let us now see how to formalize the observation that the coefficients
computed by \algoName{algo:ParametersFromReducedBasis} are polynomial
expressions of the coefficients of ${\mathcal G}$.

Assume that the terms $\bm E$ are fixed, let $\mu_1,\dots,\mu_\delta$
be the monomials not in $\langle \bm E \rangle$, ordered in an
arbitrary fashion, and let $\Gamma_{0,1},\dots,\Gamma_{s,\delta}$ be
$(s+1)\delta$ new variables over $\Z$. We set $\A_{\bm
  E}=\Z[\Gamma_{0,1},\dots,\Gamma_{s,\delta}]$.

Because the algorithm only performs additions and multiplications, and
uses constants from the image of $\Z$ in $\K$, we deduce that there
exist $P_{1,\bm E},\dots,P_{N,\bm E}$ in $\A_{\bm
  E}=\Z[\Gamma_{0,1},\dots,\Gamma_{s,\delta}]$ such that given 
  any reduced Gr\"obner basis ${\mathcal G} = (g_0,\dots,g_s)$ with
initial terms $\bm E$ and with coefficients in $\K$ (or any extension
of it, as we choose below), the Gr\"obner parameters of ${\mathcal G}$ are
obtained by evaluating $P_{1,\bm E},\dots,P_{N,\bm E}$ at the
coefficients of ${\mathcal G}$.

Correctness of the algorithm can then be restated as follows.  Let
$\Lambda_{1},\dots,\Lambda_{N}$ be another set of new variables over
$\K$, that stand for ``generic'' Gr\"obner parameters, and define
$\L=\K(\Lambda_{1},\dots,\Lambda_{N})$. Let further
$g_{0,\L},\dots,g_{s,\L}$ be the polynomials obtained as output of
$\Call{algo:ReducedBasisFromParameters}{\bm E,
  (\Lambda_{1},\dots,\Lambda_{N})}$. Since that algorithm as well
performs only additions and subtractions (\cref{prop:redfromP}), 
these polynomials actually have coefficients in
$\K[\Lambda_{1},\dots,\Lambda_{N}]\subset \L$.
For $i=0,\dots,s$ and $j=1,\dots,\delta$, let then $R_{i,j} \in
\K[\Lambda_{1},\dots,\Lambda_{N}]$ be the coefficient of the monomial
$\mu_j$ in $g_{i,\L}$. We deduce from our discussion that $P_{i,\bm
  E}(R_{0,1},\dots,R_{s,\delta}) = \Lambda_i$ holds for all $i$.
We will use this observation in the next subsection.


\subsection{Polynomial equations for the Gr\"obner parameters}

Let now $f_1,\dots,f_t$ be polynomials in $\K[x,y]$; in this
subsection, those are our inputs, and we denote by $J$ the ideal they
generate in $\K[x,y]$. Let further $I$ be an ideal in $\K[x,y]$ such
that the following properties hold:
\begin{itemize}
\item[${\sf A}_1.$] $I$ has dimension zero;
\item[${\sf A}_2.$] there exists an ideal $I' \subset \K[x,y]$ such
  that $I + I' = \langle 1 \rangle$ and $I I' = J$.
\end{itemize}
Equivalently, $I$ is the intersection (or product) of some
zero-dimensional primary components of $J$. This is for instance the
case if the origin $(0,0)$ is isolated in $V(J)$ and $I$ is the
$\langle x,y\rangle$-primary component of $J$, or if $I=J$ and $V(J)$
is finite.

Let ${\mathcal G} = (g_0,\dots,g_s) \subset \K[x,y]$ be the reduced
lexicographic Gr\"obner basis of $I$. We denote by $\bm E$ the
initial terms of the polynomials in ${\mathcal G}$, written as before as
$$\bm E = (y^{n_0}, x^{m_1} y^{n_1}, \dots, x^{m_{s-1}}
y^{n_{s-1}},x^{m_s}).$$ In what follows, we assume that $\bm E$ is
known, but not ${\mathcal G}$; we show how to recover the Gr\"obner
parameters of $I$ (and thus ${\mathcal G}$ itself).

We let $\delta$ be the degree of $I$, and $\mu_1,\dots,\mu_\delta$ be
the monomials not in $\langle \bm E \rangle$, ordered in an arbitrary
way. Let further $N=\delta+m_s$ be the number of parameters for the
Gr\"obner cell $\mathcal{C}(\bm E)$, and let
$(\lambda_1,\dots,\lambda_N)=\phi_{\bm E}^{-1}(I) \in \K^N$ be the Gr\"obner parameters
associated to $I$. In this subsection, we define a system of $t\delta$
equations $\mathscr{E}_{1,1},\dots, \mathscr{E}_{t,\delta}$ in
$\K[\Lambda_1,\dots,\Lambda_N]$, where $\Lambda_1,\dots,\Lambda_N$ are
new variables, and we prove that $(\lambda_1,\dots,\lambda_N)$ is a
solution of multiplicity $1$ to these equations.

As in the previous subsection, let $\L=\K(\Lambda_1,\dots,\Lambda_N)$
and let $g_{0,\L},\dots,g_{s,\L}$ be the parametric Gr\"obner basis of
$\mathcal{C}(\bm E)$ given by
$\Call{algo:ReducedBasisFromParameters}{\bm E,
  (\Lambda_{1},\dots,\Lambda_{N})}$. Recall that all polynomials
$g_{0,\L},\dots,g_{s,\L}$ have coefficients in
$\K[\Lambda_1,\dots,\Lambda_N]$; this implies in particular that for
$A$ in $\K[x,y]$, the remainder $A \brem \langle
g_{0,\L},\dots,g_{s,\L}\rangle$ is in
$\K[\Lambda_1,\dots,\Lambda_N][x,y]$. For $j=1,\dots,\delta$, we then
denote by $\mathcal{N}_i$ the following $\K$-linear map:
\begin{align*}
  \mathcal{N}_j: \K[x,y] & \to \K[\Lambda_1,\dots,\Lambda_N]\\
      A & \mapsto {\rm coeff}(A \brem \langle g_{0,\L},\dots,g_{s,\L}\rangle, \mu_j),
\end{align*}
with $\mu_1,\dots,\mu_\delta$ the monomials not in $\langle \bm E
\rangle$, as defined above.  For $i=1,\dots,t$, we then let
$$\mathscr{E}_{i,1},\dots,\mathscr{E}_{i,\delta} =
\mathcal{N}_1(f_i),\dots, \mathcal{N}_\delta(f_i),$$ thus defining
$t\delta$ polynomials $\mathscr{E}_{1,1},\dots,\mathscr{E}_{t,\delta}$
in $\K[\Lambda_1,\dots,\Lambda_N]$. The following key property for
these equations was inspired by~\cite[Theorem~4.8]{HaMoSz17}, which
was stated in the context of border bases.

\begin{proposition}\label{prop:mult1}
  $(\lambda_1,\dots,\lambda_N)$ is a solution of
  $\mathscr{E}_{1,1},\dots, \mathscr{E}_{t,\delta}$ of multiplicity
  $1$.
\end{proposition}
\begin{proof}
  Let $\mathcal{I}$ be the ideal generated by all polynomials
  $\mathcal{N}_i(g_j)$, for $i=1,\dots,\delta$ and $j=0,\dots,s$, and
  let $R_{0,1},\dots,R_{s,\delta} \in
  \K[\Lambda_{1},\dots,\Lambda_{N}]$ be the coefficients of
  $(g_{0,\L},\dots,g_{s,\L})$, as in the previous subsection. Then,
  for $i=1,\dots,\delta$ and $j=0,\dots,s$, the polynomial
  $\mathcal{N}_i(g_j)$ is equal to $R_{j,i}(\lambda_1,\dots,\lambda_N) -
  R_{j,i}$. In particular,
  $(\lambda_1,\dots,\lambda_N)$ is in the zero-set of $\mathcal{I}$.

  Recall further from the previous subsection the existence of
  polynomials $P_{1,\bm E},\dots,P_{N,\bm E}$, with $P_{k,\bm
    E}(R_{0,1},\dots,R_{s,\delta})=\Lambda_k$ for all $k$.  The fact
  that $R_{j,i}(\lambda_1,\dots,\lambda_N) - R_{j,i}$ is in
  $\mathcal{I}$ for all $i,j$ implies that
$$P_{k,\bm E}(R_{0,1}(\lambda_1,\dots,\lambda_N),\dots,R_{s,\delta}(\lambda_1,\dots,\lambda_N))-
  P_{k,\bm E}(R_{0,1},\dots,R_{s,\delta})$$
  is in $\mathcal{I}$ as well, for all $k=1,\dots,N$. The left-hand
  side is $\lambda_k$, and the right-hand side $\Lambda_k$, so that
  $\mathcal{I}$ contains all polynomials
  $\Lambda_1-\lambda_1,\dots,\Lambda_N-\lambda_N$. Taken together, the
  two paragraphs so far establish that $\mathcal{I} = \langle
  \Lambda_1-\lambda_1,\dots,\Lambda_N-\lambda_N \rangle$.

  Let now $\mathcal{J}$ be the ideal generated in
  $\K[\Lambda_1,\dots,\Lambda_N]$ by the polynomials
  $\mathscr{E}_{1,1},\dots, \mathscr{E}_{t,\delta}$. Remark first that for any
  $a,b \ge 0$ and $i=1,\dots,t$, 
  $$(x^a y^b f_i) \brem \langle g_{0,\L},\dots,g_{s,\L} \rangle =
  \sum_{j=1}^\delta \mathcal{N}_j(f_i) (x^a y^b \mu_j \brem \langle
  g_{0,\L},\dots,g_{s,\L} \rangle).$$ It follows that for any $A$ in
  $J=\langle f_1,\dots,f_t\rangle$, and for $j=1,\dots,\delta$,
  $\mathcal{N}_j(A)$ is in $\mathcal{J}$. For the same reason, for $A$
  in $I=\langle g_0,\dots,g_s\rangle$, and for $j=1,\dots,\delta$,
  $\mathcal{N}_j(A)$ is in $\mathcal{I}$. We will also need the fact
  that for $A$ in $I^2$, and for all $j$, $\mathcal{N}_j(A)$ is in
  $\mathcal{I}^2$; this is established similarly.

  Recall now our second assumption on $I'$: there exists an ideal $I'
  \subset \K[x,y]$ such that $I + I' = \langle 1 \rangle$ and $I I' =
  J$.  Since $J$ is contained in $I$, the statements in the previous
  paragraph imply that $\mathcal{J}$ is contained in $\mathcal{I}= \langle
  \Lambda_1-\lambda_1,\dots,\Lambda_N-\lambda_N \rangle$, so that 
  $(\lambda_1,\dots,\lambda_N)$ is in the zero-set of $\mathcal{J}$.
  This proves the first claim of the proposition.

  Let further $K,K'$ be in resp. $I$ and $I'$ such that $K+K'=1$.  For
  $i=0,\dots,s$, $g_i$ is in $I$, so that $g_i K' = g_i -g_i K$ is in
  $I I' = J$. By the remark above, for
  $j=1,\dots,\delta$, $A_{j,i}:=\mathcal{N}_j(g_i) - \mathcal{N}_j(g_i
  K)$ is then in~$\mathcal{J}$, whereas $ \mathcal{N}_j(g_i K)$ is in
  $\mathcal{I}^2$.

  Consider the Jacobian matrix $\bm J$ of all polynomials $A_{j,i}$ at
  $(\lambda_1,\dots,\lambda_N)$. Because all terms $ \mathcal{N}_j(g_i
  K)$ are in $\mathcal{I}^2=\langle
  \Lambda_1-\lambda_1,\dots,\Lambda_N-\lambda_N \rangle^2$, their
  Jacobian matrix vanishes at $(\lambda_1,\dots,\lambda_N)$, so that
  $\bm J$ is simply the Jacobian matrix of the polynomials
  $\mathcal{N}_j(g_i)$ at $(\lambda_1,\dots,\lambda_N)$. Because these
  polynomials generate the ideal $\mathcal{I}=\langle
  \Lambda_1-\lambda_1,\dots,\Lambda_N-\lambda_N \rangle$, this matrix
  has trivial kernel. Thus, $\mathcal{J}$ has multiplicity $1$ at
  $(\lambda_1,\dots,\lambda_N)$.
\end{proof}

In the particular case where $I=J$, we have a slightly stronger 
result.

\begin{corollary}
  Suppose that $I=\langle f_1,\dots,f_t\rangle$. Then, $\langle
  \mathscr{E}_{1,1},\dots,\mathscr{E}_{t,\delta} \rangle = \langle
  \Lambda_{1}-\lambda_{1},\dots,
  \Lambda_{{N}}-\lambda_{{N}}\rangle$ in
  $\K[\Lambda_{1},\dots,\Lambda_{{N}}]$.
\end{corollary}
\begin{proof}
  Using the notation in the proof of the proposition, we see that if
  $I=J$, then $\mathcal{I}=\mathcal{J}$, and we proved that
  $\mathcal{I} = \langle \Lambda_1-\lambda_1,\dots,\Lambda_N-\lambda_N
  \rangle$.
\end{proof}

In our other particular case, where $I$ is the $\langle
x,y\rangle$-primary component of $J$, we can obtain a similar stronger
statement. Recall that the punctual Gr\"obner cell $\mathcal{C}_0(\bm
E)$ has dimension $N'=\delta - n_0$, and that the parameters for
$\mathcal{C}_0(\bm E)$ are obtained by setting $N-N'$ parameters to
zero in the parameters $\Lambda_1,\dots,\Lambda_N$ of $\mathcal{C}(\bm
E)$.

Let $\tau_{1},\dots,\tau_{N-N'}$ be the indices of these parameters
set to zero, and let $\Lambda_{\sigma_1},\dots,\Lambda_{\sigma_{N'}}$
be the remaining $N'$ parameters. For $i=1,\dots,t$ and
$j=1,\dots,\delta$, let $\mathscr{F}_{i,j}$ be the polynomial in
$\K[\Lambda_{\sigma_1},\dots,\Lambda_{\sigma_{N'}}]$ obtained by
setting $\Lambda_{\tau_1},\dots,\Lambda_{\tau_{N-N'}}$ to zero in
$\mathscr{E}_{i,j}$. Then, we have the following.

\begin{corollary}\label{coro:punctuallift}
  Suppose that $I$ is $\langle x,y\rangle$-primary. Then, $\langle
  \mathscr{F}_{1,1},\dots,\mathscr{F}_{t,\delta} \rangle = \langle
  \Lambda_{\sigma_1}-\lambda_{\sigma_1},\dots,
  \Lambda_{\sigma_{N'}}-\lambda_{\sigma_{N'}}\rangle$ in 
$\K[\Lambda_{\sigma_1},\dots,\Lambda_{\sigma_{N'}}]$.
\end{corollary}
\begin{proof}
  We proved in \cref{prop:mult1} that
  $\lambda_{\sigma_1},\dots,\lambda_{\sigma_{N'}}$ is a solution of $
  \mathscr{F}_{1,1},\dots,\mathscr{F}_{t,\delta}$. Besides, since the
  Jacobian matrix of $\mathscr{E}_{1,1},\dots,\mathscr{E}_{t,\delta}$
  has trivial kernel at $(\lambda_1,\dots,\lambda_N)$ (with thus
  $\lambda_{\tau_1}=\dots=\lambda_{\tau_{N-N'}}=0$), it is also the
  case for that of $\mathscr{F}_{1,1},\dots,\mathscr{F}_{t,\delta}$ at
  $(\lambda_{\sigma_1},\dots,\lambda_{\sigma_{N'}})$.  The only
  missing property is thus that
  $(\lambda_{\sigma_1},\dots,\lambda_{\sigma_{N'}})$ is the only
  common solution to these equations. Let
  $(\lambda^\star_{\sigma_1},\dots,\lambda^\star_{\sigma_{N'}}) \in
  \overline \K{}^{N'}$ be such a solution, let ${\mathcal G}^\star$ be the
  corresponding reduced Gr\"obner basis, and let $I^\star$ be the
  ideal it generates (in particular, $V(I^\star)= \{(0,0)\}$). Since
  by assumption ${\mathcal G}^\star$ reduces $f_1,\dots,f_t$ to zero, we
  have $J \subset I^\star$.

  By assumption on $I$, there exists an ideal $I' \subset \K[x,y]$
  such that $I + I' = \langle 1 \rangle$ and $I I' = J$. Let $K,K'$ be
  in resp. $I$ and $I'$ such that $K+K'=1$; in particular, $K'$ does
  not vanish at $(0,0)$. Since $V(I^\star)= \{(0,0)\}$, it follows
  that $K'$ is a unit modulo $I^\star$.

  Recall that we write ${\mathcal G} = (g_0,\dots,g_s)$ for the reduced
  lexicographic Gr\"obner basis of $I$. Then, for $i=0,\dots,s$, the
  polynomial $g_i K'$ is in $I I'$, so in $J$, and thus in
  $I^\star$. Since $K'$ is a unit modulo $I^\star$, this means that
  $g_i$ is in $I^\star$. Altogether, this proves that $I$ is contained
  in $I^\star$. Since these ideals have the same initial ideals for
  the lexicographic order, they are then equal.  This in turn proves
  that
  $(\lambda_{\sigma_1},\dots,\lambda_{\sigma_{N'}})=(\lambda^\star_{\sigma_1},\dots,\lambda^\star_{\sigma_{N'}})$.
\end{proof}

\paragraph{Example.} In our running example, we consider
only the punctual Gr\"obner cell, and we take $f_1$ and $f_2$ as in
the introduction. To write the equations for the punctual Gr\"obner
parameters, we consider $g_{0,\L},\dots,g_{3,\L}$ and set to zero the
parameters written $\Lambda_{\tau_1},\dots,\Lambda_{\tau_{N-N'}}$
above; the resulting polynomials were given in \eqref{eq:GL}, written
in variables $\lambda_1,\dots,\lambda_5$ (recall that $N'=5$
here). After reducing $f_1$ and $f_2$ by these polynomials, and taking
coefficients (we discard those that are identically zero), we obtain
\begin{align}\label{eq:example}
    14\Lambda_1,\quad
    14\Lambda_2 - 17,\quad
    -14\Lambda_1\Lambda_5^2 + 14\Lambda_3 - 28\Lambda_4\Lambda_5,\quad
    14\Lambda_2\Lambda_5 + 34,\quad
    -18\Lambda_4 + 10\Lambda_5.
\end{align}
As claimed, these polynomials generate the  maximal ideal 
\begin{align*}
  \Lambda_1,\quad \Lambda_2-17/14,\quad \Lambda_3-40/9,\quad \Lambda_4+10/9,\quad \Lambda_5+2.
\end{align*}
Because the input $f_1,f_2$ and ${\mathcal G}$ have rather small degrees, the
equations in~\eqref{eq:example} can be solved by hand. This is of
course not the case in general (although on many examples, several of
the equations are indeed linear).


\section{Newton iteration}\label{sec:Newton}

We can finally describe our main algorithm, which computes Gr\"obner
parameters using Newton iteration. For this, we will suppose that $\K$
is the field of fractions of a domain $\A$, and we consider a maximal
ideal $\mm$ in $\A$, with residual field $\k=\A/\mm$.

Consider the following objects: polynomials $(f_1,\dots,f_t)$ in
$\A[x,y]$ and a reduced Gr\"obner basis ${\mathcal G}$ in $\K[x,y]$.  We
make the following assumptions:
\begin{itemize}
\item[${\sf A}'_1.$] the ideal generated by ${\mathcal G}$ in $\K[x,y]$ is
  the intersection of some of the primary components of $\langle
  f_1,\dots, f_t\rangle$,
\item[${\sf A}'_2.$] all polynomials in ${\mathcal G}$ are in $\A_\mm[x,y]$,
\item[${\sf A}'_3.$] the ideal generated by $\mathcal G_\mm = \mathcal G \brem
  \mm$ in $\k[x,y]$ is the intersection of some of the primary
  components of the ideal $\langle f_1 \brem \mm,\dots,f_t \brem
  \mm\rangle$.
\end{itemize}
The last two items express that $\mm$ is good for $\mathcal G$, in the
sense of \cref{def:good}. Important cases where the first and third
assumptions are satisfied are as in the previous subsection, viz. when
${\mathcal G}_{\mm}$ and ${\mathcal G}$ generate the ideals $\langle f_1 \brem
\mm,\dots, f_t \brem \mm\rangle$, resp.\ $\langle f_1,\dots,
f_t\rangle$ themselves, or when they describe the $\langle
x,y\rangle$-primary components of these ideals.

Given $\mm$, $(f_1,\dots,f_t)$ and ${\mathcal G}_{\mm}$, we show here how
to compute ${\mathcal G} \brem \mm^{K}$, for an arbitrary $K \ge 1$.

\algoName{algo:LiftOneStep} describes the core lifting procedure; it
takes as input the Gr\"obner parameters of $\mathcal G$, known modulo
$\mm^\kappa$, for some $\kappa\ge 0$, and returns these parameters
modulo $\mm^{2\kappa}$ (note that since $\mathcal G$ has coefficients in
$\A_\mm$ by ${\sf A}'_2$, its Gr\"obner parameters are in $\A_\mm$ as
well, so reducing them modulo powers of $\mm$ makes sense).

The algorithm simply applies Newton's iteration to the equations
$\mathscr{E}_{i,j}$ introduced in the previous subsection. Note
however that we never explicitly write down these equations, as they
may involve a large number of terms: instead, we only compute their
first order Taylor expansions, as this is enough to conduct the
iteration. This explains why below, we work modulo the ideal $\langle
\Lambda_1,\dots,\Lambda_N \rangle^2$.

Since we want to give a cost estimate that counts operations in
$\A_{2\kappa}$, we here assume that we already know the reductions of
the input equations $f_1,\dots,f_t$ modulo $\mm^{2\kappa}$; they are
written $f'_1,\dots,f'_t \in \A_{2\kappa}[x,y]$.  Some steps in the
algorithm require a few further comments, namely the calls to
\nameref{algo:ReducedBasisFromParameters} at \cref{step_gstar},
\nameref{algo:Reduction} at \cref{step_reduction} and
\textsc{LinearSolve} at \cref{step_linsolve}.
\begin{itemize}
\item At \cref{step_gstar}, we are working with Gr\"obner parameters
  written $(\ell_1,\dots,\ell_N)$, that are in
  $\B=\A_{2\kappa}[\Lambda_1,\dots,\Lambda_N]/ \langle
  \Lambda_1,\dots,\Lambda_N \rangle^2$ (in the algorithm, elements of
  $\B$ are written as $b_0 + \sum_{i=1}^n b_i \Lambda_i$, for some
  $b_i$'s in $\A_{2\kappa}$). Recall that
  \algoName{algo:ReducedBasisFromParameters} only does additions and
  multiplications, and uses constants from $\Z$, so we can run this
  algorithm with inputs in $\B$; we keep in mind, though, that its
  properties were only established for inputs in a field.
  
  The same remark applies at \cref{step_reduction}, for
  \algoName{algo:ReductionGeneralInput}.
  
\item The last subroutine solves a linear system over $\A_{2\kappa}$: the
  inputs are elements of $\B$, which we recall take the form $b_0 +
  \sum_{i=1}^n b_i \Lambda_i$, for some $b_i$'s in $\A_{2\kappa}$. Procedure
  \textsc{LinearSolve} then sees these elements are linear equations
  in the $\Lambda_i$'s. We will prove that a solution exists, and also
  that the corresponding matrix admits a maximal minor that does not
  vanish modulo $\mm$, so that the solution is actually unique.
\end{itemize}


\begin{algorithm}
 \algoCaptionLabel{LiftOneStep}{(f'_1,\dots,f'_t), \bm E, (\alpha_1,\dots,\alpha_N)}
  \begin{algorithmic}[1]
    \Require $(f'_1,\dots,f'_t)$ in $\A_{2\kappa}[x,y]$, monomials $\bm E$, $(\alpha_1,\dots,\alpha_N)$ in
    $\A_\kappa^N$
    \Ensure  $(\alpha''_1,\dots,\alpha''_N)$ in $\A_{2\kappa}^N$
    \State $(\alpha'_1,\dots,\alpha'_N) \gets \text{lift of~} (\alpha_1,\dots,\alpha_N) \text{~to~} \A_{2\kappa}^N$
    \State $\mu_1,\dots,\mu_\delta \gets \text{monomials not in~} \langle \bm E \rangle$
    \For{$i=1,\dots,N$}
    \State $\ell_i \gets \alpha'_i + \Lambda_i$ \Comment{all $\ell_i$ in $\B=\A_{2\kappa}[\Lambda_1,\dots,\Lambda_N]/ \langle \Lambda_1,\dots,\Lambda_N \rangle^2$}
    \EndFor
    \State ${\mathcal G}^* \gets \Call{algo:ReducedBasisFromParameters}{\bm E, (\ell_1,\dots,\ell_N)}$
    \Comment{computations done over $\B$} \label{step_gstar}
    \State $\mathscr{R} \gets [\,]$
    \For{$i=1,\dots,t$}
    \State $r_i \gets \Call{algo:ReductionGeneralInput}{f'_i, {\mathcal G}^*}$
    \Comment{computations done  over $\B$}\label{step_ri}
 \label{step_reduction}
    \State \InlineFor{$j=1,\dots,\delta$}{$r_{i,j} \gets \text{coeff}(r_i,\mu_j)$} \Comment{all $r_{i,j}$ in $\B$} 
    \label{step_rij}
    \State $\mathscr{R} \gets \mathscr{R} \text{~cat~} [r_{i,1},\dots,r_{i,\delta}]$ 
    \Comment{$\mathscr{R}$ is an array with entries in $\B$}
    \EndFor
    \State  $(\epsilon_1,\dots,\epsilon_N) \gets \textsc{LinearSolve}(\mathscr{R})$ \label{step_linsolve}
    \Comment{all $\epsilon_i$ in $\A_{2\kappa}$}
    \State \InlineFor{$i=1,\dots,N$} {$\alpha_i'' \gets \alpha_i' + \epsilon_i$}
    \Comment{all $\alpha''_i$ in $\A_{2\kappa}$}
    \State \Return $(\alpha''_1,\dots,\alpha''_N)$
  \end{algorithmic}
\end{algorithm}

\begin{proposition}
  Suppose that ${\sf A}'_1$, ${\sf A}'_2$, ${\sf A}'_3$ hold, and let
  $(\lambda_1,\dots,\lambda_N) \in \A_\mm^N$ be the Gr\"obner
  parameters of $\mathcal G$.  Given $(f_1,\dots,f_t) \brem \mm^{2\kappa}$
  and $(\lambda_1 \brem \mm^\kappa,\dots,\lambda_N \brem \mm^\kappa)$,
  \algoName{algo:LiftOneStep} correctly returns $(\lambda_1 \brem
  \mm^{2\kappa},\dots,\lambda_N \brem \mm^{2\kappa})$.
\end{proposition}
\begin{proof}
  Let $\lambda=(\lambda_1,\dots,\lambda_N) \in \A_\mm^N$ be the
  Gr\"obner parameters associated to ${\mathcal G}$.  By assumption, the
  vector $\alpha=(\alpha_1,\dots,\alpha_N)$ satisfies $\alpha =
  \lambda \brem \mm^\kappa$, and the same holds for $\alpha'$. We
  prove that the output $\alpha''= (\alpha''_1,\dots,\alpha''_N)$ is
  equal to $\lambda \brem \mm^{2\kappa}$.

  This is simply the classical proof of the validity of Newton's
  iteration. Let $\delta$ be the degree ${\mathcal G}$, and let
  $\mathscr{E}=(\mathscr{E}_{1,1},\dots,\mathscr{E}_{t,\delta})$ be
  the equations introduced in the previous subsection for the
  polynomials $f_1,\dots,f_t$ and ${\mathcal G}$, over the field
  $\K$. Since all $f_i$'s have coefficients in $\A$, and since the
  reduction process introduces no new denominator, the polynomials
  $\mathscr{E}$ are in $\A[\Lambda_1,\dots,\Lambda_N]$. Using
  \cref{prop:mult1}, assumption ${\sf A}'_1$ shows that $\lambda$ is a
  solution to these equations (and that their Jacobian matrix at
  $\lambda$ has trivial kernel, but we will not need this fact
  directly).

  Let further
  $\mathscr{E}_\mm=(\mathscr{E}_{\mm,1,1},\dots,\mathscr{E}_{\mm,t,\delta})$
  be these same equations, but this time for the polynomials $f_1
  \brem \mm,\dots,f_t \brem \mm$ and $\mathcal G_\mm$. These are
  polynomials in $\k[\Lambda_1,\dots,\Lambda_N]$, with
  $\mathscr{E}_{\mm} = \mathscr{E}\brem \mm$.  Using
  \cref{prop:mult1}, assumption ${\sf A}'_3$ shows that $\lambda \brem
  \mm$ is a solution to these equations (which we already could deduce
  from the previous paragraph) and that their Jacobian matrix at
  $\lambda\brem{\mm}$ has trivial kernel. We will use this below.
  
  We claim that for all $i,j$, the coefficient $r_{i,j}$ computed at
  \cref{step_rij} is equal to
  $\mathscr{E}_{i,j}(\ell_1,\dots,\ell_N)$, computed in
  $\B=\A_{2\kappa}[\Lambda_1,\dots,\Lambda_N]/ \langle
  \Lambda_1,\dots,\Lambda_N \rangle^2$. The only point we have to be
  careful with is that the output of
  \algoName{algo:ReducedBasisFromParameters} is specified as being a
  Gr\"obner basis only if the inputs are in a field. To deal with
  this, let $\ell_1',\dots,\ell'_N$ be arbitrary lifts of
  $\ell_1,\dots,\ell_N$ to the domain $\A[\Lambda_1,\dots,\Lambda_N]$,
  and let $\mathcal G'$ be the output of
  $\Call{algo:ReducedBasisFromParameters}{\bm E,
    (\ell'_1,\dots,\ell'_N)}$.  These polynomials form a Gr\"obner
  basis in $\K(\Lambda_1,\dots,\Lambda_N)[x,y]$, which happens to have
  all its coefficients in $\A[\Lambda_1,\dots,\Lambda_N]$, and $\mathcal G^*$ computed at \cref{step_gstar} is the reduction of $\mathcal G'$ modulo
  $\mm^{2\kappa} + \langle \Lambda_1,\dots,\Lambda_N \rangle^2$.

  Similarly, at \cref{step_reduction},
  \algoName{algo:ReductionGeneralInput} can take as input polynomials
  with coefficients in $\B$, but its output was only specified for
  polynomials with coefficients in a field.  This is handled as
  before, and gives us that for all $i$, $r_i$ is the reduction modulo
  $\mm^{2\kappa} + \langle \Lambda_1,\dots,\Lambda_N \rangle^2$ of the
  polynomial $f_i \brem \mathcal G'$. Now, the coefficients of $f_i \brem
  \mathcal G'$ are the polynomials $\mathscr{E}_{i,j}$ evaluated at
  $(\ell'_1,\dots,\ell'_N)$, so altogether, for all $i,j$,
  $r_{i,j}=\mathscr{E}_{i,j}(\ell_1,\dots,\ell_N)$, as an element of
  $\B$.  Taking all $i,j$ at once, we obtain the following equalities
  over $\B$:
  \begin{align*}
    \mathscr{R} &= \mathscr{E}(\alpha'_1+\Lambda_1,\dots,\alpha'_N+\Lambda_N)
    \\ &= \mathscr{E}(\alpha') +  \text{jac}(\mathscr{E}, \alpha') [\Lambda_1~\cdots~\Lambda_N]^T,
  \end{align*}
  where $ \text{jac}(\mathscr{E}, \alpha')$ is the Jacobian matrix of
  $\mathscr{E}$ evaluated at $\alpha'$. First, we show that the system
  of linear equations $\mathscr{R}$ has a unique solution
  $\epsilon=(\epsilon_1,\dots,\epsilon_N)$ in $\A_{2\kappa}^N$.
  Indeed, given two solution vectors $\epsilon$ and $\epsilon'$ in
  $\A_{2\kappa}^N$, we obtain the relation
  $$ \text{jac}(\mathscr{E}, \alpha')
  [\epsilon_1-\epsilon'_1~\cdots~\epsilon_N-\epsilon'_N]^T
  =[0~\cdots~0]^T$$ over $\A_{2\kappa}$.  We pointed out above that
  $\text{jac}(\mathscr{E} \brem \mm, \lambda \brem {\mm})$ has trivial
  kernel, so it admits a non-zero $N$-minor in
  $\k=\A/\mm=\A_{2\kappa}/\mm$.  Now, by assumption, $\alpha' \brem
  \mm = \lambda \brem {\mm}$, so that $\text{jac}(\mathscr{E},
  \alpha')$ itself admits an $N$-minor invertible modulo $\mm$, and
  thus in $\A_{2\kappa}$. This in turn implies that
  $\epsilon=\epsilon'$, as vectors over $\A'/\mm^{2\kappa}$. Our first
  claim is proved.

  Second, we show that $\epsilon=(\lambda-\alpha') \brem
  \mm^{2\kappa}$ is a solution to these linear equations.  Indeed,
  modulo $\mm^{2\kappa}$, we have the Taylor expansion
  $\mathscr{E}(\alpha'+\epsilon)=\mathscr{E}(\alpha') +
  \text{jac}(\mathscr{E}, \alpha') [\epsilon_1~\cdots~\epsilon_N]^T$:
  higher-order terms vanish, since all entries of $\epsilon$ are by
  assumption in $\mm^\kappa$. Now, $\alpha'+\epsilon = \lambda \brem
  \mm^{2\kappa}$, so $\mathscr{E}(\alpha'+\epsilon)=0 \brem
  \mm^{2\kappa}$, and our claim follows.
  
  The two previous paragraphs prove that at the end of the while loop,
  the value $\alpha''$ satisfies $\alpha'' = \alpha' +
  (\lambda-\alpha') \brem \mm^{2\kappa} = \lambda \brem
  \mm^{2\kappa}$, so the proof is complete.
\end{proof}

\begin{proposition}\label{prop:LiftOneStepCompl}
  Let $\bm E= (y^{n_0}, x^{m_1} y^{n_1}, \dots, x^{m_{s-1}}
  y^{n_{s-1}} , x^{m_s})$ be the initial terms of $\mathcal G$, and suppose
  that all $f_i$'s have degree at most $d$.
  
  Under assumptions ${\sf A}'_1$, ${\sf A}'_2$, ${\sf A}'_3$,
  \algoName{algo:LiftOneStep} uses $\softO (s^2 n_0 m_s + t\delta(d^2
  + d m_s + s\delta + \delta^{\omega-1}))$ operations in
  $\A_{2\kappa}$.
\end{proposition}
\begin{proof}
  By convention (see the introduction), lifting each $\alpha_i$ to
  $\alpha'_i$ takes one operation in $\A_{2\kappa}$, for a total of
  $O(N) = O(\delta)$ operations.

  By \cref{prop:redfromP}, computing $\mathcal G^*$ takes $\softO(s^2 n_0
  m_s)$ operations $(+,\times)$ in $\B$, with each such operation
  taking $O(\delta)$ operations in $\A_{2\kappa}$.

  At \cref{step_ri}, by \cref{prop:genred},
  \algoName{algo:ReductionGeneralInput} uses $\softO(d^2 + d m_s + n_0
  m_s + s\delta)$ operations $(+,\times)$ in $\B$. Here, we know that
  $n_0$ is at most $d$, so the runtime for all $f_i$'s becomes
  $\softO(t(d^2 + d m_s + s\delta))$ operations in $\B$, which is
  $\softO(t\delta(d^2 + d m_s + s\delta))$ operations in
  $\A_{2\kappa}$.

  Finally, we have to solve the linear system defined by
  $\mathscr{R}=0$ over $\A_{2\kappa}$. This is a system in $t\delta$
  equations and $N$ unknowns, and we know that it admits a unique
  solution in $\A_{2\kappa}^N$, since the corresponding matrix has
  trivial kernel modulo $\mm$. Even though $\A_{2\kappa}^N$ is not a
  field, we may still apply fast algorithms, such as the one
  in~\cite{IbMoHu82} (as extended in~\cite{jeannerod2006lsp}),
  replacing zero-tests by invertibility tests; this takes
  $\softO(t\delta^{\omega})$ operations in $\A_{2\kappa}$.
\end{proof}

As usual, if ${\mathcal G}$ (and thus $\mathcal G_\mm$) is $\langle
x,y\rangle$-primary, we may use a variant of this lifting procedure,
called \textsc{LiftOneStepPunctualParameters}, which uses
\textsc{ReducedBasisFromPunctualParameters} and
\textsc{PunctualParametersFromReducedBasis} as subroutines. It allows
us to work with $N'$ rather than $N$ unknown Gr\"obner parameters; the
proof now relies on \cref{coro:punctuallift}, and the runtime becomes
$\softO (s^2 n_0 m_s + t\delta^2(m_s + \delta^{\omega-2}))$ operations
in $\A_{2\kappa}$ (see \cref{prop:genred}).

At this stage, we are almost done with the proof of \cref{prop:quad}:
for $K=2^k$, the algorithm simply computes $\mathcal G\brem \mm^K$
through repeated calls to \algoName{algo:LiftOneStep}. However, this
procedure works with Gr\"obner parameters as input and output. Hence,
prior to entering \algoName{algo:LiftOneStep} for the first time, we
compute the Gr\"obner parameters of $\mathcal G \brem \mm$, and after
the last call to \algoName{algo:LiftOneStep}, we compute $\mathcal G
\brem \mm^K$ using \algoName{algo:ReducedBasisFromParameters}. This
extra work does not affect the asymptotic runtime, so that we do
$\softO (s^2 n_0 m_s + t\delta(d^2 + d m_s + s\delta))$ operations in
$\A/\mm^{2^i}$, for $i=1,\dots,k$.

The only operations not accounted for so far are the coefficient-wise
reductions of the polynomials $f_1,\dots,f_t$ modulo
$\mm^2,\dots,\mm^{2^k}$. These cannot be expressed in terms of
operations in the residue class rings $\A/\mm^{2^i}$; instead, as per
the convention in the introduction, we assume that each coefficient
reduction modulo $\mm^{2^i}$ takes time $T_{2^i}$, for a total of $t
d^2 T_{2^i}$ for each $i=1,\dots,k$. This concludes the proof of our
main theorem. When we work with the punctual Gr\"obner cell, we saw in
\cref{prop:genred} that only $\delta m_s$ coefficients of each input
polynomial are needed, whence $t \delta m_s T_{2^i}$ steps for
coefficient reduction, for all indices $i$.

\begin{remark}
  If one wishes to work only with Gr\"obner bases as input and output,
  it is straightforward to design algorithms called
  \textsc{LiftOneStepGroebnerBasis} (and
  \textsc{LiftOneStepPunctualGroebnerBasis}), that take
  $f'_1,\dots,f'_t$ and $\mathcal G \bmod \mm^\kappa$ as input and return
  $\mathcal G \bmod \mm^{2\kappa}$. It suffices to call
  \algoName{algo:ParametersFromReducedBasis} when entering the
  procedure, then \algoName{algo:LiftOneStep}, and finally
  \algoName{algo:ReducedBasisFromParameters} before exiting (or their
  punctual variants). This does not affect asymptotic runtimes, but is
  not useful in the context of our main theorem.
\end{remark}

\begin{remark}
  When $\mm$ is principal, we can slightly improve of the lifting
  procedure by using either divide-and-conquer techniques (folklore)
  or relaxed algorithms~\cite[Section~4]{BeLe12} to solve the linear
  system that gives $\epsilon_1,\dots,\epsilon_N$. The downside is
  that the runtime is not written in terms of operations in
  $\A_{2\kappa}$ anymore. Instead, we give runtimes for the common
  cases $\A=\Z$ and $\mm=\langle p\rangle$, and $\A=\k[t]$ and
  $\mm=\langle t - \tau \rangle$:
  \begin{itemize}
  \item In the former case, solving the system uses
    $\softO(t\delta^{\omega} \log(p))$ bit operations, for a one-time
    computation (matrix inversion) done modulo $p$, and
    $\softO(\delta^2 \kappa \log(p))$ for subsequently solving the
    system modulo $p^{2\kappa}$.
  \item In the latter case, the one time computation takes
    $\softO(t\delta^{\omega})$ operations in $\k$, after which linear
    system solving takes $\softO(\delta^2 \kappa)$ operations in $\k$.
  \end{itemize}
  To wit, each operation in $\A_{2\kappa}$, as reported in
  \cref{prop:LiftOneStepCompl}, takes $\softO( \kappa \log(p))$ bit
  operations in the former case, and $\softO( \kappa \log(p))$
  operations in the latter. The net effect is that in both case, the
  cost of solving the linear system can be neglected (up to the
  one-time computation we perform at the beginning).
\end{remark}

\begin{example}
    We show one step of the algorithm for our running example (\autoref{ex:runningex}),
focusing on the punctual Gr\"obner parameters. Our input is the
polynomials $f_1,f_2$ as in the introduction, together with the
Gr\"obner basis of the $\langle x,y\rangle$-primary component of
$\langle f_1 \brem p, f_2 \brem p\rangle$, with $p=11$; namely:
\begin{align*}
\left |
\begin{array}{l}
  y^4 + 2 xy + 7 x^2, \\[1mm]
   x y^3 +5 x^3,\\[1mm]
  x^2 y + 9x^3,\\[1mm]
  x^4.
\end{array}
\right .
\end{align*}
We deduce the punctual Gr\"obner parameters modulo 11, $\alpha=(0, 2,
2, 5, 9) \in \Z/11\Z^5$ (recall that $N'=5$ here). Following the
algorithm, we set
$(\ell_1,\dots,\ell_5)=(\Lambda_1,2+\Lambda_2,2+\Lambda_3,5+\Lambda_4,9+\Lambda_5)$
and we compute the corresponding punctual Gr\"obner basis, with
coefficients truncated modulo $11^2$ and $\langle
\Lambda_1,\dots,\Lambda_N \rangle^2$. We obtain the polynomials written 
${\mathcal G}^*$ in the pseudo-code:
\begin{align*}
\left |
\begin{array}{l}
  y^4 +\Lambda_1 xy^2 + (\Lambda_2 + 2)  xy +(40\Lambda_1 + \Lambda_3 + 103\Lambda_4 + 111\Lambda_5 + 33)x^3 +(9\Lambda_2 + 2\Lambda_5 + 18)x^2, \\[1mm]
   x y^3 +(\Lambda_4 + 5) x^3,\\[1mm]
  x^2 y + (\Lambda_5 + 9)x^3,\\[1mm]
  x^4.
\end{array}
\right .
\end{align*}
Reducing $f_1$ and $f_2$ modulo ${\mathcal G}^*$ (with calculations done
modulo $11^2$ and $\langle \Lambda_1,\dots,\Lambda_N \rangle^2$), and 
keeping coefficients, we obtain the linear equations $\mathscr{R}$ 
(we only show the non-zero ones)
$$ 14\Lambda_1=14\Lambda_2 + 11=
    76\Lambda_1 + 14\Lambda_3 + 111\Lambda_4 + 102\Lambda_5 + 99=
    5\Lambda_2 + 28\Lambda_5 + 44=
    103\Lambda_4 + 10\Lambda_5=0.$$
They admit the following unique solution modulo $11^2$:
$$\epsilon_1=0,\ \epsilon_2=77,\ \epsilon_3=110,\ \epsilon_4=88,\ \epsilon_5=110;$$
as expected, all $\epsilon_i$ vanish modulo $11$. From this, $\alpha$
is updated to take the value $\alpha + \epsilon=[0, 79, 112, 93, 119
]$ modulo $11^2$. One can verify that this coincides modulo $11^2$
with the values given in \ref{eq:valueslambda}.
\end{example} 


\section{Conclusion}

A natural question is whether our approach can be used for ideals in
more than two variables. As of now, several ingredients are missing:
the known structure results are not as complete as
Lazard's~\cite{marinari2013cerlienco}, and there is no known explicit
description of Gr\"obner cells. Algorithmically, the key operation
(reduction modulo an $n$-variate lexicographic Gr\"obner basis) seems
to be a challenging problem in itself.

As already mentioned in the introduction, using our results in order
to recover $\mathcal G$ itself, rather than $\mathcal G \brem \mm^K$,
including in particular the quantification of bad maximal ideals
$\mm$, is the subject of future work.  Beyond this, the main
algorithmic improvement we would like to achieve is reducing the
overall cost so that it matches that of~\cite{LMS13}, in cases where
both approaches are applicable.

\section{Acknowledgements}
The authors thank: NSERC (Discovery Grant and Alexander Graham Bell Canada Graduate Scholarship), FQRNT and the University of Waterloo for their generous support.

  \bibliographystyle{elsarticle-num} 
  \bibliography{biblio}
\end{document}